\documentclass[11pt,reqno,oneside]{amsart}

\usepackage{graphicx}
\usepackage{amssymb}
\usepackage{epstopdf}

\usepackage[a4paper, total={6in, 9.1in}]{geometry}

\usepackage{amsmath,amsfonts,amsthm,mathrsfs,amssymb,cite}
\usepackage[usenames]{color}

\usepackage{graphics}
\usepackage{framed}

\usepackage{enumerate}
\usepackage{hyperref}
\usepackage{xcolor}
\hypersetup{
  colorlinks,
  linkcolor={blue!80!black},
  urlcolor={blue!80!black},
  citecolor={blue!80!black}
}
\usepackage[capitalize,noabbrev]{cleveref}

\usepackage[utf8]{inputenc} 

\numberwithin{equation}{section}

\newtheorem{Theorem}{Theorem}[section]
\newtheorem{Lemma}[Theorem]{Lemma}

\newtheorem{Definition}[Theorem]{Definition}
\newtheorem{Corollary}[Theorem]{Corollary}

\newtheorem{Remark}[Theorem]{Remark}
\numberwithin{equation}{section}

 \def\p{\partial} 
\def \Vh0{\stackrel{\circ}{V}_h} \def\to{\rightarrow}

  \def\f{\frac}  
   \def\eps{\varepsilon}
\def\m{\mbox}

\def\p{\partial}

\newcommand{\lc}
{\mathrel{\raise2pt\hbox{${\mathop<\limits_{\raise1pt\hbox
{\mbox{$\sim$}}}}$}}}

\newcommand{\gc}
{\mathrel{\raise2pt\hbox{${\mathop>\limits_{\raise1pt\hbox{\mbox{$\sim$}}}}$}}}

\newcommand{\ec}
{\mathrel{\raise2pt\hbox{${\mathop=\limits_{\raise1pt\hbox{\mbox{$\sim$}}}}$}}}

\def\bb{\begin{equation}} \def\ee{\end{equation}}

\def\beqn{\begin{eqnarray}}  \def\eqn{\end{eqnarray}}

\def\beqnx{\begin{eqnarray*}} \def\eqnx{\end{eqnarray*}}

\def\bn{\begin{enumerate}} \def\en{\end{enumerate}}

\def\bd{\begin{description}} \def\ed{\end{description}}



\allowdisplaybreaks

\title[Sensitivity analysis for reconstructions in transmission problems]{Localized sensitivity analysis at high-curvature boundary points of reconstructing inclusions in transmission problems}

\author{Habib Ammari}
\address{Department of Mathematics, ETH Z\"urich, R\"amistrasse 101, CH-8092, Switzerland}
\email{habib.ammari@math.ethz.ch}

\author{Yat Tin Chow}
\address{Department of Mathematics, University of California, Riverside, USA}
\email{yattinc@ucr.edu}

\author{Hongyu Liu}
\address{Department of Mathematics, City University of Hong Kong, Kowloon, Hong Kong, China}
\email{hongyu.liuip@gmail.com, hongyliu@cityu.edu.hk}

\begin{document}
\maketitle

\begin{abstract}

In this paper, we are concerned with the recovery of the geometric shapes of inhomogeneous inclusions from the associated far field data in electrostatics and acoustic scattering. We present a local resolution analysis and show that the local shape around a boundary point with a high magnitude of mean curvature can be reconstructed more easily and stably. In proving this, we develop a novel mathematical scheme by analyzing the generalized polarisation tensors (GPTs) and the scattering coefficients (SCs) coming from the associated scattered fields, which in turn boils down to the analysis of the layer potential operators that sit inside the GPTs and SCs via microlocal analysis. In a delicate and subtle manner, we decompose the reconstruction process into several steps, where all but one steps depend on the global geometry, and one particular step depends on the mean curvature at the given boundary point. Then by a sensitivity analysis with respect to local perturbations of the curvature of the boundary surface, we establish the local resolution effects. Our study opens up a new field of mathematical analysis on wave super-resolution imaging. 

\medskip



\end{abstract}

\section{Introduction}

We are concerned with the recovery of inhomogeneous inclusions from measurement of the scattered fields in electrostatics and acoustic scattering. We are particularly interested in studying and analysing how the mean curvature of the shape of an inhomogeneity would affect propagation of information from the shape via the scattered field. Let $D$ signify the shape of an inhomogeneity. We show that information from points $x \in \partial D$ with high magnitude of mean curvature $|H(x)|$ propagate with a significantly larger magnitude.  This is reflected by the sensitivity analysis of the scattered field with respect to the change of the shape. Indeed, we can localize our analysis at those boundary points with high mean curvature. Such larger sensitivity of information allows one more easily to locate these points of high mean curvature and also in a stable manner to reconstruct the local shape around these points.

The study of the correspondence between the geometry of an inhomogeneous inclusion and its scattered field has attracted significant attention in the literature. It can also find important applications in practice including medical imaging and geophysical exploration \cite{book,CK}. From a physical intuition, ``pathological" geometries {\color{black} (i.e. shapes with sharp corners and singularities)} should help to improve the transmission of scattering information and hence enhance the imaging effect. {\color{black} As a vivid example, one may think that a sharp corner on an object looks more ``visible" than a round one}. In fact, one of the ``pathological" geometries that has been studied extensively in the literature is the corner/edge singularity, where the surface tangential vectors are discontinuous. In \cite{BPS}, it is shown that the corner singularity of an inhomogeneity always scatters a probing field nontrivially and in \cite{curv_Liu_2}, the authors further quantified the result by establishing a positive lower bound on the scattering energy. That means, a corner singularity of an inhomogeneous inclusion can always generate significant scattering, and this is consistent with the aforementioned physical intuition. From a geometric perspective, a corner singularity indicates that the ``extrinsic" curvature is infinite. Hence, it is natural to consider the scattering due to curvatures. In \cite{curv_Liu}, the scattering by curvatures is considered and it is shown that if there is a boundary point on a generic inclusion with a sufficiently high curvature, then it scatters every probing field nontrivially. {\color{black} It is pointed out that in \cite{curv_Liu}, the high-curvature boundary point possesses both high mean and Gaussian curvatures, which can be regarded as a mollified corner point.} We would also like to mention in passing the related study in the literature on the scattering from scatterers with ``pathological" geometries in various scenarios; see \cite{BL17,CDL,LPRX,LRX,Maj,CHL,DCL,LT,LX,LZ} and the references therein. In particular, in \cite{Maj} the recovery of the boundary curvature of a convex acoustic obstacle from the associated high-frequency scattered fields was established and in \cite{CHL} characterisations of the generalized polarisation tensors (GPTs) \cite{ammari2004reconstruction,book} of the scattered field due to corner singularities of an insulating cavity in electrostatics were derived. 

In this paper, we rigorously reinforce the aforementioned physical intuition from a reconstruction perspective by performing the sensitivity analysis of the reconstruction around the boundary point with high mean curvature. That means, we include the corner/edge singularity as an extreme case. We consider our study for the electrostatics and the wave scattering in the quasi-static regime, where the reconstructions are severely more ill-conditioned than the corresponding high-frequency reconstruction. Indeed, we know that the corresponding reconstructions are exponentially ill-posed {\color{black} in the sense that in the worst case the measurement error could be exponentially amplified in the reconstruction of the inclusion in terms of the Hausdorff distance (cf. \cite{curv_Liu_2,CR,Isakov,LPRX,LRX,LT})}.
{\color{black} The exponential-illposedness comes from the fact that the inverse problem is performed with data of only one single frequency/wavenumber, and is intimately related to the ill-posedness of a unique continuiation problem; (notice that the inverse problem with multiple fequency is not expoentially-illposed.)}
 {\color{black} On the other hand, we would like to mention some related studies on inverse scattering problems which established increasing stability estimates according to the increase of the frequency \cite{ILW,NUW}. }One of the major findings in our study can be roughly described as follows by taking the reconstruction in electrostatics for the discussion. The generalized polarised tensors (GPTs) of the scattered field are a natural and powerful shape descriptor of the underlying inclusion \cite{yu,book}. It is a fact that the high-frequency information of the shape of the inclusion, namely the fine details of shape, enters the higher order GPTs. Thus the boundary information around the high-curvature point enters the high-order GPTs. However, GPTs decay exponentially {\color{black} with respect to the Fourier mode/order $(L,M)$ as $|L|,|M|$ goes to infinity}; and hence, as the scattered field propagates away from the inclusion, the fine-detail information of the inclusion becomes less visible and will be contaminated by the noise. Nonetheless, if we have very large magnitude of high curvature information, these higher order GPTs, although they exponentially decay, will be pushed up to relatively high magnitudes, making them more apparent. Therefore after a further perturbation around such a point of high curvature, the fine details near it will be more apparent in the far field and stably reconstructable. Hence, it is viable to claim that one can produce super-resolution reconstruction of the inclusion around the high curvature point.
{\color{black}
This also aligns with the physical intuition and daily experience that scattered fields from sharp corner shine more brightly and propagate with a stronger magnitude to afar.} On the other hand, it is emphasized that in this work we are not suggesting a new reconstruction method.  In fact stable ways of reconstructing inclusions using the GPTs via the optimisation approach can be found in \cite{gpt,mono_1,mono_2,spectral}.
{\color{black} In the current article, Newton-type methods are considered for an illustrative purpose of our theoretical study. They are used to illustrate the local sensitivity of the scattered field at the high-curvature point.
Indeed, we illustrate our point with a brief framework of a preconditioned Gauss-Newton descent method, with a choice of a preconditioning suggested by our theoretical understanding of the local sensitivity with respect to the curvature term $|H(x)|^2$.}

{Although it is physically intuitive to expect that the local geometry of the shape of an inclusion should have an effect on the local resolution, it turns out that the corresponding derivation is highly technical. In fact, decoding the local geometric information of an inclusion from the corresponding scattered field is highly challenging. In this paper, we develop a novel mathematical scheme to understand this correspondence through analyzing the GPTs and scattering coefficients (SCs) coming from the associated scattered fields, which in turn boils down to the analysis of the operators that sit inside the GPTs and SCs via microlocal analysis. By doing so, we are able to decompose the reconstruction process into several steps in a delicate and subtle manner, where all but one steps depend only on the global geometry, and one particular step depends on the mean curvature of the surface at that point (c.f. Corollary \ref{cor:1}). Finally, we are able to see clearly how sensitivity with respect to local perturbations relates to local curvature information of the surface. {\color{black}
We would like to remark that most of the conclusions in our discussion hold for domains which are regular enough, in particular for $\mathcal{C}^{4,\alpha}$ domains. Our results provide a rigorous justification as to why among compact domains with sufficiently regular boundaries (i.e. $\mathcal{C}^{4,\alpha}$ domains) and with a certain lower curvature bound, those points at a domain with the highest
curvatures are more easily reconstructed. Our study has important applications in super-resolution wave imaging.
A similar analysis towards domains with singular geometry, e.g. Lipschitz or $C^{0,\alpha}$ domains, or non-compact domains, remains an open problem, and will be the subject of a forthcoming study. 
}

The rest of the paper is organised as follows. In Section 2, we consider the electrostatic transmission problem. We compute the semi-classical symbols of several related operators on the boundary of the inclusion including the Neumann-Poincar\'e operator and its variants, and perform a sensitivity analysis on the generalized polarization tensors and hence the scattering coefficients.  We establish an increased sensitivity at points with mean curvature of high magnitude.  Then we move on to the inverse wave scattering in the low-frequency regime governed by the Helmholtz system with a small wavenumber and observe a similar property in Section 3.

\section{Localized sensitivity analysis for reconstructions in electrostatics} \label{sec2}

In this section, we consider the reconstruction of an inhomogeneous inclusion in electrostatics. We first introduce the electrostatic transmission problem as well as the associated layer potential operators that are crucial in our subsequent analysis. We compute the semi-classical symbols of those operators when viewing them as pseudo-differential operators. Then we conduct the localised sensitivity analysis at the high-curvature point on the boundary of the inclusion.  

\subsection{Electrostatic transmission problem and layer-potential operators} 

We introduce the electrostatic transmission problem and the associated layer-potential operators. Consider an open connected domain $D$ with a $\mathcal{C}^{2,\alpha}$, $0<\alpha<1$, boundary $\partial D$ and a connected complement $\mathbb{R}^d\backslash\overline{D}$, $d\geq 2$. Physically, $D$ is the support of an inhomogeneous dielectric inclusion. Let $\gamma_c$ and $\gamma_m$ be two positive constants, signifying the electric permittivities. Consider a medium configuration as follows,  
\begin{equation}\label{eq:mc1}
\color{black}\gamma_{D}  = \gamma_c \chi(D) +  \gamma_m \chi(\mathbb{R}^d \backslash \overline{D}),
\end{equation}
where in what follows, $\chi$ stands for the characteristic function of a domain. 
Let $u_0$ be a given harmonic function that signifies a probing field for the inclusion $D$. The electrostatic transmission problem is given for a potential field $u\in H_{loc}^1(\mathbb{R}^d)$ as follows,
\beqn
    \begin{cases}
        \nabla \cdot ({\color{black}\gamma_{D}} \nabla u) = 0 &\text{ in }\; \mathbb{R}^d, \\[1.5mm]
         u - u_0 = {O}(|x|^{1-d}) &\text{ as }\; |x| \rightarrow \infty. 
    \end{cases}
    \label{transmission}
\eqn
{\color{black}The well-posedness of the electrostatic problem \eqref{transmission} can be conveniently found in \cite{McL}.}

We proceed to introduce the single-layer potential operator and the Neumann-Poincar\'e operator associated with \eqref{transmission}. They are crucial in solving \eqref{transmission} via the layer-potential theory, and moreover they provide critical ingredients in solving the inverse problem of reconstructing the inclusion $D$ from the associated scattered field $u-u_0$. 
{\color{black}
A general exposition of the potential theory may be found in \cite{kellog, folland, mcowen}, and of the Neumann-Poincar\'e operator in \cite{shapiro}.
}

Given a density function $\phi \in L^2(\partial D, d \sigma)$, the single-layer and double-layer potentials, $\mathcal{S}_{\partial D} [\phi]$ and $\mathcal{D}_{\partial D} [\phi]$, are respectively defined as follows,
\beqn
    \mathcal{S}_{\partial D} [\phi] (x) &:=& \int_{\partial D} G (x-y) \phi(y) d \sigma(y), \\
    \mathcal{D}_{\partial D} [\phi] (x) &:=& \int_{\partial D} \frac{\partial }{\partial \nu_y } G(x-y) \phi(y) d \sigma(y), 
\eqn
for $x \in \mathbb{R}^d$, where $G$ is the fundamental solution of the Laplacian in $\mathbb{R}^d$ :
\beqn
    G (x-y) =
    \begin{cases}
     \color{black}\f{1}{2\pi} \log |x-y| & \text{ if }\; d = 2 \, ,\\
     \f{1}{(2-d)\varpi_d} |x-y|^{2-d} & \text{ if }\; d > 2 \, ,
    \end{cases}
    \label{fundamental}
\eqn
with $\varpi_d$ denoting the surface area of the unit sphere in $\mathbb{R}^d$. The single-layer potential satisfies the following jump relation across $\partial D$:
\beqn
    \f{\p}{\p \nu} \left(  \mathcal{S}_{\partial D} [\phi] \right)^{\pm} = (\pm \f{1}{2} I + \mathcal{K}^*_{ \partial D} )[\phi]\,,
    \label{jump_condition}
\eqn
where the superscripts $\pm$ indicate the limits from outside and inside $D$ respectively, and
$\mathcal{K}^*_{\partial D}: L^2(\partial D, d \sigma) \rightarrow L^2(\partial D, d \sigma)$ is the Neumann-Poincar\'e operator defined by
\beqn
    \mathcal{K}^*_{\partial D} [\phi] (x) := \f{1}{\varpi_d} \int_{\partial D} \f{\langle  x-y,\nu (x) \rangle  }{|x-y|^d} \phi(y) d \sigma(y) \,,
    \label{operatorK}
\eqn
with $\nu(x)$ being
the outward normal at $x \in \partial D$. 
$\mathcal{K}^*_{\partial D}$ maps $L^2_0(\partial D)$ onto itself, where
\beqn
L_0^2(\partial D, d \sigma):=\{\phi\in L^2(\partial D, d \sigma); \int_{\partial D}\phi\ d\sigma=0\}. 
\eqn

The transmission problem (\ref{transmission}) can be rewritten as
\beqn
    \begin{cases}
        \Delta u = 0 &\text{ in }\; D \cup (\mathbb{R}^d\backslash \overline{D} ) \, ,\\[1.5mm]
        u^+ = u^- &\text{ on }\; \partial D \, ,\\[1.5mm]
       \color{black} \gamma_m \f{\p u^{+}}{\p \nu}  =  \gamma_c \f{\p u^{-}}{\p \nu} &\text{ on }\; \partial D \, ,\\[1.5mm]
        u - u_0 = O(|x|^{1-d}) &\text{ as }\; |x| \rightarrow \infty \, .
    \end{cases}
    \label{transmission2}
\eqn
With the help of the single-layer potential, one can rewrite the perturbation  $u - u_0$, which is due to the inclusion $D$, as
\beqn
    u - u_0 = \mathcal{S}_{\partial D} [\phi] \, ,
    \label{scattered}
\eqn
where $\phi \in L^2(\partial D, d \sigma)$ is an unknown density, and $\mathcal{S}_{\partial D} [\phi]$ signifies the refraction part of the potential in the presence of the inclusion. 
By virtue of the jump relation (\ref{jump_condition}), solving the above system (\ref{transmission2})
is equivalent to solving the following integral equation for the density function $\phi \in L^2(\partial D)$:
\beqn
    \frac{\partial u_0}{\partial \nu} =  \left( \f{\gamma_c+\gamma_m}{2(\gamma_c-\gamma_m)}I - \mathcal{K}_{\partial D}^* \right) [\phi] \, .
    \label{potential2}
\eqn
This gives
\begin{equation}\label{eq:a1}
    u - u_0 = \mathcal{S}_{\partial D} \circ (\lambda I - \mathcal{K}_{\partial D}^*)^{-1} \left [  \frac{\partial u_0}{\partial \nu}  \right ]  \, ,
\end{equation}
where 
\[
\lambda := \f{\gamma_c+\gamma_m}{2(\gamma_c-\gamma_m)}.
\]
The invertibility of the operator $( \f{\gamma_c+\gamma_m}{2(\gamma_c-\gamma_m)}I - \mathcal{K}_{\partial D}^*)$
from $L^2(\partial D, d \sigma)$ onto $L^2(\partial D, d \sigma )$ {\color{black} (and respectively from
$L_0^2(\partial D, d \sigma)$ onto $L_0^2(\partial D, d \sigma)$ is proved (cf. \cite{book,kellog}),)}
provided that $|\f{\gamma_c+\gamma_m}{2(\gamma_c-\gamma_m)}| > 1/2$ via the Fredholm alternative {\color{black} (and respectively $|\f{\gamma_c+\gamma_m}{2(\gamma_c-\gamma_m)}| \geq 1/2$)}, which holds when the constants $\gamma_c$ and $\gamma_m$ are positive {\color{black} (and respectively nonegative)}. 

From \eqref{eq:a1}, we see that in order to understand the quantitative behaviour of the scattered field, one needs to investigate the mapping properties of the Neumann-Poincar\'e operator. Since  $\partial D$ is $\mathcal{C}^{2,\alpha} $, the operator $\mathcal{K}_{\partial D}^*: L^2(\partial D, d \sigma) \rightarrow L^2(\partial D, d \sigma)$ is compact (cf. \cite{ACKLM}), and its spectrum is discrete and accumulates at zero. All the eigenvalues are real and bounded by $1/2$. {\color{black} Moreover, $1/2$ is always an eigenvalue and the dimension of its associated eigenspace is the number of the connected components of $\partial D$.}
In two dimensions, it is proved that if $\lambda_i\neq 1/2$ is an eigenvalue of $\mathcal{K}_{\partial D}^*$, then $-\lambda_i$ is an eigenvalue as well. This property is known as the twin spectrum property; see \cite{plasmon1}. Moreover, it can directly verified that the eigenvalues of $\mathcal{K}_{\partial D}^*$ are invariant with respect to rigid motions and scaling. 
They can be explicitly computed for ellipses and spheres. In fact, if $a$ and $b$ denote the semi-axis lengths of an ellipse then it can be shown that $\pm (\frac{a-b}{a+b})^i$ are the corresponding eigenvalues \cite{shapiro}. 
For the sphere, they are given by $1/(2(2i+1))$; see \cite{seo}.   Some other computations of Neumann-Poincar\'e eigenvalues in different scenarios can be found in \cite{AKL,curv_Liu_3,BZ,KLY,DLL1,DLL2,DLL3,LL}.
In three dimensions, in \cite{weyl2, weyl1}, it is derived that
\begin{equation}\label{eq:b1}
\lambda_j ( \mathcal{K}_{\partial D}^* ) \sim \left\{ \frac{3 \int\limits_{\partial D} H^2(x) d \sigma_x -  \int\limits_{\partial D} K(x) d \sigma_x }{128 \pi} \right\}^{\frac{1}{2}} j^{- \frac{1}{2}},
\end{equation}
where $\lambda_j$ denotes the $j$-th Neumann-Poincar\'e eigenvalue ({\color{black} ordered according to the descending order of the magnitudes of the absolute values of the eigenvalues}), and $H(x)$ and $K(x)$ are respectively the mean and Gaussian curvatures at the point $x\in\partial D$.  Therefore, one sees that the magnitude of $\lambda_j$ not only has a decay order, but also depends on a constant related to the curvature of the inclusion.

From \eqref{eq:b1}, it is natural to expect that the curvature of the boundary of the inclusion should also enter into the scattered field in an explicit way. In fact, in what follows, we shall establish such an explicit dependence locally at a boundary point with a high magnitude of mean curvature. It turns out that the corresponding derivation is technical and tricky. For that purpose, we need to introduce the so-called GPT in arbitrary dimension in the next subsection.

\subsection{Generalized polarization tensors in arbitrary dimensions}  \label{subsection:harmonic}
{\color{black} First, from the generating function of the Gegenbauer ultraspherical harmonic polynomials $C^{(\frac{d-2}{2})}_{k}$ (as generalizations of the Legendre polynomials when $d=3$)  \cite{stein}, one has 
\[
\frac{1}{(1 - 2 at + t^2)^{\frac{d-2}{2}}} = \sum_{k=0}^{\infty}  C^{(\frac{d-2}{2})}_{k} (a) t^k \, ,
\]
{\color{black} for $a \in [-1,1], t \in \mathbb{R}$,} where
\begin{equation*}
C^{(\frac{d-2}{2})}_{k}\left( \langle \, \omega_x\, , \, \omega_y\, \rangle \right) =  c_{d,k} \sum_{ |l_1| \leq l_2 \leq ... \leq l_{d-1}  = k} Y_{l_1,..., l_{d-1}} (\omega_x) \, \overline{ Y_{l_1,..., l_{d-1}} (\omega_y) },
\end{equation*}
with $ c_{d,k}$ being certain normalization constants. Here, $\omega_x, \omega_y \in\mathbb{S}^{d-1}$ and $ Y_{l_1,..., l_{d-1}} (\omega)$ with $|l_1| \leq l_2 \leq ... \leq l_{d-1} = k$ are the spherical harmonics which form an orthonormal basis in $L^2(\mathbb{S}^{d-1})$ and are given explicitly as \cite{harmonics}:
\begin{eqnarray*}
Y_{l_1,..., l_{d-1}} (\omega_x) = \frac{1}{\sqrt{2 \pi}} e^{i l_1 \theta_1} \prod_{j=2}^{d-1} \sqrt{ \frac{2 l_j + j - 1}{2} \frac{ ( l_j + l_{j-1} + j -2 )! }{ (l_j -1)!}} \sin^{\frac{2-j}{2}} (\theta_j) P^{-(l_{j-1} + \frac{j-2}{2})}_{l_j + \frac{j-2}{2}} \left(\cos(\theta_j)\right),
\end{eqnarray*}
where $P^{\mu}_{\lambda}$ are the associated Legendre functions. Using the above facts, one has for $|x| > |y|$ the following decomposition of the Newtonian potential via substituting $t = |y|/|x|$, $a = \langle \, \omega_x\, , \, \omega_y\, \rangle $ for $d \geq 3$ \cite{stein}:
\begin{equation}\label{eq:exp1}
\begin{split}
G(x-y) =  & c_d \sum_{k=0}^{\infty} \frac{|y|^k}{|x|^{k + d - 2 }} \, C^{(\frac{d-2}{2})}_{k}\left( \langle \, \omega_x\, , \, \omega_y\, \rangle \right)\\
 = &   \sum_{k=0}^{\infty} c_{d,k} \, \frac{|y|^k}{|x|^{k + d - 2 }}  \sum_{ |l_1| \leq l_2 \leq ... \leq l_{d-1}  = k} Y_{l_1,..., l_{d-1}} (\omega_x) \, \overline{ Y_{l_1,..., l_{d-1}} (\omega_y) },
\end{split}
\end{equation}
where $\omega_x:=x/|x|, \omega_y=y/|y|$ and {\color{black} $c_{d} = \frac{1}{(d-2) \varpi_d} , c_{d,k}  = \frac{d+2k-2}{(d-2)^2 \omega_d}$ (with $\varpi_d$ being surface area of unit sphere in $\mathbb{R}^d$l)} are some dimensionality constants (c.f. also \cite[Chapter 10.9]{Bateman}) }
Similar to the expansions given by generalized polarisation tensors in two and three dimensions \cite{gpt,book}, by virtue of \eqref{eq:a1} and \eqref{eq:exp1}, one can expand the scattered potential for all $|x| > \sup\{ |x| : x \in D \} $ as
\begin{equation}\label{eq:exp2}
\begin{split}
    (u - u_0) (x)   = & \, \sum_{k=0}^{\infty}  \sum_{ |l_1| \leq l_2 \leq ... \leq l_{d-1} =k }  \, c_{d,k} \,  |x|^{- k - d + 2 } \,  Y_{l_1,..., l_{d-1}} (\omega_x) \, \\
   & \times \int_{\partial D} |y|^k \overline{ Y_{l_1,..., l_{d-1}} (\omega_y) }   \,
\left\{ (\lambda I - \mathcal{K}_{\partial D}^*)^{-1} \left [  \frac{\partial u_0}{\partial \nu}  \right] \right\} (y) d \sigma(y). 
\end{split}
\end{equation}
{\color{black}
With the above expansion, which is similar to \cite{book}, we define the generalized polarisation tensor (GPT) by choosing the incident harmonic function $\color{black} u_0 (y)= |y|^{m_{d-1}}  Y_{m_1,..., m_{d-1}} (\omega_y) $ and then taking the coefficient with respect to the function $Y_{l_1,..., l_{d-1}} (\omega_x)$ in \eqref{eq:exp2}. }
That is, 
\begin{Definition}
The generalized polarisation tensors (GPTs) of dimension $d$ with a given $\lambda$ and a domain $D \subset \mathbb{R}^d$ with a $\mathcal{C}^{2,\alpha}$ boundary are defined as
\begin{equation}\label{eq:gpt1}
\begin{split}
& \mathcal{M}_{(l_1,..., l_{d-1}), (m_1,..., m_{d-1})} (\lambda, D)\\
:= & \int_{\partial D} |y|^{l_{d-1}} \overline{ Y_{l_1,..., l_{d-1}} (\omega_y) }   \,
\left\{ (\lambda I - \mathcal{K}_{\partial D}^*)^{-1} \left [  \partial_\nu {\color{black}\left(  |x|^{m_{d-1}} Y_{m_1,..., m_{d-1}} (\omega_x) \right)}  \right ] \right\} (y) d \sigma(y),
\end{split}
\end{equation}
where $|l_1| \leq l_2 \leq ... \leq l_{d-1}$  and $|m_1| \leq l_2 \leq ... \leq m_{d-1}$.
\end{Definition}
By writing $L = (l_1,...,l_{d-1})$ and $\color{black} I_{k} = \{ L : |l_1| \leq l_2 \leq ... \leq l_{d-1} = k \}$, we handily obtain the following lemma.
\begin{Lemma}
Consider a domain $D \subset \mathbb{R^d}$ of $\mathcal{C}^{2,\alpha}$ class.  The solution to \eqref{transmission} with $$u_0(x) = \sum_{k=0}^{\infty}  \sum_{ M \in I_k }   a_{M} \, {\color{black} |x|^{k} \, Y_{M} (\omega_x)} \,$$ and $|x| > \sup\{ |x| : x \in D \} $ is given by
\beqnx
  (u - u_0) (x) =  \sum_{k=0}^{\infty}  \sum_{L \in I_k} 
\sum_{n=0}^{\infty}  \sum_{ M \in I_n} 
 \, c_{d,k} \, a_{M} \,  |x|^{- k - d + 2 } \, Y_{L} (\omega_x)\, \mathcal{M}_{L, M} (\lambda, D) \,.
\eqnx
Hence, for $x \in R \cdot \mathbb{S}^{d-1}$ with $R > \sup\{ |x| : x \in D  \} $, {\color{black} if $u$ is the scattered field generated by the incidence field $u_0 = |x|^{n} Y_M(\omega_x)$,} we have
\begin{equation}\label{eq:gpt2}
 \mathcal{M}_{L,M} (\lambda, D )  =   \frac{1}{c_{d,k}}     |R|^{2k + d - 2 }  \int_{\mathbb{S}^{d-1}}  \overline{ Y_L(\omega_x) }  \bigg({\color{black}u -  |x|^{n} Y_M(\omega_x)} \bigg) (R \, \omega_x )  d \omega_{x} \,.
\end{equation}
\end{Lemma}

Notice that this definition extends the definition of genearlized polarization tensors to an arbitrary dimension $d$. 
{\color{black} We also notice that $\{ |x|^{k} \, Y_{M} (\omega_x) \}_{k \in \mathbb{N}, M \in I_k}$ forms an $L^2$ orthoganal bases in the space $\text{Ker}(\Delta) := \{ u : \Delta u = 0 \text{ in } \mathbb{R}^d \} $, and hence the expansion for $u_0$ is general.} 
Moreover, it is easy to show that the transformation rules and exponentially decaying properties in high dimensions are similar to those in \cite{gpt,book}.
The above lemma indicates that the scattering information is fully encoded in the GPTs, $\mathcal{M}_{(l_1,..., l_{d-1}), (m_1,..., m_{d-1})} (\lambda, D) = \mathcal{M}_{L,M} (\lambda, D)$.

\subsection{Sensitivity analysis of the Neumann-Poincar\'e operator} 
In this section, we present the shape derivative of the Neumann-Poincar\'e operator $\mathcal{K}^*_{\partial D}$ in \eqref{operatorK}
associated with a shape $D$ sitting inside a general space $\mathbb{R}^d$ for any $d\ge 2$.
The special two-dimensional case was first treated in \cite{zribi,expansion}, and the general case was considered in \cite{spectral}.
Since this result is of fundamental importance for our future analysis, we shall briefly derive it here for the sake of completeness.

Given a shape $D$ sitting inside $\mathbb{R}^d$ with a $\mathcal{C}^{ {\color{black} 2 },\alpha}$ boundary, we consider a {\color{black} (local)} $\mathcal{C}^{ {\color{black} 2 },\alpha}$ parametrisation of the surface $\partial D$:
\beqn
\mathbb{X}: U \subset \mathbb{R}^{d-1} &\rightarrow& \partial D \subset \mathbb{R}^{d}, \notag \\
u = (u_1, u_2, ..., u_{d-1}) &\mapsto &\mathbb{X}(u) \notag \,,
\eqn
{\color{black} We notice that with the compactness of $\partial D$, the surface $\partial D$ is a union of finite pieces of surfaces parametrized in this manner.}
For notational sake, we often write the vector $\mathbb{X}_i := \f{\p \mathbb{X}}{\p u_i} \in \mathcal{C}^{ {\color{black} 1 },\alpha}$.
{\color{black} In the sequel, a regular parametrisation is referred to as a parametrisation $\mathbb{X}$ such that the span of the vectors $\{ \mathbb{X}_i (u) \}_{i=1}^{d-1} $ is of dimension $d-1$ for all $u \in U$. Henceforth, we only consider a regular parametrisation $\mathbb{X}$ of $\partial D$, and therefore we can always view $\mathbb{X}$ as a local embedding of the manifold into the ambient space $\mathbb{R}^{d}$.} {\color{black} We notice that $ \mathbb{X}_i$ are actually the vector fields $\frac{\partial}{\partial u_i}$ generated by the coordinate functions $u_i :  \partial D \rightarrow \mathbb{R}$ on $T(\partial D)$ embedded into $T(\mathbb{R}^{d}) \cong \mathbb{R}^{d}$.
With the above notation, we realize that $\langle \mathbb{X}_i, \mathbb{X}_j \rangle $ (where $\langle \cdot,\cdot \rangle$ is the standard Euclidean inner product on $\mathbb{R}^d$) induces a metric $g = ( g_{ij } ) =  \left( g(\frac{\partial}{\partial u_i}, \frac{\partial}{\partial u_j} ) \right) $ of $\mathcal{C}^{ {\color{black} 1 },\alpha}$ on $\partial D$ as a symmetric bilinear form on $T(\partial D)$.  We denote the inverse metric of $g$ as $g^{-1} = (g^{ij})$.   Notice that the dual basis $\{ d u_i \} $ as differential $1$-forms are given {\color{black} as $d u_i (\frac{\partial}{\partial u_j} ) = \delta_{ij}$}, i.e. via the dualization $d u_i (\cdot) = g({\color{black} \sum_{j=1}^{d} g^{ij} \frac{\partial}{\partial u_j} } , \cdot ) $ under $g$. In the local coordinate system, for a given $1$-form $ \sum_{i=1}^{d-1} f_i \, d u_i  $, its dual vector field is given by $\sum_{i,j=1}^{d-1} g^{ij}  f_i \frac{\partial}{\partial u_j} $.} For a given set of $d-1$ vectors $\{v_i\}_{i = 1}^{d-1}$, we denote by
the $d-1$ cross product $\times_{i=1}^{d-1} v_i = v_1 \times v_2 ... \times v_{d-1} $
as the dual vector of the functional $\det(\, \cdot \,, v_1, v_2, ..., v_{d-1} )$, i.e., $\langle w, \times_{i=1}^{d-1} v_i \rangle = \det(w, v_1, v_2, ..., v_{d-1} ) $ for any $w$, which is guaranteed to exist by the Riesz representation theorem.  {\color{black} It is clear that if $v_i = v_j$ for some $i \neq j$, one has $\times_{i=1}^{d-1} v_i = 0$.} {\color{black} We would like to point out that the notations can be simplified if we dualize $\{v_i\}_{i = 1}^{d-1}$ in $\mathbb{R}^d$ to its dual vectors $\{E_i\}_{i = 1}^{d-1}$, in which case, $\times_{i=1}^{d-1} v_i $ will be dualized to the $d-1$-form given by the wedge product $E_1 \wedge E_2 ... \wedge E_{d-1} $ on $\mathbb{R}^d$. 
However, since we are always working with embedded surfaces with explicit local embeddings $\mathbb{X}$, for simplicity and from now on, we shall avoid the use of intrinsic notations, e.g. $ \frac{\partial}{\partial u_i }$, $ d u_i $, differential forms and the wedge product. We also do not adopt the Einstein's summation notation, and it is noticed that our upper and lower indices do not align with this convention to avoid any misunderstanding of the upper indices $i$. }

From the fact that $\mathbb{X}$ is regular, we readily infer by definition that $\times_{i=1}^{d-1} \mathbb{X}_i$ is
non-zero, and the normal vector $\nu := \times_{i=1}^{d-1} \mathbb{X}_i / |\times_{i=1}^{d-1} \mathbb{X}_i |  \in \mathcal{C}^{ {\color{black} 1},\alpha}$
is well-defined.

Now we consider an $\varepsilon$-perturbation of $D$, namely
$\partial D_\varepsilon$ given by
\beqn
    \partial D_\varepsilon :=\{\widetilde{x} \, \big| \, \widetilde{x} = x+ \varepsilon h(u) \nu(x) \, , \,  {\color{black} x = \mathbb{X}(u) }  \in \partial D \}\,,
    \label{variationvariationD}
\eqn
with $h \in \mathcal{C}^{1,\alpha} (\partial D)$.
Let $\Psi_\varepsilon(x): = x + \varepsilon h(u) \nu(x)$ be a diffeomorphism from $\partial D$ to $\partial D_{\varepsilon}$ when $\varepsilon$ is small enough.  {\color{black} (The fact that $\Psi$ is a diffeomorphism can be easily checked via Hadamard-Caccioppoli Theorem \cite{Hadamard}, noting that $D \Psi_\varepsilon: T(\partial \Omega) \rightarrow T ( \partial D_{\varepsilon} )$ is nonsingular with $D \Psi_\varepsilon =Id + O(\varepsilon)$ and the big $O$ is uniform over $x \in \partial \Omega$ since $h$ and $x 
\mapsto \nu(x)$ are both in {\color{red} $\mathcal{C}^1$.} )}
It is directly verified that
\beqn
\mathbb{X}^{\varepsilon}: U \subset \mathbb{R}^{d-1} &\rightarrow& \partial D_\varepsilon \subset \mathbb{R}^{d}, \notag \\
u = (u_1, u_2, ..., u_{d-1}) & \mapsto & \Psi_\varepsilon [u]= \mathbb{X}(u) +  \varepsilon h(u) \nu(\mathbb{X}(u)) \notag \, ,
\eqn
is a regular parametrisation over $\partial D_\varepsilon$ for sufficiently small $\varepsilon\in\mathbb{R}_+$. 
Directly from the definition, we have {\color{black} $\mathbb{X}^{\varepsilon}_i = \mathbb{X}_i +  \varepsilon \f{ \p h}{\p u_i} \nu + \varepsilon h \bar{\nabla}_{\mathbb{X}_i} \nu  = \mathbb{X}_i +  \varepsilon \f{ \p h}{\p u_i} \nu + \varepsilon h \sum_{j = 1}^{d-1} \sum_{k = 1}^{d-1} g^{jk} A_{ki} \mathbb{X}_j $,}  where the matrix $A_{ij} \in \mathcal{C}^{ { \color{black} 0 } ,\alpha}$ is defined as
\beqn
A := ( A_{ij} ) =  \langle \textbf{II}(\mathbb{X}_i ,\mathbb{X}_j), \nu \rangle \notag \,. 
\eqn
Here, $\textbf{II}$ is the second fundamental form given by \cite{Kobayash2}
\beqn
\textbf{II} : T(\p D) \times T(\p D) &\rightarrow& T^\perp (\p D), \notag \\
\textbf{II}(v,w)&=& { \color{black} \langle \bar{\nabla}_{v} \nu, w \rangle \nu = -  \langle \nu,  \bar{\nabla}_{v}  w \rangle \nu, } \notag
\eqn
where $\bar{\nabla}$ is the standard covariant derivative on the (ambient) {\color{black} Euclidean space} $\mathbb{R}^{d}$.
From the multi-linearity and alternating property of the $d-1$ cross product, we can readily calculate
at any point $a \in U$ that
\beqn
&  &\times_{i=1}^{d-1} \mathbb{X}^{\varepsilon}_i (a) \notag  \\
& = &
\times_{i=1}^{d-1} \left( \mathbb{X}_i (a)+  \varepsilon \f{ \p h}{\p u_i} \nu (a) + \varepsilon h  \sum_{j,k = 1}^{d-1}  g^{jk} A_{ki}  \mathbb{X}_j (a)  \right) \notag 
\\
&=&  \left(1+ \varepsilon h(a) \text{tr}_g (A)(a) \right) \times_{i=1}^{d-1} \mathbb{X}_i (a) \, { \color{black}+ \varepsilon \sum_{j=1}^{d-1} \f{ \p h}{\p u_j} (a)  \times_{i=1}^{j-1}  \mathbb{X}_i (a) \times \nu (a) \times \times_{i=j+1}^{d-1} \mathbb{X}_i (a)  } + {O}(\varepsilon^2), \notag \\
& &  \label{expansion1}
\eqn
where the constant in big $O$ is bounded by $|A(\mathbb{X}(a))|$ at the point $\mathbb{X}(a)$ and $||h||_{\mathcal{C}^1} $ and
$$ \text{tr}_g (A) (a) :=  \text{tr}_g (A) (\mathbb{X}(a)) =   \sum_{j,k = 1}^{d-1} g^{jk} A_{kj} ( \mathbb{X}(a)) := (d-1) H (\mathbb{X}(a)) \, , $$
with $H (\mathbb{X}(a))$ being the mean curvature at the point $ \mathbb{X}(a)$. {\color{black}Using this as well as the fact that
\begin{eqnarray*}
&  &\left\langle  \times_{i=1}^{d-1} \mathbb{X}_i (a) ,   \times_{i=1}^{j-1}  \mathbb{X}_i (a) \times \nu (a) \times \times_{i=j+1}^{d-1} \mathbb{X}_i (a) \right \rangle \\
 &=&  |\times_{i=1}^{d-1} \mathbb{X}_i (a)| \det( \nu(a) ,  \mathbb{X}_1 (a),...,  \mathbb{X}_{j-1} (a), \nu(a),  \mathbb{X}_{j+1} (a), ...,  \mathbb{X}_{d-1} (a)  ) \\
& =& 0,
\end{eqnarray*}
it is ready to obtain {\color{black} from \eqref{expansion1} } that
\beqnx
|\times_{i=1}^{d-1} \mathbb{X}^{\varepsilon}_i (a)| = \left(1+ \varepsilon h(a) \text{tr}_g (A)(a) \right)  | \times_{i=1}^{d-1} \mathbb{X}_i (a) | \, + {O}(\varepsilon^2),,
\eqnx
{\color{black}
which directly leads to
\beqnx
\frac{|\times_{i=1}^{d-1} \mathbb{X}^{\varepsilon}_i (b)|}{|\times_{i=1}^{d-1} \mathbb{X}^{\varepsilon}_i (a)|} &= &\frac{ \left(1+ \varepsilon h(b) \text{tr}_g (A)(b) \right)  | \times_{i=1}^{d-1} \mathbb{X}_i (b) |  }{ \left(1+ \varepsilon h(a) \text{tr}_g (A)(a) \right)  | \times_{i=1}^{d-1} \mathbb{X}_i (a) |  } \, + {O}(\varepsilon^2) \\
&= & \bigg( 1 +  \varepsilon h(b) \text{tr}_g (A)(b) -   \varepsilon h(a) \text{tr}_g (A)(a) \bigg) \frac{ | \times_{i=1}^{d-1} \mathbb{X}_i (b) |  }{  | \times_{i=1}^{d-1} \mathbb{X}_i (a) |  } \, + {O}(\varepsilon^2) \,.
\eqnx
Combining the above with \eqref{expansion1}, it yields that 
\beqnx
& & \times_{i=1}^{d-1} \mathbb{X}^{\varepsilon}_i (a) \, \frac{|\times_{i=1}^{d-1} \mathbb{X}^{\varepsilon}_i (b)|}{|\times_{i=1}^{d-1} \mathbb{X}^{\varepsilon}_i (a)|} \\
&=&
 \bigg(1+  \varepsilon h(b) \text{tr}_g (A)(b)  \bigg)  \times_{i=1}^{d-1} \mathbb{X}_i (a)  \frac{ | \times_{i=1}^{d-1} \mathbb{X}_i (b) |  }{  | \times_{i=1}^{d-1} \mathbb{X}_i (a) |  }  \, \\
& &  { \color{black}+ \varepsilon \sum_{j=1}^{d-1} \f{ \p h}{\p u_j} (a)  \times_{i=1}^{j-1}  \mathbb{X}_i (a) \times \nu (a) \times \times_{i=j+1}^{d-1} \mathbb{X}_i (a)  }   \frac{ | \times_{i=1}^{d-1} \mathbb{X}_i (b) |  }{  | \times_{i=1}^{d-1} \mathbb{X}_i (a) |  } \,+ {O}(\varepsilon^2)
\eqnx
or that}
\beqn
&   &\langle \, \cdot \, , \nu^{\varepsilon}(a) \rangle \, d \sigma^{\varepsilon} (b) \notag\\
 &=& \langle \, \cdot \, ,  \times_{i=1}^{d-1} \mathbb{X}^{\varepsilon}_i (a) \rangle \, \frac{|\times_{i=1}^{d-1} \mathbb{X}^{\varepsilon}_i (b)|}{|\times_{i=1}^{d-1} \mathbb{X}^{\varepsilon}_i (a)|} \, d b \notag \\
&=&  \left(1+ \varepsilon h(b) \text{tr}_g (A)(b) \right) \langle \, \cdot \, , \nu(a) \rangle \, d \sigma (b)  \notag \\
&  &
{ + \color{black} \varepsilon \sum_{j=1}^{d-1} \f{ \p h}{\p u_j} (a)  \left\langle \, \cdot \, ,    \frac{ \times_{i=1}^{j-1}  \mathbb{X}_i (a) \times \nu (a) \times \times_{i=j+1}^{d-1} \mathbb{X}_i (a)  }{|\times_{i=1}^{d-1} \mathbb{X}_i (a)|}  \right \rangle   \, d \sigma (b)   } + O(\varepsilon^2 )\,,
\label{Taylor1}
\eqn
where $\nu^{\varepsilon}(a)$ denotes the normal vector at $\mathbb{X}^{\varepsilon}(a)$.} Moreover, for two arbitrary points $x,y \in \partial D$ given by $x =\mathbb{X} (a), y = \mathbb{X} (b)$ for some $a, b \in U$,
we have
\beqn
\mathbb{X}^{\varepsilon} (a) - \mathbb{X}^{\varepsilon} (b)  = \mathbb{X} (a) - \mathbb{X} (b) + \varepsilon K(a,b) [h],
\label{Taylor2}
\eqn
where $K(a,b)[h] := h(a) \nu(a)- h(b) \nu(b)$.
Therefore, by the Taylor expansion of $|\mathbb{X}^{\varepsilon} (a) - \mathbb{X}^{\varepsilon} (b)|^{-d}  $ in $\varepsilon$, it follows that
\beqn
\begin{split}
& | \mathbb{X}^{\varepsilon} (a) - \mathbb{X}^{\varepsilon} (b)|^{-d} \\
= &  |\mathbb{X} (a) - \mathbb{X} (b)|^{-d} - d \, \varepsilon  |\mathbb{X} (a) - \mathbb{X} (b)|^{-d-2}  \langle \mathbb{X} (a) - \mathbb{X} (b) ,  K(a,b) [h] \rangle + O(\varepsilon^2)\,. \label{Taylor3}
\end{split}
\eqn
Combining \eqref{Taylor1} and \eqref{Taylor3}, we obtain the following series expression
\beqn
\frac{ \langle \mathbb{X}^{\varepsilon} (a) - \mathbb{X}^{\varepsilon} (b) \, , \nu^{\varepsilon}(a) \rangle}{| \mathbb{X}^{\varepsilon} (a) - \mathbb{X}^{\varepsilon} (b)|^{d}} \, d \sigma^{\varepsilon} (b)
 := \sum_{n=0}^\infty \varepsilon^n  \mathbb{K}_{h,n} (a,b) \,  d \sigma (b) , \label{Taylorfinal}
\eqn
{\color{black}where
\beqn
\mathbb{K}_{h,0} (a,b) &:=& \frac{ \langle \mathbb{X} (a) - \mathbb{X} (b) \, , \nu(a) \rangle}{| \mathbb{X}(a) - \mathbb{X}(b)|^{d}} ,   \notag\\
\mathbb{K}_{h,1} (a,b) &:=&  \frac{ \langle \mathbb{X} (a) - \mathbb{X} (b) \, , h(b) \text{tr} (A)(b)   \nu(a) \rangle + \langle K(a,b)[h], \nu(a) \rangle }{| \mathbb{X}(a) - \mathbb{X}(b)|^{d}}    \notag\\
&  & { \color{black} + \frac{  \sum_{j=1}^{d-1} \f{ \p h}{\p u_j} (a)  \, \langle \mathbb{X} (a) - \mathbb{X} (b) \, ,  \times_{i=1}^{j-1}  \mathbb{X}_i (a) \times \nu(a) \times \times_{i=j+1}^{d-1} \mathbb{X}_i  (a)   \rangle }{| \mathbb{X}(a) - \mathbb{X}(b)|^{d} |\times_{i=1}^{d-1} \mathbb{X}_i (a)| }  }  \notag\\
&&  - d  \, \frac{ \langle \mathbb{X} (a) - \mathbb{X} (b) \, , K(a,b)[h]  \rangle \langle \mathbb{X} (a) - \mathbb{X} (b) , \nu(a) \rangle }{| \mathbb{X}(a) - \mathbb{X}(b)|^{d+2}}  \,,  \notag
\eqn
and the higher-order terms $\mathbb{K}_{h,i}$ can be explicitly calculated from the higher-order expansions in \eqref{Taylor1}-\eqref{Taylor3} in a similar fashion, with coefficients depending only on $h(x)$ and $\partial h$.}
As a result, we see that the kernel of the Neumann-Poincar\'e operator varies analytically with respect to
$\varepsilon$ along any direction $h \in \mathcal{C}^{2,\alpha} (\partial D)$.
{\color{black}
In fact, upon a careful inspection of the coefficients and tracing back the coefficients in the series, one easily verifies that each of the kernels $\mathbb{K}_{h,n} (a,b)$ is of the form 
\[
  \mathbb{K}_{h,n} (a,b) = \sum_{i=1}^{d-1} F_{i,n}(a) \frac{ a_i - b_i }{| \mathbb{X}(a) - \mathbb{X}(b)|^{d}}   +  \frac{G_n(a,b)}{| \mathbb{X}(a) - \mathbb{X}(b)|^{d-2}}\, ,
\]
where $ \| F_{j,n}(a) \|_{C^0} + \| G_n(a,b) \|_{C^0} \leq P_d(n)  \| \mathbb{X} \|_{C^2}^n \| h\|_{C^1}^n $ for a certain polynomial $P_d(n)$ with respect to $n \geq 0$ (depending on $d$.)   Now, via an application of the Coifman-McIntosh-Meyer theorem (c.f. \cite{CMM}, Theorem IX), one can obtain that for each $n \geq 0$, the operator
\[
 \phi(\mathbb{X} (a)) \in L^2(U)  \mapsto \text{p.v.} \int_{U} \mathbb{K}_{h,n}(a,b) \, \phi( \mathbb{X}(b))d \sigma(b) 
\]
is bounded from $L^2(U)$ to $L^2(U)$, with the corresponding norm bounded by $ P_d(n)  \| \mathbb{X} \|_{C^2}^n \| h\|_{C^1}^n$ for some polynomial $P_d(n)$.  

Starting from now on, by a slight abuse of notation we shall not distinguish
between $F(x)$ and $F(a)$ for any function $F$ over $\p D$ if $x =\mathbb{X} (a) \in \p D$, which should be clear from the context.
Next we define a sequence of integral operators $\mathcal{K}^{(n)}_{D,h}$: $L^2 (\partial D) \to L^2 (\partial D)$ by
\beqn
   \mathcal{K}^{(n)}_{D,h} [\phi](x) := \text{p.v.} \int_{\partial D} \mathbb{K}_{h,n}(x,y)\phi(y)d \sigma(y), ~~\forall\, \phi \in L^2 (\partial D)\, ,
   \label{eq:kernel1}
\eqn
for $n \geq 0$ where each of the operators is bounded:
\beqn
   \| \mathcal{K}^{(n)}_{D,h} \|_{L^2(\partial D)} \leq P_d (n) \| \mathbb{X} \|_{C^1} \| h\|_{C^1}^n,
\label{bound_K_n}
\eqn
for a polynomial $P_d(n)$ with respect to $n$. It is easily seen that for sufficiently small $\varepsilon>0$, the series $\sum_{n=0}^\infty \mathcal{K}^{(n)}_{D,h} \varepsilon^n $ converges absolutely on $\mathcal{L} (L^2(\partial D) , L^2(\partial D) )$.}  Notice here that the notation $\mathbb{K}_{h,0}(x,y)$ is nothing but the kernel of $\mathcal{K}^*_{\partial D}$ itself.  
Therefore we can directly obtain the following result.

\begin{Theorem} \label{theoremK}
Given $\partial D \in \mathcal{C}^{ {\color{black} 2 },\alpha}$ and $h \in \mathcal{C}^{ {\color{black} 1 },\alpha} (\partial D)$, for $N \in \mathbb{N}$, there exists a constant $C$ depending only on $N$, $||\mathbb{X}||_{\mathcal{C}^2}$ and $||h||_{\mathcal{C}^1}$ such that
the following estimate holds for any $\widetilde{\phi} \in L^2 (\partial D_\varepsilon)$ and
$\phi := \tilde{\phi}\circ \Psi_\varepsilon$:
\beqn
   \bigg|\bigg|  ( \mathcal{K}^*_{\partial D_\varepsilon} [ \tilde{\phi} ]  ) \circ \Psi_\varepsilon  - \mathcal{K}^*_{\partial D}[{\phi} ] - \sum_{n=1}^N \varepsilon^n  \mathcal{K}^{(n)}_{D,h} [\phi] \bigg|\bigg|_{L^2(\partial D)}
   \leq C \varepsilon^{N+1} ||\phi||_{L^2(\partial D)}\,.
   \label{seriesvariation}
\eqn
In particular, the kernel of $\mathcal{K}^{(1)}_{D,h}$ can be explicitly expressed by
\begin{equation}\label{eq:ker1}
\begin{split}
\mathbb{K}_{h,1} (x,y) =& \frac{ \langle x- y \, ,    \nu(x)  \rangle h(y) \text{tr}_g (A)(y)  + \langle K(x,y)[h], \nu(x) \rangle }{|x-y |^{d}}\\
&  { \color{black} + \frac{  \sum_{j=1}^{d-1} \p_i h(x)   \, \langle x - y \, ,  \times_{i=1}^{j-1}  \mathbb{X}_i (x) \times \nu (x) \times \times_{i=j+1}^{d-1} \mathbb{X}_i (x)    \rangle }{| x-y |^{d}  |\times_{i=1}^{d-1} \mathbb{X}_i (x)|}   }  \\
&   { \color{black} - } d  \, \frac{ \langle x-y \, , K(x,y)[h]  \rangle \langle x-y , \nu(x) \rangle }{|x-y|^{d+2}}  \,, \\
\end{split}
\end{equation}
where $K(x,y)[h] := h(x) \nu(x)- h(y) \nu(y)$ and $ \text{tr}_g (A)(y)  = (d-1) H(y) $ with  $H(y)$ denoting the mean curvature of the surface at $y$. 
\end{Theorem}

{\color{black} We remark that this result concides with Lemma 3.1 in  \cite{zribi} and Theorem 2.1 in \cite{expansion} that they coincide. Here, we notice that the formula in Lemma 3.1 in \cite{zribi} is given with respect to the arc-length parametrisation and hence $|\mathbb{X}_1| = 1$, and that when $d=2$, in Lemma 3.1 $\tau=-A_{11}$ in our convention. We also notice that for any $v \in \mathbb{R}^2$, one has from our definition that $\langle v ,\times \nu(x) \rangle = \det( v , \nu(x) ) = - \langle v , \mathbb{X}_1(x) \rangle$ and hence $\times \nu(x) =  - \mathbb{X}_1(x) $; and similarly $\times \mathbb{X}_1(x)  =  \nu(x)  $.
}

\subsection{$\mathcal{K}^*_{\partial D} $ and $\mathcal{K}^{(1)}_{D,h}$ as pseudo-differential operators}

In this subsection, we derive some crucial properties of $\mathcal{K}^*_{\partial D} $ and $\mathcal{K}^{(1)}_{D,h}$, verifying that they are pseudo-differential operators with particular orders and obtain their principal symbols.

{ \color{black}
Before we continue, in order to facilitate our subsequent discussions, we recall some notations in Riemannian geometry (cf. e.g. \cite{Kobayash1,Kobayash2,DoCarmo,Klingenberg,Jost,Lee,Petersen}). {\color{black}For the time being, let us assume $\partial D \in \mathcal{C}^{\infty}$.}
We write $\nabla$ for the Levi-Civita connection on $\partial \Omega$, and we recall the induced connection $1$-form as follows: given a vector field $X$, one defines the map
\beqnx
Y \mapsto \nabla_X Y ,
\eqnx
and now the map sending $X$ to this linear map is an endomorphism-valued $1$-form.
In a local coordinate system, it can be represented in terms of the Christoffel symbols as:
\[
 \nabla_{\mathbb{X}_{i}} \mathbb{X}_{j}  := \sum_{k=1}^{d-1}  \Gamma^k_{ij} \mathbb{X}_{k} \quad \mbox{for\ \ $i,j = 1,..., d-1$.}
\]
We also denote the connection $1$-form from $\bar{\nabla}$ on the ambient space $\mathbb{R}^d$ via the Christoffel symbols $\bar{\Gamma}^k_{ij} $ as:
\[
 \bar{\nabla}_{e_{i}} e_{j}  := \sum_{k=1}^{d}  \bar{\Gamma}^k_{ij}  e_k \, ,
\]
for the basis $\{ e_i \}_{i=1,...,d-1} = \{ \mathbb{X}_{i} \}_{i=1,...,d-1}  $ and $e_d = \nu$.  We realize that with this choice of the semi-geodesic coordinate (i.e. the metric in the ambient space $\bar{g} = \begin{pmatrix} 1 & 0 \\ 0 & g \end{pmatrix}$), we have $\bar{\Gamma}^k_{ij} = \Gamma^k_{ij} $ for $i,j,k\neq d$ and  $ \bar{\Gamma}^d_{id} = \bar{\Gamma}^d_{dj}=0 $ and $\bar{\Gamma}^d_{ij} = A_{ij} $ for $i,j \neq d$.
Let $R$ be the Riemannian curvature tensor on $\partial D$ as follows: given vector fields $X,Y $, we define the map
\beqnx
Z \mapsto R(X,Y) Z := ( \nabla_X \nabla_Y - \nabla_Y \nabla_Z - \nabla_{[X,Y]} ) Z,
\eqnx
and now the map sending $(X,Y)$ to this linear map is an endomorphism-valued $2$-form \cite{Kobayash1,DoCarmo,Jost}.
In a local coordinate system, we can write
\[
R( \mathbb{X}_{i},  \mathbb{X}_{j})  \mathbb{X}_{k}   : = \sum_{l=1}^{d-1} R^l_{ijk}  \mathbb{X}_{l} = \sum_{l=1}^{d-1}  \left( \partial _{j}\Gamma^{l}_{ik}-\partial _{i}\Gamma ^{l}_{jk}+ \sum_{m=1}^{d-1}  \Gamma ^{l }_{j m }\Gamma ^{m}_{ i k }- \sum_{m=1}^{d-1} \Gamma ^{l}_{ m k }\Gamma ^{m }_{i j } \right) \mathbb{X}_{l}. 
\]
We also write $ R_{ijkl} := \sum_{m=1}^{d-1}  g_{im} R^{m}_{ijk}$.
Then given two vector fields $Y$ and $Z$, the Ricci curvature $\text{Ric} (Y,Z) $ is defined to be the trace of the map $X \mapsto R(X,Y) Z $, i.e. 
\beqnx
\text{Ric} (Y,Z)  :=\sum_{i=1}^{d-1} ( R(e_i,Y) Z  )^i. 
\eqnx
In a local coordinate system, we have
\[
\text{Ric}_{jk} := \text{Ric} (e_j,e_k) = \sum_{i=1}^{d-1} R^i_{ijk} \,.
\]
Notice that $(X,Y) \mapsto \text{Ric} (X,Y)$ is a symmetric bilinear form.
Readers may check the full exposition of these notations in e.g. \cite{Kobayash1,Kobayash2,DoCarmo,Klingenberg,Jost,Lee,Petersen}.
}

{ \color{black}
For notational sake, let us also recall the following definitions for pseudo-differential operators \cite{Hor1,Hor2,homogeneous,100years} that will be considered in our study:
\[
\begin{split}
 \bigcup_{i} V_i = \partial D \,,\quad & F_i : \pi^{-1} (V_i) \rightarrow V_i \times \mathbb{R}^{d-1}\,,  \quad  \sum_i \psi_i = 1 \,, \quad \text{supp}(\psi_i) \subset V_i\, ; \\
S^m (T^*(\partial D) )  :=&  \left \{ a: T^*(\partial D)  \backslash \partial D \times \{0\}  \rightarrow \mathbb{C} \, ; \, a = \sum_{i} \psi_i F_i^* a_i, a_i \in S^m (V_i \times \mathbb{R}^{d-1} \backslash \{0\} ) \right \}; \\
S^m (V_i \times \mathbb{R}^{d-1} )  := & \bigg\{ a:  V_i \times ( \mathbb{R}^{d-1} \backslash \{0\} )  \rightarrow \mathbb{C} \, ;\\
&\qquad a \in \mathcal{C}^{\infty} (  V_i \times ( \mathbb{R}^{d-1} \backslash \{0\} )  ) \, , \, | \partial_\xi^\alpha \partial_x^{\beta} a (x,\xi) | \leq C_{\alpha, \beta} ( |\xi| )^{ m - |\alpha| } \bigg\}. 
\end{split}
\]
{\color{black} where $\pi: T^*(\partial D)  \rightarrow \partial D$ is the bundle projection from the total space $T^*(\partial D) $ to the base space $\partial D$, and for each $i$, $V_i$ is a local trivilization neigborhood, $F_i$ are the local trivialization and $F_i^* a_i$ are the pullback of the function $a_i$ by $F_i$. }
Here and also in what follows, $ \Phi \text{SO}^{m}$ is the space of pseudo-differential operators with the action defined locally as $\mathrm{Op}_{a} := \mathcal{F}^{-1} \circ m_{a} \circ\mathcal{F} $, where $\mathcal{F}$ is the Fourier transform, and $m_{a}$ is the multiplication by $a$, given that $a \in \mathcal{S}^m(T^*(\partial D))$ belongs to the symbol class of order $m$.  To be more precise, if $x \in U_i$, $\mathbb{X}: U_i \rightarrow V_i $, then 
\[
 [\mathrm{Op}_{a} (\phi) ] (x) : = \int_{\mathbb{R}^{d-1}} \int_{U_i} a \left(F_i^{-1} (\mathbb{X}(x),  \xi ) \right) e^{ i \langle x-y,\xi \rangle } \phi(\mathbb{X}(y)) \psi_i (\mathbb{X}(y) )\, 
d y\, d \xi \,.
\]
{\color{black} whenever $\phi$ is a function on $\partial D$ and that the above integral is well-defined.}
Note that this uniquely defines an operator modulus $\Phi \text{SO}^{m-1}$.
Moreover, if an operator $K \in \Phi \text{SO}^{m}$ is given by
\[
[K (\phi) ] (x) : =  \int_{\partial D} K(x,y) \phi(y) d \sigma(y) \,,
\]
for $x \in \partial D$,
then
\beqn
(2 \pi )^{ -(d-1)} K (x,y) =   \mathcal{F}^{-1} \left[  a\left(x , \cdot \right) \right] (x-y) + \mathcal{O}( |x-y|^{m-1}) \,,
\label{symbol}
\eqn
via the partition of unity and the local trivialisation (by an abuse of notation to identify $x,y$ with $a,b$ and $F_i^{-1} (\mathbb{X}(a),  \xi ) $ and $(a,\xi)$ in the above manner).   We also name $a_0(x,\xi) := a(x,\xi)  (\text{mod} \, S^{m-1})$ as the principal symbol of the operator $K$, and write
\[
a(x,\xi) = a_0(x,\xi) + O(|\xi|^{m-1}) \quad \text{ and } \quad
K = \mathrm{Op}_{ a_0 } (  \text{mod}\, \Phi \text{SO}^{m-1} ) \,.
\]
We remark that this definition of the principal symbol turns out to be independent of the choice of the coordinate system. The remaining term in $\Phi \text{SO}^{m-1} $ depends on the coordinate functions and the corresponding symbol depends on the $\mathcal{C}^2$ norm of the change of the coordinates.
Nevertheless, we would also like to remind that there is a definition for the sub-principal symbol which is independent of the choice of the coordinate system, but we shall not require it for our study. In fact, there is a coordinate-free way to uniquely define the pseudo-differential operator on a manifold; see e.g. \cite[Chapter 6]{Melrose}. However, we would like to suppress further technicality. We notice that our purpose can be achieved by fixing a particular coordinate system, and after some computational steps, we can obtain results that are independent of the choice of the coordinate system, and hence a geometric property. 
}

{
{\color{black}
If we have to work on $\partial D \in \mathcal{C}^{k, \alpha}$ for some $k \geq 0$, namely a local graph of a $\mathcal{C}^{k, \alpha}$ function, we can still view it as a $\mathcal{C}^{\infty}$ manifold (i.e. only allowing smooth change of coordinates) but endowed with a $\mathcal{C}^{k-1,\alpha}$ metric $g_{ij} \in \mathcal{C}^{k-1,\alpha}$ (and a parametrisation $\mathbb{X} \in  \mathcal{C}^{k,\alpha}$ can now be viewed as a $\mathcal{C}^{k,\alpha}$-local isometric embedding of the $\mathcal{C}^{\infty}$-manifold into $\mathbb{R}^d$.)  Now if we take a choice of parametrisation $\mathbb{X} \in  \mathcal{C}^{\tilde{k},\tilde{\alpha}}$ with $\tilde{k} \leq k, \tilde{\alpha} \leq \alpha$, then at a first glance, it seems that we only have
\[\vspace*{-2mm}
g_{ij} \in \mathcal{C}^{\tilde{k}-1,\tilde{\alpha}}, \quad  \Gamma_{ij}^k , A_{ij} \in \mathcal{C}^{\tilde{k}-2,\tilde{\alpha}}  \quad \text{ and }\ \ R_{ijkl}, \text{Ric}_{ij} \in \mathcal{C}^{\tilde{k}-3,\tilde{\alpha}} \,,
\]
which make $R,\text{Ric}$ not well-defined when $\tilde{k} = 2$. However, we can actually improve the regularity and define $R$ and $\text{Ric}$ via the Arzela-Ascoli theorem (i.e. the compactness of $\mathcal{C}^{\tilde{k}-2,\beta}$ in $\mathcal{C}^{\tilde{k}-2,\alpha}$ for $\tilde{\beta} > \tilde{\alpha}$), Gauss-Codazzi formula \cite{Jost,Kobayash2} and the fact that the ambient space $\mathbb{R}^d$ is flat. In fact, we can always approximate $\mathbb{X} \in  \mathcal{C}^{\tilde{k},\tilde{\alpha}}$ by a sequence of smooth embeddings, $\tilde{ \mathbb{X} }$, with the induced $\tilde{g}, \tilde{\Gamma}, \tilde{A}, \tilde{R}$ in $C^{\infty}$, which satisfy the Gauss-Codazzi formula:
\[
0 = \tilde{R}_{ijkl}(x) + \tilde{A}_{ik}(x) \tilde{A}_{jl}(x) - \tilde{A}_{jk}(x) \tilde{A}_{il}(x).
\]
Notice that the term $\tilde{A}_{ik} \tilde{A}_{jl} - \tilde{A}_{jk} \tilde{A}_{il}$ is uniformly bounded in $\mathcal{C}^{\tilde{k}-2,\tilde{\alpha}}$. Then with the Arzela-Ascoli theorem, we render a convergence of the approximation $\tilde{A}_{ik} \tilde{A}_{jl} - \tilde{A}_{jk} \tilde{A}_{il}$ in $\mathcal{C}^{\tilde{k}-2,\tilde{\beta}}$.  With this we therefore obtain
\beqn
g_{ij} \in \mathcal{C}^{\tilde{k}-1,\tilde{\alpha}}, \quad  \Gamma_{ij}^k , A_{ij} \in \mathcal{C}^{\tilde{k}-2,\tilde{\alpha}}  \quad \text{ and } R_{ijkl}, \text{Ric}_{ij} \in \mathcal{C}^{\tilde{k}-2,\tilde{\beta}} \, ,
\label{regularity}
\eqn
which make the last two terms well-defined even when $\tilde{k}=2$.
On $\partial D \in \mathcal{C}^{k,\alpha}$, we also need to consider pseudo-differential operators {\color{black} with coefficients of limited regularity (in particular in H\"older spaces $\mathcal{C}^{\tilde{k},\tilde{\alpha}}$ for $\tilde{k} \leq k,  \tilde{\alpha} \leq \alpha $
\cite{Taylorbook,Pfeuffer,Hor1,Hor2}) }. We can only have the highest possible regularity $\psi_i(\,\cdot\,) , F_i^{-1} ( \, \cdot \, ,\xi) \in  \mathcal{C}^{k,\alpha}$, and hence $  \phi_i (\cdot)  F_i^* a_i(\, \cdot \,, \xi) $ only has regularity of $\mathcal{C}^{\min\{k, \tilde{k} \}, \min\{\alpha, \tilde{\alpha} \}}$ whenever $a_i ( \, \cdot \m, \xi ) \in \mathcal{C}^{\tilde{k},\tilde{\alpha}} $.  For this purpose, we only consider $(\tilde{k},\tilde{\alpha})$ such that $ 0 \leq \tilde{k} \leq k,  0 <  \tilde{\alpha} \leq \alpha $ and write
\[
\begin{split}
\mathcal{C}^{\tilde{k},\tilde{\alpha}} S^m (T^*(\partial D) )  :=&  \left \{ a: T^*(\partial D)  \backslash \partial D \times \{0\}  \rightarrow \mathbb{C} \, ; \, a = \sum_{i} \psi_i F_i^* a_i, a_i \in \mathcal{C}^{\tilde{k},\tilde{\alpha}}  S^m (V_i \times \mathbb{R}^{d-1} \backslash \{0\} ) \right \}; \\
\mathcal{C}^{\tilde{k},\tilde{\alpha}} S^m (V_i \times \mathbb{R}^{d-1} )  := & \bigg\{ a:  V_i \times ( \mathbb{R}^{d-1} \backslash \{0\} )  \rightarrow \mathbb{C} \, ;
\| \partial_\xi^\beta a (x,\xi) \|_{\mathcal{C}^{\tilde{k},\tilde{\alpha}}} \leq C_{\beta} ( |\xi| )^{ m - |\beta| } \bigg\}, 
\end{split}
\]
{\color{black} where $\| \cdot \|_{\mathcal{C}^{\tilde{k},\tilde{\alpha}}}$ in here is to be understood for the $x$ dependence of $\partial_\xi^\beta a (x,\xi)$ only}
with an additional specification $\mathcal{C}^{\tilde{k},\tilde{\alpha}} $ in front of $S^m$.
By the same token, let us define $ \mathcal{C}^{\tilde{k},\tilde{\alpha}} \Phi \text{SO}^{m}$ accordingly as before as the space of operators with the action defined locally as $\mathrm{Op}_{a} := \mathcal{F}^{-1} \circ m_{a} \circ\mathcal{F} $ provided that $a \in  \mathcal{C}^{\tilde{k},\tilde{\alpha}}  \mathcal{S}^m(T^*(\partial D))$. That is, if $x \in U_i$, $\mathbb{X}: U_i \rightarrow V_i \in \mathcal{C}^{k,\alpha} $, then 
\[
 [\mathrm{Op}_{a} (\phi) ] (x) : = \int_{\mathbb{R}^{d-1}} \int_{U_i} a \left(F_i^{-1} (\mathbb{X}(x),  \xi ) \right) e^{ i \langle x-y,\xi \rangle } \phi(\mathbb{X}(y)) \psi_i (\mathbb{X}(y) )\, 
d y\, d \xi \, ,
\]
which now uniquely defines an operator modulus $ \mathcal{C}^{\tilde{k}-1,\tilde{\alpha}}  \Phi \text{SO}^{m-1}$ (coming from a smooth change of coordinates).
Similarly the principal symbol $a_0(x,\xi)$ will be such that
\[
a(x,\xi) = a_0(x,\xi) ( \text{mod} \,  \mathcal{C}^{\tilde{k}-1,\tilde{\alpha}}  \Phi \text{SO}^{m-1} )  \quad \text{ and } \quad
\mathrm{Op}_{ a } = \mathrm{Op}_{ a_0 } (  \text{mod} \, \mathcal{C}^{\tilde{k}-1,\tilde{\alpha}}  \Phi \text{SO}^{m-1} ) \,.
\]
There are again definitions independent of the choice of the coordinate system. However we are only going to fix one particular coordinate system to compute the symbol expansion.
We would also like to remark that one will need to consider the H\"older-Zygmund space $C_{*}^{k + \alpha} $ \cite{Taylorbook,Pfeuffer} with the property that for $k \in \mathbb{N}$,
\[
C_{*}^{k + \alpha} = C^{k,\alpha}\ \text{ for }  0 < \alpha < 1\ \  \text{ and } \ \ C^{k} \subset C^{k-1,1} \subset C_{*}^{k},
\]
which makes it possible to introduce an algebra over the pseudo-differential operators $ C_{*}^{k + \alpha}  \Phi \text{SO}^{m} : C_{*}^{k + m + \alpha} \rightarrow C_{*}^{k + \alpha} $ for all $k,\alpha,m \in \mathbb{R}$ \cite{Taylorbook,Pfeuffer}. 
However, since we shall be mainly concerned with obtaining expansions of the symbols in this article, we retain minimum technicality and shall not define the H\"older-Zygmund norms, as it involves defining further technical objects including the Littlewood-Paley decomposition and projection.  The readers may refer to \cite{Taylorbook,Pfeuffer,Hor1,Hor2}) for more relevant details.
}

{\color{black}
With these preparations, we are ready to compute the principal symbol of $\mathcal{K}^*_{\partial D}$. {\color{black} First, let us assume $\partial D \in \mathcal{C}^{\infty}$.} We note that for a fixed $x\in \partial D$, if one takes the geodesic normal coordinate in a neighborhood of $x$ to give $v \in \text{Dom}(\exp_x) \subset  T_x( \partial D) \cong \mathbb{R}^{d-1} \mapsto \mathbb{X}(v) : = \exp_x (v) \in  \partial D$, then at the point $x$, we have $g_{ij}(x) = \delta_{ij}$ and $\Gamma_{ij}^k(x) = 0$ for $i,j,k = 1,...,d-1.$
Moreover, if we also choose on the ambient space the semi-geodesic normal coordinate in a neighborhood of $x$, which gives, in a neighbourhood of $x \in \partial D$ as $(a,s)$, $\tilde{\mathbb{X}}(a,s) = \exp_{x}(a) + \, s \, \nu(\exp_{x}(a))$, we have $\bar{\Gamma}^k_{ij}(x) = \Gamma_{ij}^k(x) =0 $ for $k\neq d$ and $\bar{\Gamma}^d_{ij}(x) = A_{ij}(x) $.
For $y = \exp_x( \delta \, \omega)$ with $| \, \omega \, | = 1$, we now have (cf. \cite{DoCarmo,Jost,Lee})
\beqn
y & =  & x +  \delta \omega - \frac{1}{2}   \delta^2 \langle A(x) \omega , \omega \rangle \,  \nu(x)  + O(\delta^3 )  \,, \label{expansion0} \\
\nu(y) &=& \nu(x) +  \delta  A(x) \omega + \frac{1}{2} \delta^2 \left[ \partial_{\omega} A(x) \omega  
- |A(x) \omega |^2 \nu(x) \right]    + O(\delta^3 ) \,, \label{expansion1} \\
\sqrt{ \det ( g(y)) }   &=&  1 + \frac{1}{6} \delta^2 \, \text{Ric}_{x}(\omega,\omega)  + O(\delta^3)  \label{expansion2} \,.
\eqn
The first two expansions come readily from our choice of $g_{ij}$ and the ambient metric $\bar{g}_{ij}$, as well as the definition of $A_{ij}$, $\nu$, $\bar{\nabla}$ and $\bar{\Gamma}_{ij}^k$ that
\beqnx
A_{ij}(x) &= & \langle \bar{\nabla}_{\mathbb{X}_i} \nu(x), \mathbb{X}_j \rangle = - \langle \bar{\nabla}_{\mathbb{X}_i} \mathbb{X}_j ,  \nu(x) \rangle \,, \quad 0 \; = \; \bar{\nabla}_{\mathbb{X}_i}  \langle  \nu(x) , \nu(x)\rangle =  2 \langle \bar{\nabla}_{\mathbb{X}_i} \nu (x), \nu(x)\rangle \,, \\
0 &=& \frac{1}{2} \bar{\nabla}_{\mathbb{X}_i} \bar{\nabla}_{\mathbb{X}_j} \langle \nu(x), \nu (x) \rangle 
=  \langle  \bar{\nabla}_{\mathbb{X}_i} \bar{\nabla}_{\mathbb{X}_j} \nu(x), \nu(x) \rangle +  \sum_{k=1}^{d-1} A_{ik}(x) A_{kj}(x) \,, \\
\partial_i A_{jk}(x) &= &\bar{\nabla}_{\mathbb{X}_i}  \langle \bar{\nabla}_{\mathbb{X}_j} \nu(x), \mathbb{X}_k \rangle 
=   \langle \bar{\nabla}_{\mathbb{X}_i}  \bar{\nabla}_{\mathbb{X}_j} \nu(x), \mathbb{X}_k \rangle 
\,.
\eqnx
The third expansion comes from Jacobi's formula and the Taylor expansion of the metric applied to a Jacobi field along the radial geodesic in the normal coordinate:\cite{Jost,Lee,Kobayash2}:
\begin{equation}
 g_{ij}(y)=\delta _{ij}- {\tfrac {1}{3}} \delta^2 \sum_{k,l=1}^{d-1} R_{ikjl}(x) \omega_{k} \omega_{l}+O\left(\delta^3\right) \,. 
\label{expansion_g}
\end{equation}
With the above expansions of $\nu$ and $\sqrt{|g|}$,} we readily obtain the following expansion for the kernel of $ \mathcal{K}^*_{\partial D} $,
\beqnx
\mathbb{K}_{h,0} (x,y) d \sigma(y)  &=& { \color{black} \frac{1}{2} } \delta^{- d +2 } \,  \langle A(x)  \, \omega, \, \omega \rangle  \, d y +  O(   \delta^{- d + 3  }  )  .
\eqnx
Following \cite{weyl2,weyl1} {\color{black} and using a Fourier transform as in \eqref{symbol}} of the kernel of $\mathbb{K}_{h,0} (x,y)$ with respect to $ v :=  \delta  \omega$, we obtain that the symbol around $x=y$ in the geodesic normal coordinate (noting that $g_{ij}(x)= \delta_{ij}$) is
\[
\begin{split}
& p_{\mathcal{K}^*_{\partial D}}(x,\xi)  : = { \color{black} \frac{1}{2} }  {\color{black} \mathcal{F}_{v} \left[  \langle A(x)  \, v, \, v \rangle \, |v|^{-d} \right] (\xi) }
+ O(|\xi|^{-2}) \\
=&   { \color{black} \frac{1}{2} }  \sum_{i,j=1}^{d-1}  A_{ij}(x) \, \mathcal{F}_{v} \left[ v_i v_j \, |v|^{-d} \right] (\xi)
+  O(|\xi|^{-2}) = { \color{black} - } { \color{black} \frac{1}{2} }  \sum_{i,j=1}^{d-1}  A_{ij}(x) \, \partial_{i} \partial_{j} \,  | \, \xi \, |
+  O(|\xi|^{-2}) \\
=&  { \color{black} - } { \color{black} \frac{1}{2} }  \sum_{i,j=1}^{d-1}  A_{ij}(x) \, \left(  \frac{\delta_{ij}}{|\xi|} -  \frac{\xi_i \xi_j}{|\xi|^3} \right)
+  O(|\xi|^{-2}) \\
=&  { \color{black} - }  { \color{black} \frac{1}{2} } (d-1) H(x) \,  |\xi|^{-1} { \color{black} + }  { \color{black} \frac{1}{2} }  \langle A(x)  \, \xi, \, \xi \rangle \, |\xi|^{-3}
+  O(|\xi|^{-2}) \,,
\end{split}
\]
{\color{black} where the leading term is non-zero if $d > 2$, and $O(|\xi|^{-2})$ signifies a symbol in $ S^{-2}$.}
Therefore $\mathcal{K}^*_{\partial D}$ is a pseudo-differential operator of order $-1$ on $\partial D \in \mathcal{C}^{\infty}$ and 
\[
\mathcal{K}^*_{\partial D} = { \color{black} - }  { \color{black} \frac{1}{2} }  \mathrm{Op}_{ (d-1) H(x) \,  |\xi|^{-1} -  \langle A(x)  \, \xi, \, \xi \rangle \, |\xi|^{-3} } (  \text{mod} \Phi \text{SO}^{-2} ).
\]

{ \color{black}
Next, considering that $\partial D \in \mathcal{C}^{k, \alpha}$, we have a general $g_{ij} \in \mathcal{C}^{k-1, \alpha}$. Through solving the Hamiltonian equation for the geodesic, the geodesic normal coordinate gives a parametrisation $\mathbb{X} \in \mathcal{C}^{k-1, \alpha} $. Therefore, via \eqref{regularity}, $\mathcal{K}^*_{\partial D}$ is a pseudo-differential operator of order $-1$ in the sense of $ \mathcal{C}^{k-3,\alpha}  \Phi \text{SO}^{-1} $, and\vspace*{-2mm}
\[
\mathcal{K}^*_{\partial D} = { \color{black} - }  { \color{black} \frac{1}{2} }  \mathrm{Op}_{ (d-1) H(x) \,  |\xi|^{-1} -  \langle A(x)  \, \xi, \, \xi \rangle \, |\xi|^{-3} } (  \text{mod} \,   \mathcal{C}^{k-4,\alpha}   \Phi \text{SO}^{-2} ) \,.
\]
}\vspace*{-2mm}

We summerize the above discussion in the following theorem, which generalizes the three-dimensional result in \cite{weyl2,weyl1}. 
\begin{Theorem} \label{first_theorem}
{\color{black}For $d > 2$,} {\color{black}  if $\partial D \in \mathcal{C}^{k, \alpha}$, $k\geq4$,} the operator $\mathcal{K}^*_{\partial D}$ is a pseudo-differential operator of order $-1$ in a sense of $ \mathcal{C}^{k-3,\alpha} \,  \Phi \text{SO}^{-1} $, with its symbol given as follows in the geodesic normal coordinate around each point $x$:
\beqnx
p_{\mathcal{K}^*_{\partial D}}(x,\xi) =  { \color{black} - } { \color{black} \frac{1}{2} }  (d-1) H(x) \,  |\xi|^{-1}  { \color{black} + } { \color{black} \frac{1}{2} }   \langle A(x)  \, \xi, \, \xi \rangle \, |\xi|^{-3} 
+  O(|\xi|^{-2}) \in \mathcal{C}^{k-3,\alpha}  S^{-1}  \, ,
\eqnx
where the big O represents a symbol in $ \mathcal{C}^{k-4,\alpha}  S^{-2} $.
\end{Theorem}

\noindent  A remark is that the above result holds also for $\mathcal{K}_{\partial D}$ instead of $\mathcal{K}^*_{\partial D}$.
We would also like to remark that if the geodesic normal coordinate is not chosen, and for a general coordinate, tracing back the above steps, we have instead
\beqnx
p_{\mathcal{K}^*_{\partial D}}(x,\xi) &=&  { \color{black} - } { \color{black} \frac{1}{2} }  (d-1) H(x) \,  |\xi|_{g(x)}^{-1}  { \color{black} + } { \color{black} \frac{1}{2} }   \langle A(x)  \, g^{-1}(x) \, \xi, \, g^{-1}(x) \,\xi \rangle \, |\xi|_{g(x)}^{-3}
+  O(|\xi|_{g(x)}^{-2}) \,.
\eqnx
where $|\xi|_{g(x)}^2 = \sum_{i,j=1}^{d-1} g^{ij}(x) \xi_i \xi_j$.
The above remark is in force to indicate that our choice of the geodesic normal coordinate is just for simplification of the resulting computation and is not a necessary move. {\color{black} This shows that the symbol class actually has one order higher regularity, and hence $\mathcal{K}^*_{\partial D} \in \mathcal{C}^{k-2,\alpha} \,  \Phi \text{SO}^{-1} $.  However we shall not need this sharp regularity result in our subsequent analysis.}

 { \color{black} Another remark is that the symbol can be loosely interpreted as given by $\frac{1}{2} \langle B(x) \, \xi, \, \xi \rangle/ |\xi|^{-2} $ where $B(x)$ is the trace-free part of the second fundamental form. } 

{ \color{black}
A last remark is that, if we further suppose strict convexity of $D$ and smoothness of $\partial D$, i.e. $- A (x) \geq R^{-1} g(x) $ for some $R > 0$, (it is noted here that our convention yields that $A$ is strictly negative instead of positive, see e.g. \eqref{expansion0}), then we have that the following condition given in \cite{ergodicity} holds:
\vskip 2mm

\noindent \textbf{Assumption (A)} 
We have $\langle A(x) \, g^{-1}(x) \, \omega \, ,\,  g^{-1}(x) \, \omega \rangle \neq (d-1) H(x)$ for all $x \in \partial D$ and $\omega \in \{ \xi :  |\xi |_{g(x)}^2 = 1 \} \subset T_x^*(\partial D)$.

\vskip 2mm

\noindent Indeed from $- A (x) \geq R^{-1} g(x)$ we have for all $ |\xi |_{g(x)}^2 = 1 $ that
\[
 (d-1) H(x) |\xi |_{g(x)}^2 - \langle A(x) \, g^{-1}(x) \, \xi \, ,\,  g^{-1}(x) \, \xi \rangle \leq 
 (d-2) R^{-1} |\xi |_{g(x)}^2 < 0.
\]
With these assumptions, it is shown in \cite{ergodicity,weyl2, weyl1} that the Weyl asymptotics (cf. \cite{Hor1,Hor2,100years,Taylor}) holds:  
\beqnx
 \sum_{r h \leq \lambda_i \leq s h}  1  =  (2 \pi h )^{ \frac{1-d}{2} }  \left|  \left \{ r \leq { \color{black} { \color{black} - }  \frac{1}{2} }  (d-1) H(x) \,  |\xi|^{-1} { \color{black} + }  { \color{black} \frac{1}{2} }  \langle A(x)  \, \xi, \, \xi \rangle \, |\xi|^{-3} \leq  s \right \} \right|+ o_{r,s}(h^{\frac{1-d}{2} }) \,m \,,
\eqnx
{\color{black} where $\lambda_i$ is the $i$-th eigenvalue of the operator $\mathcal{K}^*_{\partial D}$.}
Therefore we can infer that $\mathcal{K}^*_{\partial D}$ is a compact operator of Schatten $p$ class $S_p$ for $p > d-1$ and $d > 2$. Here, an operator $K$ is of Schatten $p$ class $S_p$ if 
\[
\| K \|_{S_p}^p := \sum_{\sigma_i \in \sigma(K)} \sigma_i^p < \infty \,,
\]
where $\sigma(K)$ denotes the set of singular values of $K$.
Since it is not the main focus of our study, we refer the readers to \cite{ergodicity,weyl2, weyl1} for more relevant details. Nevertheless, we point out that Theorem \ref{first_theorem} holds without the convexity assumption or Assumption (A).

}

In a similar manner, we can obtain:
\begin{Theorem}
{\color{black}For $d > 2$,} let $\partial D \in \mathcal{C}^{{\color{black} k },\alpha}$ and $h \in  \mathcal{C}^{{\color{black} k -1 },\alpha}$ for $k \geq 4$. Then $\mathcal{K}^{(1)}_{D,h}$ in \eqref{eq:kernel1} can be decomposed as
\begin{equation}\label{eq:decom1}
\mathcal{K}^{(1)}_{D,h}= \mathcal{K}^{(1)}_{D,h,0}  + \mathcal{K}^{(1)}_{D,h,-1} + \mathcal{K}^{(1)}_{D,h,-2} \,,  
\end{equation}
where $\mathcal{K}^{(1)}_{D,h,0}$, $\mathcal{K}^{(1)}_{D,h,-1}$ and $\mathcal{K}^{(1)}_{D,h,-2}$ are pseudodifferential operators of order $0,-1,-2$ respectively on $\partial D$ {\color{black} (in a sense of 
$ \mathcal{C}^{k-2,\alpha}  \Phi \text{SO}^{0} $, $ \mathcal{C}^{k-3,\alpha}  \Phi \text{SO}^{-1} $ and $ \mathcal{C}^{k-4,\alpha}  \Phi \text{SO}^{-2} $ respectively)}
with their symbols given as follows in the geodesic normal coordinate around each point $x$:
\beqnx
p_{ \mathcal{K}^{(1)}_{D,h,0} }(x,\xi)  &=&  -  { \color{black} 2 i }  \partial_\xi  h(x) | \, \xi \, |^{-1}  = O(1) \in  \mathcal{C}^{k-2,\alpha}  S^{0}  \,, \\
p_{ \mathcal{K}^{(1)}_{D,h,-1} }(x,\xi)  &=&   { \color{black} - \frac{1}{2} }  h(x) \bigg\{  { \color{black}  (d-1)^2 (d + 1) }    | H(x) |^2 \,  | \, \xi \, |^{-1}  { \color{black} - (2 d -1) (d-1) }  H(x)  \langle A (x) \xi \,, \, \xi \rangle \, |\, \xi \, |^{-3} \\
& & 
+ { \color{black}  ( 2d -1) }    | A(x) |_F^2  \, |\, \xi \, |^{-1} +  { \color{black} 4 d }  |  A(x) \xi |^2  \, |\, \xi \, |^{-3}  - { \color{black}  12 d }   \,  | \, \langle A (x) \xi \,, \, \xi \rangle  \, |^2  \, | \,\xi \,|^{-5} \bigg\} \\
& & 
+  { \color{black}  \frac{1}{2} } \, \Delta h(x) \,| \, \xi \, |^{-1}  
 -  \frac{1}{2}  \text{Hess}_{\xi , \xi }  h(x)  \, | \, \xi \,|^{-3}   \\
&=& O(|\xi|^{-1}) \in \mathcal{C}^{k-3,\alpha}  S^{-1} \,, \\
p_{ \mathcal{K}^{(1)}_{D,h,-2} }(x,\xi)  &=&    O(|\xi|^{-2}) \in \mathcal{C}^{k-4,\alpha}  S^{-2} \, ,
\eqnx
{\color{black} where $\partial_\xi  h := \sum_{i=1}^{d-1} \xi_i \partial_i h $, $\text{Hess}_{\xi,\xi} h = \sum_{i,j=1}^{d-1} \xi_i \xi_j \nabla_j \partial_i h$}, and $| \,\cdot \, |_F$ is the Frobenius norm of the matrix and the constant of the big O depends on $\| h \|_{\mathcal{C}^3} $. 
\end{Theorem}

\begin{proof}
{\color{black}
{\color{black} We first let $\partial D \in \mathcal{C}^{\infty}$.} Then via a straightforward calculation and substitution of 
\eqref{expansion1}-\eqref{expansion2} and
\[
h(y) = h(x) + \delta \partial_\omega h(x) + \frac{1}{2} \delta^2 \text{Hess}_{\omega , \omega } h(x)   + O(\delta^3) 
\]
}into \eqref{eq:ker1}, we can obtain the kernel of $\mathcal{K}^{(1)}_{D,h}$ in the geodesic normal coordinate as follows,
\beqnx
&  & \mathbb{K}_{h,1} (x,y)  d \sigma(y)  \\
&=&
  -  { \color{black} 2 } \, \delta^{- d + 1}   \, \partial_{\omega} h(x)   \, d y    \\
& &
 + \, \delta^{- d +2 } \, \bigg(  h(x)  \, \left [ { \color{black} \frac{1}{2} }(d-1)  H(x) \,  \langle A(x)  \, \omega, \, \omega \rangle    { \color{black} -\frac{1}{2} }  d  | \, \langle A(x)  \, \omega, \, \omega \rangle \, |^2  { \color{black} - }  \frac{1}{2} \, |\, A(x) \,\omega \, |^2  \right] \\
 & &  -  \frac{1}{2} \text{Hess}_{\omega , \omega } h(x)   \bigg) \,  d y + \, O(  \delta^{- d + 3  }  )  ,
\eqnx
where the constant of big $O$ now depends on $||h||_{\mathcal{C}^3}$. Using the Fourier transform of the kernel $\mathbb{K}_{h,1} (x,y)$  with respect to $ v :=  \delta  \omega$, $\omega\in\mathbb{S}^{d-1}$, 
we obtain the symbol around $x=y$ in the geodesic normal coordinate (and noticing $\nabla_v v = 0$) that 
\beqnx
&     &   p_{\mathcal{K}^{(1)}_{D,h}}(x,\xi) \\
& : =& - { \color{black} 2 } \mathcal{F}_{v} \left[ 
  \, |v|^{- d }   \, \partial_{v} h(x)  \right] (\xi) \\
& & 
 + { \color{black} \frac{1}{2} }  \mathcal{F}_{v} \bigg[   \, |v|^{- d -2 } \, \bigg(  h(x)  \, \big [ (d-1)  H(x) \,  \langle A(x)  \, v, \, v \rangle \, |v|^{2}    -  d  | \, \langle A(x)  \, v, \, v \rangle \, |^2 { \color{black} - }  \, |\, A(x) \,v \, |^2 |v|^{2 }  \big]\\
 & &  -   \text{Hess}_{v , v } h(x) |v|^{2}   \bigg) 
 \bigg] (\xi) + O(|\xi|^{-2})\\
& =& -  { \color{black} 2 }   \sum_{i = 1} ^{d-1} \partial_i  h(x)  \, \mathcal{F}_{v} \left[ 
  v_i\, |v|^{- d }    \,  \right] (\xi) \\
& & 
 + { \color{black} \frac{1}{2} }  \sum_{i,j = 1} ^{d-1}  \left( (d-1)   h(x)  H(x) \,  A_{ij} (x)   { \color{black} - }  h(x)   \sum_{k=1}^{d-1} A_{ik}(x) A_{kj}  -   \partial_i \partial_j h(x)  \right)   \, \mathcal{F}_{v} \left[  v_i v_j  |v|^{- d }  \right] (\xi) \\
& &
- { \color{black}  \frac{1}{2} } d  \,  h(x)  \sum_{i,j,k,l = 1} ^{d-1} A_{ij}(x) A_{kl}(x) \,  \mathcal{F}_{v} \left[   v_i v_j  v_k v_l \,  |v|^{- d -2 }  \right] (\xi) \, + O(|\xi|^{-2})  \\
& =& -   { \color{black} 2 i }  \sum_{i = 1} ^{d-1} \partial_i  h(x)  \, \partial_{i} \, | \, \xi \, | 
  {  \color{black} - \frac{1}{2} }   \sum_{i,j = 1} ^{d-1}  \bigg( (d-1)   h(x)  H(x) \,  A_{ij} (x) { \color{black} -}   h(x)  \sum_{k=1}^{d-1} A_{ik}(x) A_{kj} \\
 & & -   \partial_i \partial_j h(x)  \bigg)   \, \partial_{i} \partial_{j} \, | \, \xi \, | 
 - {\color{black}  \frac{1}{2} } d  \,  h(x)  \sum_{i,j,k,l = 1} ^{d-1} A_{ij}(x) A_{kl}(x) \, \partial_{i} \partial_{j} \partial_k \partial_l \, | \, \xi \, |^3 + O(|\xi|^{-2}) \\
& =& -   { \color{black} 2 i }  \sum_{i = 1} ^{d-1} \partial_i  h(x)  \, \xi_i \, | \, \xi \, |^{-1}  \\
&  &  {  \color{black} - \frac{1}{2} }   \sum_{i,j = 1} ^{d-1}  \bigg( (d-1)   h(x)  H(x) \,  A_{ij} (x) { \color{black} -}   h(x)  \sum_{k=1}^{d-1} A_{ik}(x) A_{kj} -   \partial_i \partial_j h(x)  \bigg)  \left(\delta_{ij} | \, \xi \, |^{-1}  - \xi_i \xi_j | \, \xi \, |^{-3}  \right) \,  \\
& & - {\color{black}  \frac{1}{2} } d  \,  h(x)  \sum_{i,j,k,l = 1} ^{d-1} A_{ij}(x) A_{kl}(x) \,  \\
&   & \quad  \bigg( (\delta_{ij} \delta_{kl} + \delta_{ik} \delta_{jl} + \delta_{il} \delta_{jk} ) | \, \xi \, |^{-1} \\
&   & \quad  \quad   -  \left( \xi_i \xi_j  \delta_{kl}  + \xi_i \xi_k  \delta_{jl} + \xi_i \xi_l  \delta_{jk} +\xi_j \xi_k  \delta_{il}  + \xi_j \xi_l  \delta_{ik} + \xi_k \xi_l  \delta_{ij} \right)  | \, \xi \, |^{-3}  \\
&   &  \quad \quad -12 \, \xi_i\xi_j\xi_k\xi_l   | \, \xi \, |^{-5}  \bigg)
 + O(|\xi|^{-2}) \,.
\eqnx
Therefore, we have
\begin{eqnarray*}
    p_{\mathcal{K}^{(1)}_{D,h}}(x,\xi) 
& =& -   { \color{black} 2 i }  \partial_\xi  h(x) | \, \xi \, |^{-1} \\
& & 
{ \color{black} - \frac{1}{2} }  h(x) \bigg\{  { \color{black}  (d-1)^2 (d + 1) }    | H(x) |^2 \,  | \, \xi \, |^{-1}  { \color{black} - (2 d -1) (d-1) }  H(x)  \langle A (x) \xi \,, \, \xi \rangle \, |\, \xi \, |^{-3} \\
& & 
+ { \color{black}  ( 2d -1) }    | A(x) |_F^2  \, |\, \xi \, |^{-1} +  { \color{black} 4 d }  |  A(x) \xi |^2  \, |\, \xi \, |^{-3}  - { \color{black}  12 d }   \,  | \, \langle A (x) \xi \,, \, \xi \rangle  \, |^2  \, | \,\xi \,|^{-5} \bigg\} \\
& & 
+  { \color{black}  \frac{1}{2} } \, \Delta h(x) \,| \, \xi \, |^{-1}  
 -  \frac{1}{2}  \text{Hess}_{\xi , \xi }  h(x)  \, | \, \xi \,|^{-3}  \\
& & + O(|\xi|^{-2}) ,
\end{eqnarray*}
{\color{black}where $O(|\xi|^{-2})$ signifies a symbol in $ S^{-2}$. }

{ \color{black}
Next if $\partial D \in \mathcal{C}^{k, \alpha}$, considering the fact that the geodesic normal coordinate gives a parametrisation $\mathbb{X} \in \mathcal{C}^{k-1, \alpha} $ and \eqref{regularity}, we trace back the regularity in each term and readily obtain \eqref{eq:decom1}. 
}

The proof is complete. 

\end{proof}

{\color{black}
We would like to point out that some of the above expressions can be simplified midway via the Gauss-Codazzi formula \cite{Jost,Kobayash2} and the fact that the ambient space $\mathbb{R}^d$ is flat.  In fact, taking the trace of the Gauss-Codazzi formula, we have for $| \, \omega \,| = 1$
\[
\text{Ric} (\omega, \omega) =  (d-1) H(x) \,  \langle A(x)  \, \omega, \, \omega\rangle - \, |\, A(x) \omega \, |^2   \,,
\]
and this helps simplify for instance
\beqnx
&  & \mathbb{K}_{h,1} (x,y)  d \sigma(y)  \\
&=&
  -  { \color{black} 2 } \, \delta^{- d + 1}   \, \partial_{\omega} h(x)   \, d y \\  
& &  + \, \delta^{- d +2 } \, \bigg(  h(x)  \, \left [ { \color{black} \frac{1}{2} } \text{Ric} (\omega, \omega)    { \color{black} -\frac{1}{2} }  d  | \, \langle A(x)  \, \omega, \, \omega \rangle \, |^2  \right]  -  \frac{1}{2} \text{Hess}_{\omega , \omega } h(x)   \bigg) \,  d y + \, O(  \delta^{- d + 3  }  )  \,.
\eqnx
However we do not intend to go further along this direction, as it will yield an expression of the symbol with additional terms, e.g. scalar curvature and Ricci curvature, which will make the resulting symbol more cumbersome.
}

\subsection{A property of the generalized polarisation tensors} \label{section_D_N}

By \eqref{eq:a1}, in order to analyse the quantitative behaviour of the scattered field, one needs to analyse the operator $ (\lambda I - \mathcal{K}_{\partial D}^*)^{-1} \circ   \partial_\nu $. In this subsection, 
we analyse the symbol of the operator $ (\lambda I - \mathcal{K}_{\partial D}^*)^{-1} \circ   \partial_\nu $ in terms of the GPT $\mathcal{M}_{L,M} (\lambda, D) $. {\color{black} In what follows, we use the notation $r\omega\in\mathbb{R}^d$ with $r$ and $\omega$ respectively signifying the radial and angular variables}. We have the following lemma for the subsequent use. 
\begin{Lemma} \label{haha1}
{ \color{black} For $d > 2$,} let $\partial D \in \mathcal{C}^{{\color{black} k },\alpha}$ for $k \geq 4$. 
The GPT $\mathcal{M}_{L,M} (\lambda, \partial D) $ in \eqref{eq:gpt2} has the following representation
\begin{equation}\label{eq:s0}
\mathcal{M}_{L,M} (\lambda, \partial D) =
\left \langle   |r|^{k}  Y_{L} (\omega)  \,,\,
\left( P_{D,1} + P_{D,0} + P_{D,-1}  \right)
\left(  |r|^{n}  Y_{M} (\omega)  \right)  \right\rangle_{L^2 (\partial D, d \sigma )}, 
\end{equation}
where $P_{D,m}$ are pseudo-differential operators of order $m$ for $m = 1,0,-1$, 
{\color{black} (in a sense of 
$ \mathcal{C}^{k-2,\alpha}  \Phi \text{SO}^{1} $, $ \mathcal{C}^{k-3,\alpha}  \Phi \text{SO}^{0} $ and $ \mathcal{C}^{k-4,\alpha}  \Phi \text{SO}^{-1} $ respectively)}
and in the geodesic normal coordinate around each point $x$, it holds that
\begin{equation}\label{eq:dd1}
\begin{split}
p_{P_{D,1} }(x,\xi) 
=& \lambda^{-1}  |\xi| \in \mathcal{C}^{k-1,\alpha} S^{1}, \\
p_{P_{D,0} }(x,\xi) =& 
{ \color{black} - \frac{1}{2} } \lambda^{-1}  \left(   \lambda^{-1}  { \color{black} + 1 }  \right)  \bigg( (d-1) H(x)  -  \langle A(x)  \, \xi, \, \xi \rangle \, |\xi|^{-2}  \bigg) \in \mathcal{C}^{k-3,\alpha} S^{0} , \\
p_{P_{D,-1} }(x,\xi) =&   O\left( \lambda^{-1}  |\xi|^{-1}\right)  \in \mathcal{C}^{k-4,\alpha} S^{-1}  \, .
\end{split}
\end{equation}
\end{Lemma}

\begin{proof}

{\color{black} First, let us assume $\partial D \in \mathcal{C}^{\infty}$.}
{\color{black}
Since $\mathcal{K}^*_{\partial D} \in \Phi \text{SO}^{-1}$, we have the following Neumann series
\[
(\lambda I - \mathcal{K}^*_{\partial D} )^{-1} = \sum_{l=0}^{\infty} [ \mathcal{K}^*_{\partial D} ]^l \lambda^{- l -1}  (  \text{mod} \, \Phi \text{SO}^{-\infty} ) = \lambda^{-1} + \lambda^{-2} \mathcal{K}^*_{\partial D} \,  (  \text{mod} \, \Phi \text{SO}^{-2} ).
\]
The series converges in norm when $\lambda > \| \mathcal{K}^*_{\partial D}  \| = \frac{1}{2}$ and converges only in the sense of $\text{mod} \, \Phi \text{SO}^{-\infty} $ otherwise (i.e.  $ \text{mod} \, \Phi \text{SO}^{-m} $ for all $m \in \mathbb{N}$ ).} Therefore, we render the symbol of the operator $(\lambda I - \mathcal{K}^*_{\partial D} )^{-1}$ as follows
\begin{equation}\label{eq:ar1}
p_{(\lambda I - \mathcal{K}^*_{\partial D} )^{-1}}(x,\xi) =  \lambda^{-1}  { \color{black} - } { \color{black} \frac{1}{2} }   \lambda^{-2} (d-1) H(x) \,  |\xi|^{-1}  { \color{black} + } { \color{black} \frac{1}{2} }   \lambda^{-2}  \langle A(x)  \, \xi, \, \xi \rangle \, |\xi|^{-3}
+  O\left( \lambda^{-2}  |\xi|^{-2}\right) \, ,
\end{equation}
where {\color{black}$O( \lambda^{-2} |\xi|^{-2})$ signifies a symbol in $S^{-2}$ with a constant of order at least $\lambda^{-2} $.}

Next, it is noticed that the function $u_0(x)= r^{k} \,  Y_{L} (\omega)$ satisfies  $\Delta u_0 = 0$, where $(r,\omega)\in\mathbb{R}_+\times\mathbb{S}^{d-1}$ is the spherical coordinate of $x\in\mathbb{R}^d$. Hence, the map $\Lambda_0: r^{k} \,  Y_{L} (\omega) \mapsto \partial_\nu  r^{k} \,  Y_{L} (\omega) $ is in fact a Dirichlet-to-Neumann (DtN) map associated with the Laplacian.   {\color{black} In fact, we have that (cf. \cite{Jost,DoCarmo})
{\color{black}
\beqn
\Delta =  \partial_{\nu}^2 + (d-1) H(x) \partial_{\nu} + \Delta_{\partial D } \,, \label{decompsition} 
\eqn
}
With this decomposition, in Theorem 3.1 in \cite{uhlmann}, it is shown that the Laplacian can be factorised into a product of two operators modulo a smoothing operator
\beqn
- \Delta = \left(i \partial_\nu + i E(x) + i B_0(x, i \nabla ) \right) \left(i \partial_\nu - i B_0(x, i  \nabla ) \right), \label{decompsition2} 
\eqn
and when $\partial D \in \mathcal{C}^{\infty}$, and
\[
\Lambda_{0} = B_0(x, i  \nabla ) |_{\partial D}  \, (  \text{mod} \, \Phi \text{SO}^{-\infty} ),
\]
where it is recalled that $ \text{mod} \, \Phi \text{SO}^{-\infty} $ means $ \text{mod} \, \Phi \text{SO}^{-m} $ for all $m \in \mathbb{N}$, and via the recursive formula (3.11)-(3.14) in \cite{uhlmann}, we can see the principle symbol of $B(x, i  \nabla ) $ is  $|\xi|$, which is of order $1$ with respect to $\xi$.  Hence the Dirichlet-to-Neumann (DtN) map is a pseudo-differential operator of order $1$.} 
Again in the sequel, we choose a semi-geodesic normal coordinate on a neighbourhood of $x \in \partial D$ as $(a,s)$ with $\tilde{\mathbb{X}}(a,s) = \exp_{x}(a) + \, s \, \nu(\exp_{x}(a))$.

We proceed to compute the symbol of the composition operator, $(\lambda I - \mathcal{K}^*_{\partial D} )^{-1} \circ \Lambda_0$. 
Under our choice of coordinates, one has, for $1 \leq i,j \leq d-1$
\beqnx
 \partial_{\varepsilon} \mid_{\eps = 0} \tilde{g}_{ij} (x+ \eps \nu(x)) =  \partial_{\varepsilon} \mid_{\eps = 0}  \langle \tilde{\mathbb{X}}(e_i,\eps) ,  \tilde{\mathbb{X}}(e_j,\eps) \rangle = { \color{black} 2} A_{ij}(x) \, . 
\eqnx
We may now take the factorization \eqref{decompsition2} of \eqref{decompsition} and use the recursive formula (3.11)-(3.14) in \cite{uhlmann} (keeping in mind that the second-order derivatives of $g$ do not vanish) to obtain that the symbol of $\Lambda_0$ and its derivative with respect to $x$ at the point $x$ are given by, 
\begin{equation}\label{eq:s1}
\begin{split}
p_{\Lambda_0}(x,\xi) 
=&
 |\xi| 
+ \frac{1}{2}  \langle A(x) \xi, \xi \rangle \,   |\xi|^{-2}  
-  \frac{d-1}{2}   H(x)  \\
 &
-  \frac{1}{4}    \langle \partial_\nu A(x)  \xi,  \xi \rangle \,   |\xi|^{-3}
 { \color{black}   + } \frac{1}{2}  | A(x)  \xi |^2  \, |\xi|^{-3} 
  + \frac{1}{4}  | \langle A(x)  \xi,  \xi \rangle |^2 \,   |\xi|^{-5}   \\
 & 
+  \frac{d-1}{4}  \partial_\nu H(x)  \, |\xi|^{-1}  
{\color{black} + }   \frac{1}{4}  (d-1) H(x)  \langle A(x) \xi, \xi \rangle \,   |\xi|^{-3}  \\
 &
+  \frac{(d-1)^2}{4}   | H(x) |^2  \, |\xi|^{-1} 
-  \frac{1}{2}  \langle \partial_\xi  A(x) \xi, \xi \rangle \,   |\xi|^{-4} 
-  \frac{d-1}{4}   \partial_\xi  H(x) \, |\xi|^{-1} \\
&{ \color{black} + \frac{1}{12}   \text{Ric}_x (\xi,\xi) |\xi|^{-3}  }
 +  O\left( |\xi|^{-2}\right) \, .
 \end{split}
 \end{equation}
Here, the last term in \eqref{eq:s1} is obtained via \eqref{expansion_g} that at $x$, we have $ \langle \Delta g(x) \xi ,\xi \rangle  = - \frac{1}{3} \text{Ric}_x (\xi,\xi) $ and 
$\langle \text{Hess}_{\xi,\xi} g(x) \xi , \xi \rangle = - \frac{1}{3} \sum_{i,j,k,l = 1}^{d-1}R_{ijkl}(x) \xi_i\xi_j\xi_k\xi_l = 0$ by the anti-symmetry, as well as
 \begin{equation}\label{eq:s2}
 \begin{split}
\frac{\partial}{\partial x_l} p_{\Lambda_0}(x,\xi) =&    \frac{1}{2}  \langle  \partial_l A(x) \xi, \xi \rangle \,   |\xi|^{-2} -  \frac{d-1}{2} \partial_l H(x) { \color{black} + \frac{1}{2} \sum_{i,k=1}^{d-1}  \partial_l \Gamma^{k}_{ii}(x) \xi_k |\xi|^{-1} }
+  O\left( |\xi|^{-1}\right) \, ,
\end{split}
\end{equation}
where we again make use of the fact that $ \langle \text{Hess}_{(\cdot),\xi} g(x) \xi , \xi \rangle =  - \frac{1}{3}  \sum_{i,j,k = 1}^{d-1} R_{ikjl} (x) \xi_i \xi_j \xi_k = 0$ by the anti-symmetry {\color{black}and $O(|\xi|^{-1})$ signifies a symbol in $S^{-1}$.}
It is worth noticing that in \eqref{eq:s1}, the symbol $p_{\Lambda_0}(x,\xi)$ is computed with one more term for our later use in the next section. By combining \eqref{eq:s1} and
\eqref{eq:s2}, we have 
\[
\begin{split}
& p_{(\lambda I - \mathcal{K}^*_{\partial D} )^{-1} \circ \Lambda_0 }(x,\xi) \\
=&  
p_{(\lambda I - \mathcal{K}^*_{\partial D} )^{-1}}(x,\xi) p_{\Lambda_0}(x , \xi)  
+ \frac{\partial }{\partial \xi} p_{(\lambda I - \mathcal{K}^*_{\partial D} )^{-1}}(x,\xi)  \frac{\partial }{\partial x} p_{ \Lambda_0}(x , \xi)  +  O\left( \lambda^{-1}  |\xi|^{-1}\right)  \\
=& 
 \lambda^{-1}  |\xi|  { \color{black} - \frac{1}{2} } \lambda^{-1}  \left(   \lambda^{-1}  { \color{black} + 1 }  \right) \bigg( (d-1) H(x)  -  \langle A(x)  \, \xi, \, \xi \rangle \, |\xi|^{-2}  \bigg)
+  O\left( \lambda^{-1}  |\xi|^{-1}\right),  
\end{split}
\]
{\color{black}where $O( \lambda^{-1} |\xi|^{-1})$ signifies a symbol {\color{black} $R(\lambda, x,\xi)$ such that fixing $\lambda$, $   R(\lambda, x,\xi) \in S^{-1}$, and $ | R(\lambda, x,\xi) | \leq  \lambda^{-1} \tilde{R}(\lambda, x,\xi)  $ for some $\tilde{R} \in S^{-1}$.  (From now on, we refers the symbol $R(\lambda, x,\xi)$ in this situation as ``a symbol in $S^{-1}$ with a constant of order at least $\lambda^{-1}$". In the rest of this work, a similar wording will be applied to describe the respective situation when we have a possibly different symbol order $S^{m}$ and order of $\lambda$ as $\lambda^{n}$. whenever no confusion will arise.)  }    }
Together with \eqref{eq:gpt2}, we readily have \eqref{eq:s0}--\eqref{eq:dd1} with the specification of H\"older regularity dropped.

{ \color{black}
Now if $\partial D \in \mathcal{C}^{k, \alpha}$, with the geodesic normal coordinate parametrisation $\mathbb{X} \in \mathcal{C}^{k-1, \alpha} $ and \eqref{regularity}, we trace back the regularity in each term to have \eqref{eq:s0}--\eqref{eq:dd1} with additional specification of the H\"older regularity. 
}

The proof is complete.

\end{proof}

{\color{black}
A remark on the regularity specification in the above Lemma is the fact that the symbol expansion is made with a coordinate parametrisation of $\mathbb{X} \in \mathcal{C}^{k-1, \alpha} $ and the symbol of $P_{D,1} $ does not contain a non-smooth term.  All the regularity specification can be increased by one upon careful inspection, but we do not intend to go further along that direction.
}

\subsection{Sensitivity analysis of the generalized polarisation tensor and the scattered potential field} \label{section_pert_D_N}

In this subsection, we compute the shape derivative of the generalized polarisation tensor $\mathcal{M}_{L,M} (\lambda, D)$,  $L\in I_k , M \in I_n $, which fully accounts for the shape derivative of the scattered field $(u - u_0)(x)$ associated with $D$.  From that, we can analyse the semi-classical symbols of the operators involved in the sensitivity of $\mathcal{M}_{L,M} (\lambda, D)$. This is of crucial importance to understand how the (local) sensitivity of the scattered field behaves under the influence of the mean curvature in the next subsection.

The main result of this subsection is contained in the following theorem. 

\begin{Theorem} \label{lemmaDM}
{\color{black}For $d > 2$,} let $\partial D \in \mathcal{C}^{{\color{black} k },\alpha}$ and $h \in  \mathcal{C}^{{\color{black} k -1 },\alpha}$ for $k \geq 4$. 
Let $N \in \mathbb{N}$, there exists a positive constant $C$ depending only on $N$, $L\in I_{k}, M \in I_n$, 
 $||\mathbb{X}||_{\mathcal{C}^2}$ and $||h||_{\mathcal{C}^1}$  such that 
\beqn
\left | \mathcal{M}_{L, M} (\lambda, D_{\varepsilon})  - \mathcal{M}_{L,M} (\lambda, D) 
- \sum_{n=1}^N \varepsilon^n  
\mathcal{M}^{(n)}_{L,M} (\lambda, D , h) \right|
   \leq C \varepsilon^{N+1},
   \label{seriesvariation2}
\eqn
for some $ \mathcal{M}^{(n)}_{L,M} (\lambda, D , h) $, with $\mathcal{M}^{(1)}_{L,M} (\lambda, D , h) $ given as
\beqnx
\mathcal{M}^{(1)}_{L,M} (\lambda, D , h) =  \left \langle   |r|^{k}  Y_{L} (\omega)  \,,\,
 Q_{D,h}  \left(  |r|^{n}  Y_{M} (\omega)  \right)  \right\rangle_{L^2 (\partial D, d \sigma )}, 
\eqnx
where
\begin{equation}\label{eq:t1}
 Q_{D,h} = Q_{D,h,1,I} + Q_{D,h,1,II} + Q_{D,h,0}  ,
\end{equation}
with $Q_{D,h,1,I}, Q_{D,h,1,II}$ being pseudo-differential operators of order $1$ and $Q_{D,h,0}$ being of order $0$,
{\color{black} (in a sense of 
$ \mathcal{C}^{k-1,\alpha}  \Phi \text{SO}^{1} $, $ \mathcal{C}^{k-3,\alpha}  \Phi \text{SO}^{1} $ and $ \mathcal{C}^{k-4,\alpha}  \Phi \text{SO}^{0} $ respectively)}
 and that in normal coordinate around each point $x$, 
\beqnx
p_{Q_{D,h,1,I}}(x,\xi)  &=& 
{ \color{black} 2 i } \lambda^{-2} \partial_\xi  h(x)  = O( \lambda^{-2} |\xi|) \in \mathcal{C}^{k-1,\alpha}  S^{1} ,\\
p_{Q_{D,h,1,II}}(x,\xi)  &=& 
 - \lambda^{-1} \bigg( (d-1)  h(x)  H(x)  | \xi |  - h(x)  \langle A(x) \xi, \xi \rangle \,   |\xi|^{-1}  \bigg)   = O( \lambda^{-1} |\xi |) \in \mathcal{C}^{k-3,\alpha}  S^{1} ,\\
p_{Q_{D,h,0}}(x,\xi)  &=&  O\left( \lambda^{-1} \right) \in \mathcal{C}^{k-4,\alpha}  S^{0} \,.
\eqnx
where the constant of the big O also depends on $\| \mathbb{X} \|_{\mathcal{C}^2} $ and $\| h \|_{\mathcal{C}^1} $.
\end{Theorem}

\begin{proof}

Consider a point $y \in \partial D$ given by $ y = \mathbb{X} (b)$ for some $ b \in U$. We have
$  \mathbb{X}^{\varepsilon} (b)  = y + h(y) \nu(y)$.
By using \eqref{Taylor1}, the decomposition \eqref{decompsition}, the understanding that $\partial_\nu$ acting on $ r^{k} \,  Y_{L} (\omega)$ is in fact a DtN map $\Lambda_0$ which is self-adjoint on $\langle \,\cdot\,,\,\cdot\, \rangle_{\frac{1}{2}, -\frac{1}{2}}$ coupling {\color{black} (where for any $\phi \in H^{1/2}(\partial D, d \sigma), \psi \in H^{-1/2} (\partial D, d \sigma) $, $\langle \, \phi \, , \, \psi \, \rangle_{\frac{1}{2}, -\frac{1}{2}}  := \psi (\phi) $ where we notice $\psi$ is a bounded linear functional on $H^{1/2}(\partial D, d \sigma)$,) }
and that $\Delta \left( r^{k} \,  Y_{L} (\omega) \right) = 0$, we can deduce that
\beqn
& & \mathcal{M}_{L,M} (\lambda, D_{\varepsilon})  - \mathcal{M}_{L,M} (\lambda, D)  \notag \\
&=&  \frac{\varepsilon}{c_{d,k}}     |R|^{ d - 2 + k }  \int_{\mathbb{S}^{d-1}}  \overline{ Y_L(\omega_x) }  \, \frac{\delta}{ \delta h} \left[ \bigg(u -  r^{n} Y_M(\omega) \bigg) (R \, \omega_x ) \right] (h)  d \omega_{x} + O(\varepsilon^2) \notag \\
&=& \varepsilon \int_{\partial D}   |y|^{k} \, \overline{ Y_{L} (\omega_y) }    \,
\left\{ h \circ \Lambda_0 \circ (\lambda I - \mathcal{K}_{\partial D}^*)^{-1} \circ \Lambda_0 \left [   r^{n} \, Y_{M} (\omega)  \right ] \right\} (y)  \, d \sigma(y)  \notag \\
& & -  \varepsilon \int_{\partial D}  |y|^{k} \, \overline{ Y_{L} (\omega_y) }    \,
\left( \left\{ (\lambda I - \mathcal{K}_{\partial D}^*)^{-1}  \circ \mathcal{K}^{(1)}_{D,h}  \circ (\lambda I - \mathcal{K}_{\partial D}^*)^{-1} \circ \Lambda_0 \right\} \left [ r^{n} \, Y_{M} (\omega)    \right ] \right ) (y) \, d \sigma(y)  \notag \\
& & +  \varepsilon (d-1)  \int_{\partial D}    |y|^{k} \, \overline{ Y_{L} (\omega_y) }     \,
\left( \left\{ \left[  h  \, H \, ,  (\lambda I - \mathcal{K}_{\partial D}^*)^{-1}  \right]  \circ \Lambda_0 \right\} \left( r^{n} \, Y_{M} (\omega)  \right)   \right) (y)  \,  \,  d \sigma(y)  \notag \\
& & -  \varepsilon \int_{\partial D} |y|^{k} \, \overline{ Y_{L} (\omega_y) }   \,
\left\{ \left\{ (\lambda I - \mathcal{K}_{\partial D}^*)^{-1} \circ  h \circ  \Delta_{\partial D }  \right \} \left[ r^{n} \, Y_{M} (\omega) \right] \right\} (y) \, d \sigma(y)  \notag \\
& & + O(\varepsilon^2)  \, , \label{variationvariation}
\eqn
where $[A,B]$ is the commutator of $A$ and $B$. Here and also in what follows, when a function is written as an operator, it signfies the multiplicative operator as multiplication by the function.  

{\color{black} To perform the symbol calculation, we first assume $\partial D \in \mathcal{C}^{\infty}$.} To compute the derivatives of the symbols,
we shall be more careful and keep in mind that the second derivatives of $g$ do not vanish at $x$.  For instance, we obtain the following:
\beqnx
&&\frac{\partial }{\partial x_l} p_{(\lambda I - \mathcal{K}^*_{\partial D} )^{-1}}(x,\xi)    = 
  { \color{black} -\frac{1}{2} } \lambda^{-2} (d-1)  \partial_l  H(x) \,  |\xi|^{-1} { \color{black} + \frac{1}{2} }  \lambda^{-2}  \langle \partial_l A(x)  \, \xi, \, \xi \rangle \, |\xi|^{-3}
+  O\left( \lambda^{-2}  |\xi|^{-2}\right) \, , \\
&&p_{ h \Delta_{\partial D}}(x , \xi) = h(x) | \xi |^2 \,,   \quad 
\frac{\partial }{\partial x_l} p_{ h \Delta_{\partial D}}(x, \xi) =   \partial_l h(x) \, |\xi|^2  + h(x) \sum_{i,k=1}^{d-1}  \partial_l \Gamma^{k}_{ii}(x) \, \xi_k  \,,
\eqnx
{\color{black}where $O( \lambda^{-2} |\xi|^{-2})$ indicates a symbol in $S^{-2}$ with a constant of order at least $\lambda^{-2} $.}
Together with the symbol of $\Lambda_0$ and its derivatives, we could render the symbols of the following four operators in terms of the geodesic normal coordinate (where $ | \lambda | > \frac{1}{2} $) as follows:
\beqnx
& & p_{h \circ \Lambda_0 \circ (\lambda I - \mathcal{K}_{\partial D}^*)^{-1} \circ \Lambda_0 } (x,\xi)  \\
&=&h(x) p_{\Lambda_0  } (x,\xi)  p_{ (\lambda I - \mathcal{K}_{\partial D}^*)^{-1} } (x,\xi)  p_{\Lambda_0 } (x,\xi)   + h(x)   \frac{\partial}{\partial \xi} p_{\Lambda_0  } (x,\xi)   \frac{\partial}{\partial x} p_{ (\lambda I - \mathcal{K}_{\partial D}^*)^{-1} } (x,\xi)  p_{\Lambda_0 } (x,\xi)  \\
& &
+ h(x) \frac{\partial}{\partial \xi} p_{\Lambda_0  } (x,\xi)  p_{ (\lambda I - \mathcal{K}_{\partial D}^*)^{-1} } (x,\xi)  \frac{\partial}{\partial x}  p_{\Lambda_0 } (x,\xi) \\
& &
+ h(x)  p_{\Lambda_0  } (x,\xi) \frac{\partial}{\partial \xi} p_{ (\lambda I - \mathcal{K}_{\partial D}^*)^{-1} } (x,\xi)  \frac{\partial}{\partial x} p_{\Lambda_0 } (x,\xi)  + O(\lambda^{-1} | \xi|^{0} ) \\
&=& 
 \lambda^{-1} h(x) |\xi|^2 
{ \color{black} - \frac{1}{2} \lambda^{-1} \left( \lambda^{-1} + 2 \right)  }  \bigg( (d-1)  h(x)  H(x)  | \xi |  - h(x)  \langle A(x) \xi, \xi \rangle \,   |\xi|^{-1}  \bigg)  +  O(\lambda^{-1})\, ,
\eqnx
{\color{black}where $O( \lambda^{-1})$ indicates a symbol in $S^{0}$ with a constant of order at least $\lambda^{-1} $,}
and
\beqnx
& & p_{ (\lambda I - \mathcal{K}_{\partial D}^*)^{-1}  \circ \mathcal{K}^{(1)}_{D,h}  \circ (\lambda I - \mathcal{K}_{\partial D}^*)^{-1} \circ \Lambda_0  } (x,\xi)  \\
&=& 
p_{ (\lambda I - \mathcal{K}_{\partial D}^*)^{-1} } (x,\xi) p_{ \mathcal{K}^{(1)}_{D,h} } (x,\xi) p_{ (\lambda I - \mathcal{K}_{\partial D}^*)^{-1}  } (x,\xi) p_{\Lambda_0  } (x,\xi) +  O(\lambda^{-1}) \\
&=& - { \color{black} 2 i } \lambda^{-2} \partial_\xi  h(x) +  O(\lambda^{-1})\, ,
\eqnx
as well as
\beqnx
& & p_{ \left[  h  \, H \, ,  (\lambda I - \mathcal{K}_{\partial D}^*)^{-1}  \right]  \circ \Lambda_0 } (x,\xi)  \\
& = & \left(  -  \frac{\partial }{\partial x} ( h(x) H(x)) \,  \frac{\partial }{\partial \xi}  \, p_{ (\lambda I - \mathcal{K}_{\partial D}^*)^{-1} } (x,\xi)  +   \frac{\partial }{\partial \xi }  ( h(x) H(x))  \,   \frac{\partial }{\partial x}  \, p_{ (\lambda I - \mathcal{K}_{\partial D}^*)^{-1} } (x,\xi) \right) p_{\Lambda_0}(x,\xi)\\
& & +  O\left( \lambda^{-2}  |\xi|^{-2}\right)  \\
&=& { \color{black} - \frac{1}{2 } } \lambda^{-2} \langle [ (d-1) H(x) I +  2 A(x) ] \xi \, , \, \partial_x ( h(x) H(x)) \rangle \, |\xi|^{-2}  +  O\left( \lambda^{-2}  |\xi|^{-2}\right) \\
&=& O\left( \lambda^{-2} |\xi|^{-1} \right) \, ,
\eqnx
{\color{black}where $O( \lambda^{-2} |\xi|^{-1})$ indicates a symbol in $S^{-1}$ with a constant of order at least $\lambda^{-2} $,}
and
\beqnx
&  &p_{(\lambda I - \mathcal{K}^*_{\partial D} )^{-1} \circ h \circ \Delta_{\partial D}  }(x,\xi) \\
&=& 
p_{(\lambda I - \mathcal{K}^*_{\partial D} )^{-1}}(x,\xi) p_{ h \Delta_{\partial D}}(x , \xi)  
+ \frac{\partial }{\partial \xi} p_{(\lambda I - \mathcal{K}^*_{\partial D} )^{-1}}(x,\xi)  \frac{\partial }{\partial x} p_{ h \Delta_{\partial D}}(x , \xi) + O\left( \lambda^{-2} 1 \right)\\
&=&
  \lambda^{-1}   h(x) | \xi |^2  { \color{black} - \frac{1}{2} }  \lambda^{-2} \bigg(  (d-1)  h(x) H(x) \,   |\xi| -  h(x) \langle A(x)  \, \xi, \, \xi \rangle \, |\xi|^{-1} \bigg)
+  O\left( \lambda^{-2} \right) \,,
\eqnx
{\color{black}where $O( \lambda^{-2})$ indicates a symbol in $S^{0}$ with a constant of order at least $\lambda^{-2} $.}
Therefore we can combine the above results to obtain 
\beqnx
& & p_{\mathcal{Q}}(x,\xi) \\
&=&
{ \color{black} 2 i} \lambda^{-2} \partial_\xi  h(x) 
 - \lambda^{-1} \bigg( (d-1)  h(x)  H(x)  | \xi |  - h(x)  \langle A(x) \xi, \xi \rangle \,   |\xi|^{-1}  \bigg) 
+  O\left( \lambda^{-1} \right),
\eqnx
{\color{black}where $O( \lambda^{-1})$ indicates a symbol in $S^{0}$ with a constant of order at least $\lambda^{-1} $,}
with the operator $\mathcal{Q}$ defined as
\[
\begin{split}
\mathcal{Q}:=& h \circ \Lambda_0 \circ (\lambda I - \mathcal{K}_{\partial D}^*)^{-1} \circ \Lambda_0 - (\lambda I - \mathcal{K}_{\partial D}^*)^{-1}  \circ \mathcal{K}^{(1)}_{D,h}  \circ (\lambda I - \mathcal{K}_{\partial D}^*)^{-1} \circ \Lambda_0 \\
&+ (d-1)  \left[  h  \, H \, ,  (\lambda I - \mathcal{K}_{\partial D}^*)^{-1}  \right]  \circ \Lambda_0 - (\lambda I - \mathcal{K}^*_{\partial D} )^{-1} \circ h \circ \Delta_{\partial D} \,. 
\end{split}
\]
We finally notice that the constants in the big O also depend on $\| \mathbb{X} \|_{\mathcal{C}^2} $ and $\| h \|_{\mathcal{C}^1} $.

{ \color{black}
Next if $\partial D \in \mathcal{C}^{k, \alpha}$, with the geodesic normal coordinate providing a parametrisation $\mathbb{X} \in \mathcal{C}^{k-1, \alpha} $ and \eqref{regularity}, we trace back the regularity to obtain \eqref{eq:t1} with additional specification of the H\"older regularity. 
}

The proof is complete. 

\end{proof}


\subsection{Localization of sensitivity of generalized polarization tensors at points of high mean curvature} \label{sechahaha}

Consider the space 
\beqnx
 \text{tr}_{\partial D} \text{Ker}( \Delta ) :=  \{ u\mid_{\partial D} : \Delta u = 0 \text{ in } \mathbb{R}^d \} .
\eqnx
Notice that $ \overline{\text{tr}_{\partial D} \text{Ker}( \Delta ) }^{H^{s}(\partial D, d \sigma) } = H^{s} (\partial D, d \sigma) $ is well-defined {\color{black} for all $s \in [-k,k]$ when $\partial D \in \mathcal{C}^{k,\alpha}$ \cite{regularity1,regularity2,regularity3,regularity4,regularity5}.}  Considering the fact that $ Q_{D,h} $ is a pseudo-differential operator of order $1$, and the closure of operators under the weak operator topology {\color{black} in the following sense: we define, for all $\phi \in H^{t} (\partial D, d \sigma)$, the value $Q_{D,h} (\phi)$ be such that if we consider the operator 
\[
 Q_{D,h,s}  :=  \text{Proj}_{W_{s,\partial D} }^*   Q_{D,h}    \text{Proj}_{W_{s,\partial D} }
\]
where $\text{Proj}_{W_{s,\partial D} } : L^2(\partial D, d \sigma) \rightarrow W_{s,\partial D} := \text{Span}\{  r^{n} Y_M(\omega)  |_{\partial D}  \}_{m_{d-1} \leq s} \subset \mathcal{C}^\infty (\partial D)$ with $ s \in \mathbb{N}$,
then we have for any 
\[
\left \langle \psi  \,,\,
 Q_{D,h,s} \, \phi \right\rangle_{L^2 (\partial D, d \sigma )}   \rightarrow \left \langle \psi  \,,\,
 Q_{D,h} \, \phi \right\rangle_{L^2 (\partial D, d \sigma )} 
\]
as $s \rightarrow \infty$.} Then we can have that 
\beqnx
&  &\bigg\{ \left \langle \psi  \,,\,
 Q_{D,h} \, \phi \right\rangle_{L^2 (\partial D, d \sigma )} \, : \, \psi \in  H^{s} (\partial D, d \sigma) , \phi \in  H^{t} (\partial D, d \sigma) ,  s , t \in [-k,k], s + t - 1 = 0   \bigg\} \\
&\subset & {\color{black} \left\{ \lim_{n \rightarrow \infty}  \sum_{\tiny \begin{matrix} k \leq n, \\ m \leq n \end{matrix} }  \sum_{\tiny \begin{matrix} L\in I_k, \\ M \in I_m \end{matrix} } a_L b_M  \mathcal{M}^{(1)}_{L,M} (\lambda, D , h)  \, : \, a_L, b_M  \in \mathbb{C} \text{ such that the limit exists.} \right\} } \,, \\
\eqnx
where from now on we abuse the notation of $ \langle \, \psi \, ,\, \, \phi \, \rangle_{L^2 (\partial D, d \sigma )} $ as an $L^2$-pivoting as soon as the resulting $  \overline{\psi} \, \phi \in L^1 (\partial D, d \sigma )$.
Therefore, the map $h \mapsto \left(  \mathcal{M}^{(1)}_{L,M} (\lambda, D , h) \right)_{L\in I_k, M \in I_n, k,n \in \mathbb{N}} $ can fully reconstruct the operator-valued map $ h \mapsto  Q_{D,h} = Q_{D,h,1,I} + Q_{D,h,1,II} + Q_{D,h,0}  $.  Now with suitable choices of $\psi \in  H^{s} (\partial D, d \sigma) , \phi \in  H^{t} (\partial D, d \sigma) $ such that $ s , t \in \mathbb{R}, s + t  -1 =0 $, we can obtain the principal symbol of $Q_{D,h}$ in the geodesic normal coordinate at each point $x$ as follows
\beqn
 \lim_{t \rightarrow \infty} t^{-1} e^{- i t \varphi_{x,\xi} }  Q_{D,h}  e^{i t \varphi_{x,\xi} } \chi_x =  p_{Q_{D,h,1,I}}(x,\xi)  + p_{Q_{D,h,1,II}}(x,\xi) \, , \label{step1}
\eqn
where $\xi \in \mathbb{S}^{d-1}$ and $\varphi_{x,\xi} (\,\cdot\,) = \langle \xi, \log_x(\,\cdot\,) \rangle $ {\color{black} in half of the injectivity radius of the geodesically convex neighborhood of $x$ {\color{black} (in the sense that any two points in the neighborhood can be joined by a unique minimizing geodesic, and that the exponential map is a diffeomorphism within the neighborhood)} and is zero outside $3/4$ of the injectivity radius, and $\chi_x (\,\cdot\,) $ is a cut-off function such that $\chi_x (\, x \,) = 1 $ in half of the injectivity radius and valued zero outside $3/4$ of injectivity radius.}  {\color{black} We remark that the functions $ \chi_x, \, e^{- i t \varphi_{x,\xi} }  \in \mathcal{C}^{\infty} $ when $\partial D \in \mathcal{C}^{\infty}$, whereas 
$ \chi_x, \, e^{- i t \varphi_{x,\xi} }  \in \mathcal{C}^{k-1,\alpha} \subset H^{k-1} (\partial D, d \sigma)$ when $\partial D \in \mathcal{C}^{k,\alpha}$.} Then $p_{Q_{D,h,1,I}}(x,\xi)  + p_{Q_{D,h,1,II}}(x,\xi) $ can be reconstructed in full by the property of being homogenous of degree one. On the other hand, one can recover $h(x) H(x)$ from $ p_{Q_{D,h,1,I}}(x,\xi)  + p_{Q_{D,h,1,II}}(x,\xi)$ as follows via Theorem \ref{lemmaDM}, 
\beqn
&  & \int_{\mathbb{S}^{d-1}}  \left( p_{Q_{D,h,1,I}}(x,\xi)  + p_{Q_{D,h,1,II}}(x,\xi) \right) |\xi|^{-1} d \sigma (\xi)  \notag \\
&=& 
\int_{\mathbb{S}^{d-1}}  
\bigg( { \color{black} 2i} \lambda^{-2} \, \partial_\xi  h(x)  | \xi |^{-1}  - \lambda^{-1} (d-1)  h(x)  H(x)  + \lambda^{-1} h(x)  \langle A(x) \xi, \xi \rangle \,   |\xi|^{-2}  \bigg)   d \sigma_\xi  \notag \\
&=&
{ \color{black}   \omega_d \, (2-d) }   \lambda^{-1}  h(x)  H(x) . \label{step2}
\eqn
Hence the inverse composition map
\beqnx
\begin{split}
& \left(  \mathcal{M}^{(1)}_{L,M} (\lambda, D , h) \right)_{L\in I_k, M \in I_n, k,n \in \mathbb{N}}\\
& \mapsto^{\text{inv}_1} Q_{D,h}  \mapsto^{\text{inv}_2} p_{Q_{D,h,1,I}}(x,\xi)  + p_{Q_{D,h,1,II}}(x,\xi) \mapsto^{\text{inv}_3}  { \color{black}   \omega_d \, (2-d) } h(x)  H(x)
\end{split}
\eqnx
is well-defined. 
To make the above description more precise, let us consider the following complete orthornormal bases on $L^2(\partial D, d \sigma)$:
\beqnx
 \{ \eta_{ p , \partial D} \}_{k \in \mathbb{N}} \quad \text{ where } \quad -\Delta_{\partial D}  \eta_{ p , \partial D} = \lambda_p^2 \eta_{ p , \partial D}  \,,
\eqnx
with $\lambda_p$ aligned in ascending order, and write $ \lambda(\Delta_{\partial D}) $ to be the set of eigenvalues $\lambda$ satisfying the above.
{\color{black}
By Weyl's asymptotics
\cite{Annal,Taylor}, {\color{black} when $\partial D \in \mathcal{C}^{\infty}$ is a $d-1$ dimensional compact smooth manifold}, we have
\[
\lim_{p \rightarrow \infty} \frac{ p  }{ \lambda_p^{d-1}} =   \frac{ \omega_{d-2} }{ (d-1) ( 2 \pi )^{d-1} }  | \partial D | ,
\]
where $\omega_{d-2}$ is the measure of $ \mathbb{S}^{d-2}$. This implies that one has at least $\lambda_p^{-1}  \sim p^{- \frac{1}{d-1}}$.  
}Therefore for any smooth function $ \phi $ on $\partial D$, we have $ \langle \eta_{p,\partial D} \,,\,  \phi \rangle_{L^2(\partial D, d \sigma)} =  O(  p^{-l} )$ for any $l$.
{\color{black} By virtue of \eqref{eigenvalue} and \eqref{eigenvector} in the proof of Lemma \ref{perturb_GPT}, we see that the Weyl's asymptotics still holds for $\partial D \in \mathcal{C}^{k,\alpha}$ for $k \geq 4$ (c.f. also \cite{regularity1}), and for any function $ \phi \in H^s (\partial D, d \sigma) $, $s \leq k$, we have $ \langle \eta_{p,\partial D} \,,\,  \phi \rangle_{L^2(\partial D, d \sigma)} =  O(  p^{-s} )$, {\color{black} where the constant contained in the big $O$ depends on $\phi$.}}
By the density of its subspace $\text{tr}_{\partial D} \text{Ker} (\Delta)$ in $L^2(\partial D, d \sigma)$, $ \{  r^{n} Y_M(\omega)  |_{\partial D} \}_{ M \in I_n, n \in \mathbb{N}}$ is also a complete frame in $L^2(\partial D, d \sigma)$.
{\color{black} Hence there is a change of basis map $\left( U_{  L , p,  \partial D} \right)$ which is the matrix for the change of the basis from $ \{  r^{n} Y_M(\omega)  |_{\partial D} \}_{ M \in I_n, n \in \mathbb{N}}$  to the corresponding orthonormal one $ \{ \eta_{ p , \partial D} \}_{k \in \mathbb{N}} $. We write $ \left( U^{-1}_{L , p , \partial D} \right) $ as its inverse, i.e. 
 the change of the basis from the orthonormal $ \{ \eta_{ p , \partial D} \}_{k \in \mathbb{N}} $  to the the bases $ \{  r^{n} Y_M(\omega)  |_{\partial D} \}_{ M \in I_n, n \in \mathbb{N}}$.
The precise definition of $\left( U_{ L , p, \partial D} \right)$ and $ \left( U^{-1}_{L , p , \partial D} \right) $ are discussed in more details in Remark \ref{Remark_2.10}. }
Moreover, since $\eta_{p,\partial D} $ is orthornomal,
\beqn
U^{-1}_{L , p , \partial D} =  \langle r^{k} Y_L(\omega)  \,,\,  \eta_{p,\partial D}  \rangle_{L^2(\partial D, d \sigma)}  \,.
\label{step3}
\eqn
{\color{black} As a remark though, the maps $U_{L , p , \partial D}$ and $U^{-1}_{L , p , \partial D}$ are in general unbounded {\color{black} from $l_2 $ to $l_2$, where $l_2 :=  \{ (a_i)_{i=1}^{\infty} : \sum_{i} |a_i|^2 < \infty \}$.}  This point will be further discussed in Lemma \ref{perturb_GPT} and Corollary~\ref{all_GPT_full}.}

Combining the above discussions {\color{black} which combines \eqref{step1}, \eqref{step2}, the defintion of how a vector in Hilbert space is expressed under an orthonormal basis and the definition of change of basis in \eqref{step3} }, we have the following theorem in force.
\begin{Theorem} \label{all_GPT}
{ \color{black} For $d > 2$,}
we have the inversion formula for $\partial D \in \mathcal{C}^{{\color{black}4},\alpha}$ and $h \in \mathcal{C}^{ {\color{black} 3},\alpha}$,
{
\beqn
&  & {\color{black}  \lambda^{-1}   \omega_d \, (2-d) } h(x) H(x) \notag \\
& = & \text{\rm inv}_3 \circ \text{\rm inv}_2 \circ \text{\rm inv}_1 \left[  \left(  \mathcal{M}^{(1)}_{L,M} (\lambda, D , h) \right)_{L\in I_k, M \in I_n, k,n \in \mathbb{N}}  \right ] \notag \\
&: =&
  \int_{\mathbb{S}^{d-1}}   \lim_{t \rightarrow \infty} \mathcal{G}(\xi, t, x) d \sigma (\xi) ,\notag \\
\label{invinvinv}
\eqn
}
where
\[
\begin{split}
\mathcal{G}(\xi, t, x):=\sum_{ \begin{matrix} L \in I_k, \ M \in I_n , \\   k, n, r, s \in \mathbb{N} \end{matrix} }
& |\xi|^{-1} t^{-1} e^{- i t \varphi_{x,\xi} }  \, \eta_{s , \partial D }  \, U_{  L , s,  \partial D}  \,  \mathcal{M}^{(1)}_{L,M} (\lambda, D , h) \\
&\quad \times  U^{-1}_{M , r , \partial D}  \, \langle \eta_{r , \partial D } \,,\,  \chi_x\, e^{i t \varphi_{x,\xi} } \rangle_{L^2(\partial D, d \sigma)}.
\end{split}
\]
\end{Theorem}
Now let $\partial D_\varepsilon$ be given as in \eqref{variationvariationD} with $\partial D \in \mathcal{C}^{{\color{black}4},\alpha}$ and $h \in \mathcal{C}^{ {\color{black} 3},\alpha}$, and $u_{\partial D_\varepsilon}$ be $u$ satisfying \eqref{transmission2} with $u_0 =  r^{n} Y_M(\omega)$ and the support of the inhomogeneity $D_\varepsilon$.
We further define 
{
\beqnx
\begin{split}
&\frac{\partial }{\partial \varepsilon} \left( \left \langle Y_L(\omega_x)  \, , \, \bigg(u_{\partial D_\varepsilon} -  r^{n} Y_M(\omega) \bigg) (R \, \omega_x ) \right \rangle_{ L^2( R \mathbb{S}^{d-1} , d \omega_x ) }   \right)_{L\in I_k, M \in I_n, k,n \in \mathbb{N}} \\
&\mapsto^{\text{inv}_0}
\left(  \mathcal{M}^{(1)}_{L,M} (\lambda, D , h) \right)_{L\in I_k, M \in I_n, k,n \in \mathbb{N}}\,,
\end{split}
\eqnx
}{\color{black} for a choice of radius $R > 0$.} Hence we have the following corollary:
\begin{Corollary} \label{all_GPT_full}
 { \color{black} For $d > 2$} and $\partial D \in \mathcal{C}^{{\color{black}4},\alpha}$ and $h \in \mathcal{C}^{{\color{black}3},\alpha}$, we have
\small
\begin{equation}\label{eq:comp1}
\begin{split}
&  { \color{black}   \lambda^{-1}  \omega_d \, (2-d) }  h(x) H(x)  \\
=&  \text{\rm inv}_3 \circ \text{\rm inv}_2 \circ \text{\rm inv}_1 \circ \text{\rm inv}_0 \left[ 
\frac{\partial }{\partial \varepsilon} \left( \left \langle Y_L(\omega_x)  \, , \, \bigg(u_{\partial D_\varepsilon} -  r^{n} Y_M(\omega) \bigg) (R \, \omega_x ) \right \rangle_{ L^2( R \mathbb{S}^{d-1} , d \omega_x ) }   \right)_{L\in I_k, M \in I_n, k,n \in \mathbb{N}} 
  \right ] \,. 
\end{split}
\end{equation}
\end{Corollary}

{\color{black}
\begin{Remark}
\label{Remark_2.10}
We would like to remark that $U_{L , p , \partial D}$ and $U^{-1}_{L , p , \partial D}$ amount to the maps $(a_p) \mapsto (b_L)$ and $ (b_L)  \mapsto (a_p) $ which are given via the following change of coordinates for $\phi \in L^2(\partial D)$,\vspace*{-2mm}
\[
\sum_{p \in \mathbb{N}} a_p \eta_{p,\partial D} = \phi = \sum_{ L \in I_k,  k \in \mathbb{N}} b_{L} r^k Y_{L} (\omega)   \,.
\]
In general these change of coordinate maps are unbounded maps between $(a_p) \in l_2$ to $ (b_L)   \in l_2$, where we again recall $l_2 :=  \{ (a_i)_{i=1}^{\infty} : \sum_{i} |a_i|^2 < \infty \}$, since $\{ \eta_{p,\partial \Omega} \}_{p \in \mathbb{N}}$ is an orthornormal basis in $L^2(\partial D)$ but that $ \{  r^{n} Y_M(\omega)  |_{\partial D} \}_{ M \in I_n, n \in \mathbb{N}}$ have their $L^2(\partial D)$-norms either exponentially growing or exponentially decaying.  To quickly illustrate this point, let us take (as in Examples I.1. that we will discuss) a simple example with $\partial D = R_0 \mathbb{S}^{d-1}$.  In this case, instead of indexing $\eta_{ p , \partial D}$ via $p \in \mathbb{N}$, we order them with $M \in I_n, n \in \mathbb{N}$, since\vspace*{-2mm}
\[
\eta_{ M, R_0 \mathbb{S}^{d-1} }  = \omega_{d-1}^{-1}  R_0^{-\frac{d-1}{2}} Y_M(\omega) , 
\]
where $\omega_{d-1}$ is the measure of $ \mathbb{S}^{d-1}$.  Therefore if $(a_M)$ and $(b_L)$ satisfy
\[
\sum_{ L \in I_k,  k \in \mathbb{N}} b_{L} R_0^k Y_{L} (\omega)  = \phi = \sum_{ M \in I_n,  n \in \mathbb{N}} a_{M} \eta_{ M, R_0 \mathbb{S}^{d-1} }   =  \sum_{ M \in I_n,  n \in \mathbb{N}} a_M \omega_{d-1}^{-1}  R_0^{-\frac{d-1}{2}} Y_M(\omega),
\]
then for all $M \in I_n,  n \in \mathbb{N}$, 
\[
a_M    = b_M \omega_{d-1} R_0^{n-\frac{d-1}{2}}.
\]
Hence, unless $R_0 = 1$, either the map $(a_M) \in l_2 \mapsto (b_L)   \in l_2$ or that $(b_L) \in l_2 \mapsto (a_M)   \in l_2$ is unbounded. As we will analyze in Examples I.2., the unboundedness occurs in a more general class of $\partial D$, and the following condition number
$$
\kappa( U_{L , p , \partial D } |_{\mathcal{L} ( V_{s,\partial D }, W_{s,\partial D } ) }  ) :=  
\| U_{L , p , \partial D }  |_{\mathcal{L} ( V_{s,\partial D }, W_{s,\partial D } ) }  \|_{l^2 \rightarrow l^2}
\| U^{-1}_{L , p , \partial D }  |_{\mathcal{L} ( V_{s,\partial D }, W_{s,\partial D } ) }  \|_{l^2 \rightarrow l^2}
$$
(where $V_{s,\partial D }, W_{s,\partial D }$ will be defined soon below) grows exponential with respect to $s$.  This confirms the unboundedness of this change of coordinate maps between $(a_p) \in l_2$ and $ (b_L)   \in l_2$.
\end{Remark}

In general, it is well known that the inverse problem from the scattered field to the reconstruction of $\partial D$ is exponentially ill-posed {\color{black}(especially when the map $\text{inv}_0$ introduces an additional exponentially growing factor $|R|^{2k+d + 2}$ from the Fourier coefficients of the scattered field to the GPTs in \eqref{eq:gpt2} as also described in Corollary \ref{cor:1} when $R$ is large).} However we shall dissect the composition map in Corollary~\ref{all_GPT_full} and understand more on how the ill-posedness is given by different geometric properties of the domain.
}Now let us gaze at the composition $ \text{inv}_3 \circ \text{inv}_2 \circ \text{inv}_1$ in \eqref{eq:comp1} and establish the properties of the composition of maps under a specific assumption on $\partial D$.

Before we proceed with the inverse problem, let us understand the perturbation of the change of the above coordinate maps and how they affect the condition number of a restriction of $U$ to a particular finite dimensional subspace.
{\color{black}
For this purpose, let us consider the following restriction and extension:
\beqnx
U^{-1}_{L , p , \partial D } |_{\mathcal{L} ( V_{s,\partial D }, W_{s,\partial D } ) } &:=&  \text{Proj}^*_{W_{s,\partial D } } \circ  U^{-1}_{L , M , \partial D }  \circ \text{Proj}_{V_{s,\partial D }} , \\
U_{L , p,  \partial D } |_{\mathcal{L} ( W_{s,\partial D } , V_{s,\partial D })} &:=&  \text{Proj}^*_{V_{s,\partial D }} \circ  U_{L, M , \partial D }  \circ \text{Proj}_{W_{s,\partial D }} ,
\eqnx
where $\text{Proj}_{V_{s,\partial D} } : L^2(\partial D, d \sigma) \rightarrow V_{s,\partial D} := \text{Span}\{ \eta_{p, \partial D}  \}_{p <  | \{ M : m_{d-1} \leq s \} | }$ and  $\text{Proj}_{W_{s,\partial D} } : L^2(\partial D, d \sigma) \rightarrow W_{s,\partial D} := \text{Span}\{  r^{n} Y_M(\omega)  |_{\partial D}  \}_{m_{d-1} \leq s}$ with $ s \in \mathbb{N}$, and the $U_{L, M , \partial D } $ is defined such that the operator $U^{-1}_{L , M , \partial D } $ is represented as
\[
U^{-1}_{L , M , \partial D }  := ( \langle r^{k} Y_L(\omega)  \,,\,  \eta_{p,\partial D}  \rangle_{L^2(\partial D, d \sigma)}  )_{}
\]
in the bases
$  \eta_{p, \partial D}  \}_{p <  | \{ M : m_{d-1} \leq s \}  }$ for $V_{s,\partial D} $ and
$ \{  r^{n} Y_M(\omega)  |_{\partial D}  \}_{m_{d-1} \leq s} $ for $ W_{s,\partial D}  $.
}

\begin{Lemma} \label{perturb_GPT}
{ \color{black} For $d > 2$,} consider a general $\partial D \in \mathcal{C}^{{\color{black} 4 },\alpha}$, and {\color{black} suppose $\text{Ric} \geq - C_0 \, g  $ for some $C_0$.} Let $\partial D_{\varepsilon}$ be an $\varepsilon$-perturbation under a direction $h \in \mathcal{C}^{ {\color{black} 3 },\alpha}$. Let $S= | \{ T : t_{d-1} \leq s \} | $, for $\varepsilon\in\mathbb{R}_+$ small enough, we have
\beqnx
& &\max \bigg\{  \left| \| U^{-1}_{L , p , \partial D_{\varepsilon} } |_{\mathcal{L} ( V_{s,\partial D_{\varepsilon} }, W_{s,\partial D_{\varepsilon} } ) }  \|_{l^2 \rightarrow l^2} -  \| U^{-1}_{L , p , \partial D } |_{\mathcal{L} ( V_{s,\partial D }, W_{s,\partial D } ) }  \|_{l^2 \rightarrow l^2}  \right|  \,, \,\\
& & \quad \quad \quad \left| \| U_{L , p , \partial D_{\varepsilon} } |_{\mathcal{L} ( V_{s,\partial D_{\varepsilon} }, W_{s,\partial D_{\varepsilon} } ) }  \|^{-1}_{l^2 \rightarrow l^2} -  \| U_{L , p , \partial D } |_{\mathcal{L} ( V_{s,\partial D }, W_{s,\partial D } ) }  \|^{-1}_{l^2 \rightarrow l^2}  \right| \bigg \} \\
&\leq&  { \color{black} C_d } \, \varepsilon \, \max_{1\leq P \leq S}  \bigg\{   \max \left \{ 1, \max_{z \neq \lambda_p, z \in \lambda(\Delta_{\partial D}) }{\frac{  { \color{black} \| g^{-1} \|_{C^0}^2 \max\{ 1,  \|g\|_{C^1}^3  \}   }  }{|z^2 - \lambda_p^2|}} \right \} \\
&   &  \quad \times { \color{black} \|  h \|_{C^1} }    \| A \|_{C^1}  (  \lambda_p^2 +  { \color{black} C_0 } )   \|r^k|_{\partial D}\|_{L^2(\partial D,d \sigma)}  \\
& & \quad +  \|  h \|_{C^1}   \| \partial_\nu ( r^k Y_L(\omega) ) |_{\partial D} \|_{L^2(\partial D,d \sigma)} +  (d-1) \|  h \|_{C^0}     \| H \|_{C^0} \|  r^k Y_L(\omega)  |_{\partial D} \|_{L^2(\partial D,d \sigma)} \bigg\}. 
\eqnx
for some dimensional constant $C_d$.
\end{Lemma}
\begin{proof}

For a general $\partial D \in \mathcal{C}^{4,\alpha}$, considering $\partial D_{\varepsilon}$ under a perturbation $h$ of $\mathcal{C}^{3,\alpha}$ { \color{black} (which guarantees that $\partial D_{\varepsilon}  \in \mathcal{C}^{3,\alpha}$) }, we compare the Laplacian on $\partial D$ with the one on $\partial D_{\varepsilon}$ 
via {\color{black} their corresponding Dirichlet forms.} By direct computations in a local coordinate, we have that
{\color{black}
\beqnx
&  & \Delta_{\partial D_{\varepsilon}}  - \Delta_{\partial D}  \\
&=& -  \varepsilon  \sum_{i,j,k,l}^{d-1}\frac{1}{\sqrt{|g|}}  \partial_j  \left(  h(x) \sqrt{|g|}  \left( 2 g^{ik}(x) A_{kl}(x) g^{lj}(x) +  (d-1) H(x) g^{ij}(x) \right)   \partial_i (\cdot) \right)  + O(\varepsilon^2), 
\eqnx
where we make use of $\partial_\nu g^{-1} = 2 g^{-1} A g^{-1} $ as well as $\partial_\nu  \sqrt{ |g|} = (d-1) H \sqrt{ |g| } $.}
Since $\Delta_{\partial D_{\varepsilon}}^{-1}$ is collectively compact {\color{black} from $L^2(\partial D_\varepsilon, d \sigma)$ to  $L^2(\partial D_\varepsilon, d \sigma) $ (which is identified as $L^2(\partial D, d \sigma) $ via the diffemorphism $\Psi_\varepsilon(x)$ introduced below \eqref{variationvariationD})}  with respect to $\varepsilon$, 
by Osborn's Theorem \cite{osborn} and that for repeated eigenvalues, we have, after applying $\Delta_{\partial D_{\varepsilon}}^{-1} =  - \varepsilon \Delta_{\partial D_{\varepsilon}}^{-1}  \partial_\varepsilon  \Delta_{\partial D_{\varepsilon}} \mid_{\varepsilon = 0} \Delta_{\partial D_{\varepsilon}}^{-1} + O(\varepsilon^2)$, that if we consider $- \lambda^2 \neq 0$ an eigenvalue of $\Delta_{\partial D}$ with multiplicity $m$ and $E_{\lambda}$ be its eigenspace, then there exists $\{ \eta_{\lambda,s, \partial D} \}_{s=1}^m $ being an orthonormal basis of $E_{\lambda}$ such that
{\color{black}
\beqnx
& & \lambda_{s, \varepsilon}^2 - \lambda^2 \\
&=&  \varepsilon \left \langle  \nabla \eta_{\lambda,s, \partial D} ,  h(x)  \left( 2 g^{-1} (x) A(x) g^{-1} (x) +  (d-1) H(x) g^{-1} (x) \right)  \nabla  \eta_{\lambda,s, \partial D}  \right \rangle_{L^2(\partial D, d \sigma)} + O(\varepsilon^2) ,
\eqnx
and
\beqnx
& &
{\left[ \eta_{\lambda,s, \partial D_{\varepsilon}} \circ \Psi_{\varepsilon} - \eta_{\lambda,s, \partial D} \right] }\big|_{E_\lambda^{\perp}}  \\
&=& \varepsilon \sum_{z \neq \lambda} \frac{ \left \langle  \nabla \eta_{z,s, \partial D} ,  h(x)  \left( 2 g^{-1} (x) A(x) g^{-1} (x) +  (d-1) H(x) g^{-1} (x) \right)  \nabla  \eta_{\lambda,s, \partial D}  \right \rangle_{L^2(\partial D, d \sigma)} }{z^2 - \lambda^2} \, \eta_{\lambda, \partial D} \\
&& + O(\varepsilon^2),
\eqnx
}where $\left( -\lambda_{s, \varepsilon}^2, \eta_{\lambda,s,, \partial D_{\varepsilon}}  \right)$ is an eigenpair of $\Delta_{\partial D_{\varepsilon}}$, and {\color{black}$\Psi_{\varepsilon}$ brings $\partial D$ to $\partial D_{\varepsilon}$}.  
{\color{black}
Moreover, by Bochner's formula \cite{Lee} (or the facts that $\nabla g = 0$, $\nabla_{e_i} \nabla_{e_j}  = \nabla_{e_j} \nabla_{e_i} + R(e_i, e_j) $ and integration by parts), we have
\begin{eqnarray*}
& &\int_{\partial \Omega}  \text{tr}_g \left( \text{Hess} \, \eta_{\lambda,s, \partial D} \, g^{-1} \, \text{Hess} \, \eta_{\lambda,s, \partial D}  \right)  d \sigma \\
& = &
\int_{\partial \Omega}  \left(  | \Delta_{\partial D} \eta_{\lambda,s, \partial D}  |^2 d \sigma
-  \text{Ric} \left( \partial \, \eta_{\lambda,s, \partial D} , \partial  \, \eta_{\lambda,s, \partial D} \right) \right)  d \sigma \\
& \leq &
 \int_{\partial \Omega}  \left(  | \Delta_{\partial D} \eta_{\lambda,s, \partial D}  |^2 d \sigma
+ {\color{black} C_0} \, g^{-1} \left( \partial \, \eta_{\lambda,s, \partial D} , \partial \, \eta_{\lambda,s, \partial D} \right) \right)  d \sigma \\
& = &  \lambda^4 + { \color{black} C_0 }  \lambda^2 \,.
\end{eqnarray*}
We also have
\begin{eqnarray*}
&  & \left \langle  \nabla \eta_{\lambda,s, \partial D} ,  h(x)  \left( 2 g^{-1} (x) A(x) g^{-1} (x) +  (d-1) H(x) g^{-1} (x) \right)  \nabla  \eta_{\lambda,s, \partial D}  \right \rangle_{L^2(\partial D, d \sigma)} \\
& \leq &  d \| h \|_{C^0}  \| A \|_{C^0} \| g^{-1}\|_{C^0} \left \langle  \nabla \eta_{\lambda,s, \partial D} , g^{-1} (x)  \nabla  \eta_{\lambda,s, \partial D}  \right \rangle_{L^2(\partial D, d \sigma)}   \\
& \leq &  C_d \| h \|_{C^0}  \| A \|_{C^0}  \| g^{-1}\|_{C^0}  \lambda^2  . 
\end{eqnarray*}
Similarly, we have (together with the fact that $\nabla g = 0$
) that
\begin{eqnarray*}
&  & \left \|  \text{div} \left( h(x)  \left( 2 g^{-1} (x) A(x) g^{-1} (x) +  (d-1) H(x) g^{-1} (x) \right)  \nabla  \eta_{\lambda,s, \partial D}  \right) \right \|_{L^2(\partial D, d \sigma)}^2  \\
&\leq & 
C_d \|h \|_{C^1}^2 \|A\|_{C^0}^2 \| g^{-1} \|_{C^0}^3 \left  \langle  \nabla \eta_{\lambda,s, \partial D} , g^{-1} (x)  \nabla  \eta_{\lambda,s , \partial D}  \right \rangle_{L^2(\partial D, d \sigma)}  \\
& & + C_d \|  h \|_{C^0}^2
\| A \|_{C^1}^2 \|g^{-1}\|_{C^0}^2  \max\{ 1,  \|g\|_{C^1}^5  \}  \left \langle  \nabla \eta_{\lambda,s, \partial D} , g^{-1} (x)  \nabla  \eta_{\lambda,s, \partial D}  \right \rangle_{L^2(\partial D, d \sigma)}  \\
&  & + C_d \|  h \|_{C^0}^2 \|A\|_{C^0}^2 \| g^{-1} \|_{C^0}^2 \int_{\partial \Omega}  \text{tr}_g \left( \text{Hess} \, \eta_{\lambda,s, \partial D} \, g^{-1} \, \text{Hess} \, \eta_{\lambda,s, \partial D}  \right)  d \sigma \\
&\leq & 
 C_d \|  h \|_{C^1}^2 \| A \|_{C^1}^2  \| g^{-1} \|_{C^0}^3  \max\{ 1,  \|g\|_{C^1}^5  \}     \left(  \lambda^4 + { \color{black} C_0 }  \lambda^2 \right). 
\end{eqnarray*}
Combining the previous estimates, we obtain
\beqnx
\lambda_{s, \varepsilon}^2 - \lambda^2 =  \varepsilon \, C_d \| h \|_{C^0}  \| A \|_{C^0}   \| g^{-1}\|_{C^0}  \lambda^2  + O(\varepsilon^2) \,,
\eqnx
and hence
\beqn
| \lambda_{s, \varepsilon}^2 - \lambda^2 | \leq  \varepsilon \, C_d \int_0^1 \| h \|_{C^0}  \| A(t) \|_{C^0}   \| g^{-1} (t)\|_{C^0}  \lambda^2(t)  dt \leq  \varepsilon \, C_d \| h \|_{C^0}  \| A \|_{C^0}   \| g^{-1}\|_{C^0}  \lambda^2, \ \
 \label{eigenvalue}
\eqn
where $ ( g(t), A(t) , \lambda(t)) $ are the respective metric, curvature and eigenvalue of $\partial D_{t \varepsilon}$ with $\| A(t) \|_{C^0}$ and $\| g^{-1} (t)\|_{C^0} $ guaranteed to have uniform bounds across $t \in [0,1]$ for small $\varepsilon$ and $h \in \mathcal{C}^{3,\alpha}$.
In a similar manner, by combining the previous estimates, we also have
\beqn
& &\| {\left[ \eta_{\lambda,s, \partial D_{\varepsilon}} \circ \Psi_{\varepsilon} - \eta_{\lambda,s, \partial D} \right] }|_{E_\lambda^{\perp}} \|^2_{L^2(\partial D, d \sigma)} \notag \\
& \leq & \varepsilon^2 \, \max_{z \neq \lambda, z \in \lambda(\Delta_{\partial D}) }{\frac{1}{|z^2 - \lambda^2|^2}}\times \notag \\
&& 
\left \|  \text{div} \left( h(x)  \left( 2 g^{-1} (x) A(x) g^{-1} (x) +  (d-1) H(x) g^{-1} (x) \right)  \nabla  \eta_{\lambda,s, \partial D}  \right) \right \|_{L^2(\partial D, d \sigma)}^2 
\notag \\
& \leq &  C_d \,  \varepsilon^2 \, \max_{z \neq \lambda, z \in \lambda(\Delta_{\partial D}) }{\frac{1}{|z^2 - \lambda^2|^2}} \|  h \|_{C^1}^2 \| A \|_{C^1}^2  \| g^{-1} \|_{C^0}^3  \max\{ 1,  \|g\|_{C^1}^5  \}     \left(  \lambda^4 + { \color{black} C_0 }  \lambda^2 \right), \quad  \quad  
\label{eigenvector}
\eqn
where we make use of the fact that $\| A(t) \|^2_{C^1}   \| g^{-1} (t)\|^3_{C^1} \max\{ 1, \| g^{-1} (t)\|^5_{C^1} \} $ is guaranteed to have a uniform bound across $t \in [0,1]$ for small $\varepsilon$ as $h \in \mathcal{C}^3$.
}From the fact that $\| \eta_{\lambda,s, \partial D_{\varepsilon}}  \|^2_{L^2(\partial D_{\varepsilon}, d \sigma)} =1$, $d \sigma_{\partial D_{\varepsilon}} = (1 + \varepsilon (d-1) h(x) H(x) )d \sigma_{\partial D} + O(\varepsilon^2)$ and $L^2 (\partial D, d \sigma) =  E_{\lambda} \oplus E_{\lambda}^{\perp}$, we further have for $s = 1,...,m$, 
\beqnx
& &\|  \eta_{\lambda,s, \partial D_{\varepsilon}} \circ \Psi_{\varepsilon} - \eta_{\lambda,s, \partial D}  \|^2_{L^2(\partial D, d \sigma)} \\
& \leq &  { \color{black} C_d } \varepsilon^2 \, \max \left \{ 1, \max_{z\neq \lambda}{\frac{ { \color{black} \| g^{-1} \|_{C^0}^3  \max\{ 1,  \|g\|_{C^1}^5  \}   }}{|z^2 - \lambda^2|^2}} \right \} { \color{black} \|  h \|_{C^1}^2 }     \| A \|^2_{C^1}  \left(  \lambda^4 + { \color{black} C_0 }  \lambda^2 \right)   
. 
\eqnx
Therefore we have that if $\varepsilon\in\mathbb{R}_+$ is sufficiently small, 
\beqnx
& & \left|  \langle r^{k} Y_L(\omega)  \,,\,  \eta_{p,\partial D_{\varepsilon}}  \rangle_{L^2(\partial D_{\varepsilon}, d \sigma)} -   \langle r^{k} Y_L(\omega)  \,,\,  \eta_{p,\partial D}  \rangle_{L^2(\partial D, d \sigma)} \right|  \\
&\leq&  { \color{black} C_d } \, \varepsilon \, \bigg\{   \max \left \{ 1, \max_{z \neq \lambda_p, z \in \lambda(\Delta_{\partial D}) }{\frac{  { \color{black} \| g^{-1} \|_{C^0}^2 \max\{ 1,  \|g\|_{C^1}^3  \}   }  }{|z^2 - \lambda_p^2|}} \right \} { \color{black} \|  h \|_{C^1} }    \| A \|_{C^1}  (  \lambda_p^2 +  { \color{black} C_0 } )   \|r^k|_{\partial D}\|_{L^2(\partial D,d \sigma)}  \\
& & \quad +  \|  h \|_{C^1}   \| \partial_\nu ( r^k Y_L(\omega) ) |_{\partial D} \|_{L^2(\partial D,d \sigma)} +  (d-1) \|  h \|_{C^0}     \| H \|_{C^0} \|  r^k Y_L(\omega)  |_{\partial D} \|_{L^2(\partial D,d \sigma)} \bigg\}. 
\eqnx
Next we recall that
\beqnx
 \| U^{-1}_{L , p , \partial D } |_{\mathcal{L} ( V_{s,\partial D }, W_{s,\partial D } ) }  \|_{l^2 \rightarrow l^2}  = \sigma_{\max}(  U^{-1}_{L , M , \partial D } |_{\mathcal{L} ( V_{s,\partial D }, W_{s,\partial D } ) } ) , \\
\| U_{ L, q , \partial D } |_{\mathcal{L} ( W_{s,\partial D } , V_{s,\partial D })}   \|_{l^2 \rightarrow l^2} = 1/  \sigma_{\min}(U_{ L , M, \partial D } |_{\mathcal{L} ( W_{s,\partial D } , V_{s,\partial D })}  )    ,
\eqnx
where $\sigma_{\max} (T)$ and $\sigma_{\min} (T)$ are the respective maximum and minimum singular values of an operator $T$. Finally, by Osborn's theorem once again, one can show the lemma.
\end{proof}

{\color{black}  
\begin{Remark}
Notice that Lemma \ref{perturb_GPT} requires $h \in C^{ {\color{black} 3 },\alpha} $ but the error bound is given in terms of $ \| h \|_{\mathcal{C}^1} <1$. Hence, we are actually considering a $\mathcal{C}^1$ $\varepsilon$-ball intersecting $C^{ {\color{black} 3 },\alpha}$.  The only reason to have a finite $C^{ {\color{black} 3 },\alpha}$ bound is to guarantee $\partial D_{t \epsilon}$ to be in $\mathcal{C}^{{\color{black} 3 } ,\alpha}$ across $0 \leq t \leq 1$ for small $\varepsilon$ and keep $\| A(t) \|_{C^1}$ uniformly bounded across the choice of perturbations.
\end{Remark}
}
{ \color{black}
\begin{Remark}
We note that via the Gauss-Codazzi formula \cite{Jost,Kobayash2} and the fact that the ambient space $\mathbb{R}^d$ is flat, one has in the semi-geodesic normal coordinate for $| \, \omega \,| = 1$:
\[
\text{Ric} (\omega, \omega) =  (d-1) H(x) \,  \langle A(x)  \, \omega, \, \omega\rangle - \, |\, A(x) \omega \, |^2   \, . 
\]
Hence if assuming the strict convexity of $\partial D$, i.e. $- A  \geq R^{-1} g $ for some $R > 0$ (cf. \eqref{expansion0}), we then have in the geodesic normal coordinate that
\beqn\label{eq:aaaa2}
\text{Ric} (\omega, \omega) \geq  (d-2) R^2 | \, \omega \,|^2  \,. \label{curvature_bound}
\eqn
By \eqref{eq:aaaa2}, together with the Bonnet-Myers theorem \cite{DoCarmo,Myers} and the fact that $g$ on $\partial D$ is induced from $\mathbb{R}^d$, we can show that
\[
 \text{diam} (\partial D)  := \sup_{p,q \in \partial D}  |p - q|  \leq  \sup_{p,q \in \partial D}  d_g(p,q)  < C_d R, 
\]
where $C_d$ is a dimensionality constant. That is, $\partial D$ sits in an open ball of radius $R$. We would also like to point out that if $\partial D$ is assumed to be strictly convex, the statement of Lemma \ref{perturb_GPT} can be simplified, e.g.,
\[ \| g\|_{C^0}  \| g^{-1}\|_{C^0}  \leq C_d  \, \left(1 + \frac{  \sup_{p,q \in \partial D}  d_g(p,q) \, \| g\|_{C^1}  }{\| g\|_{C^0}} \right) \leq C_d \, \left(1 +  { \color{black} C_d R } \frac{\| g\|_{C^1}  }{\| g\|_{C^0}} \right)  \,. \]
But the validity of Lemma \ref{perturb_GPT} does not require this assumption.
\end{Remark}
}

{\color{black} \noindent In order to consolidate our study, we next compute the condition number
$$
\kappa( U_{L , p , \partial D } |_{\mathcal{L} ( V_{s,\partial D }, W_{s,\partial D } ) }  ) :=  
\| U_{L , p , \partial D }  |_{\mathcal{L} ( V_{s,\partial D }, W_{s,\partial D } ) }  \|_{l^2 \rightarrow l^2}
\| U^{-1}_{L , p , \partial D }  |_{\mathcal{L} ( V_{s,\partial D }, W_{s,\partial D } ) }  \|_{l^2 \rightarrow l^2}
$$
for several concrete examples of $\partial D$.
}

\medskip

\noindent \textbf{Example I.1.} Let us first consider $\partial D = R_0 \mathbb{S}^{d-1}$.  In this case, instead of indexing $\eta_{ k , \partial D}$ via $k \in \mathbb{N}$, we order them with $M \in I_n, n \in \mathbb{N}$, since
\[
\eta_{ M, R_0 \mathbb{S}^{d-1} }  = \omega_{d-1}^{-1}  R_0^{-\frac{d-1}{2}} Y_L(\omega) , 
\]
where $\omega_{d-1}$ is the measure of $ \mathbb{S}^{d-1}$. One has
\beqnx
U^{-1}_{L , M , R_0 \mathbb{S}^{d-1}} =    \omega_{d-1}^{-1}  R_0^{k -\frac{d-1}{2}}  \langle  Y_L(\omega)  \,,\,   Y_M(\omega)   \rangle_{L^2(\partial D, d \sigma)} =  \omega_{d-1}^{-1}R_0^{k -\frac{d-1}{2}}   \delta_{L M}  \,.
\eqnx
Hence, it can deduced that
{
\beqnx
&& \| U^{-1}_{L , M , R_0 \mathbb{S}^{d-1}} |_{\mathcal{L} ( V_{s,R_0 \mathbb{S}^{d-1}}, W_{s,R_0 \mathbb{S}^{d-1}} ) }  \|_{l^2 \rightarrow l^2}\\
 &  = & \sigma_{\max}(  U^{-1}_{L , M , R_0 \mathbb{S}^{d-1}} |_{\mathcal{L} ( V_{s,R_0 \mathbb{S}^{d-1}}, W_{s,R_0 \mathbb{S}^{d-1}} ) } )  =  \omega_{d-1}^{-1} \, R_0^{-\frac{d-1}{2}}  \max \{ R_0^s , 1 \},  \\
&&\| U_{M , L , R_0 \mathbb{S}^{d-1}} |_{\mathcal{L} ( W_{s,R_0 \mathbb{S}^{d-1}} , V_{s,R_0 \mathbb{S}^{d-1}})}   \|_{l^2 \rightarrow l^2}\\
& = & 1/  \sigma_{\min}(U_{M , L , R_0 \mathbb{S}^{d-1}} |_{\mathcal{L} ( W_{s,R_0 \mathbb{S}^{d-1}} , V_{s,R_0 \mathbb{S}^{d-1}})}  )    =  \omega_{d-1} \,  R_0^{\frac{d-1}{2} }  \max \{ R_0^{-s} , 1 \}.    
\eqnx
}
Therefore we have the following estimate of the condition number when $\partial D = R_0 \mathbb{S}^{d-1}$:
\beqnx
\kappa( U_{M , L , R_0 \mathbb{S}^{d-1}} |_{\mathcal{L} ( W_{s,R_0 \mathbb{S}^{d-1}} , V_{s,R_0 \mathbb{S}^{d-1}})}  ) =  \max\{ R_0^{-s}, R_0^{s}  \}\,.
\eqnx

\noindent \textbf{Example I.2.}  Consider $\partial D_{\delta} $ as follows: $\partial D = R_0 \mathbb{S}^{d-1}$ with perturbation $\tilde{h} \in  \mathcal{C}^{ {\color{black} 4 },\alpha} $ such that $\| \tilde{h}  \|_{\mathcal{C}^1} <1$ and a sufficiently small magnitude $\delta$. {\color{black} Since $\partial D$ is $\mathcal{C}^{\infty}$, for small $\delta$, $\partial D_{\delta}  \in \mathcal{C}^{ {\color{black}4}, \alpha}$.  Now on $R_0 \mathbb{S}^{d-1}$, it is directly verified that $-A_{R_0 \mathbb{S}^{d-1}} \geq C_d R_0^{-1}g$. Hence by the Gauss-Codazzi formula, we have {\color{black} \eqref{curvature_bound} } with $R_0$ replacing $R$.} 
Applying Lemma \ref{perturb_GPT}, together with $\lambda_{T,\mathbb{S}^d} = s(s + d-2)$ for $t_{d-1} = s$, we have for $\delta\in\mathbb{R}_+$ sufficiently small and $s\in\mathbb{R}_+$ sufficiently large that 
\beqnx
& &\max \bigg\{  \left| \| U^{-1}_{L , p , \partial D_{\delta} } |_{\mathcal{L} ( V_{s,\partial D_{\delta} }, W_{s,\partial D_{\delta} } ) }  \|_{l^2 \rightarrow l^2} - \omega_{d-1}^{-1} \, R_0^{-\frac{d-1}{2}}  \max \{R_0^s , 1\} \right|  \, , \, \\
&  & \quad \quad \quad \left| \| U_{L , p , \partial D_{\delta} } |_{\mathcal{L} ( V_{s,\partial D_{\delta} }, W_{s,\partial D_{\delta} } ) }  \|_{l^2 \rightarrow l^2}^{-1} -   \omega_{d-1}^{-1} \,  R_0^{- \frac{d-1}{2} }  \min \{R_0^s , 1\}   \right| \bigg\}  \\
&<&   \delta C_d \, \max \bigg\{    s(s + d-2)  R_0^{s+ 1 }  +   s   R_0^{s-2} +  R_0^{s-1}  \, , \, (d-1)  R_0^{2}  +   R_0^{-1} +  1 \bigg\} \\
&<&   \delta C_d \, \max \bigg\{    (s + d-2)^2  R_0^{s+ 1 }  \, , \, R_0^{-1} +  d \bigg\}  \, .
\eqnx
Therefore for small enough $\delta\in\mathbb{R}_+$, we obtain
\beqnx
& & \kappa( U_{L , p , \partial D_{\delta} } |_{\mathcal{L} ( V_{s,\partial D_{\delta} }, W_{s,\partial D_{\delta} } ) }  ) \\
&\leq&
\frac{ \max \{R_0^s , 1\} +  \delta C_d  R_0^{ \frac{d-1}{2} }  \, \max \bigg\{    (s + d-2)^2 R_0^{s+1 } \, , \, R_0^{-1} +  d \bigg\}    }{  \min \{R_0^s , 1\}  -  \delta  C_d  R_0^{ \frac{d-1}{2} }  \, \max \bigg\{    (s + d-2)^2  R_0^{s+ 1 } \, , \, R_0^{-1} +  d \bigg\}    } \\
&\leq&
\begin{cases}
\frac{ R_0^s +  \delta C_d R_0^{ \frac{d-1}{2} } (s + d-2)^2  R_0^{s+ 1 }}{  1 -   \delta C_d  R_0^{ \frac{d-1}{2} }   (s + d-2)^2  R_0^{s+ 1 }    } & \text{ if } R_0 \geq 1 \,,\\
\frac{ 1 +  \delta C_d R_0^{ \frac{d-1}{2} }\,  ( R_0^{-1} +  d )  }{  R_0^s   -   \delta C_d  R_0^{ \frac{d-1}{2} } \, ( R_0^{-1} +  d )  } & \text{ if } R_0 \leq 1  \,.
\end{cases}
\eqnx

From the above example, we have the following corollary.

{
\begin{Corollary}\label{cor:1}

{ \color{black} For $d > 2$,} let us consider $\partial D_{\delta} $ as a $\delta$-perturbation of $\partial D = R_0 \mathbb{S}^{d-1}$ along the direction $\tilde{h}  \in C^{ {\color{black} 4 },\alpha} (\partial D)$ with {\color{black}$\| \tilde{h}  \|_{\mathcal{C}^1} <1$ for sufficiently small $\delta\in\mathbb{R}_+$}.  
Then for $h \in \mathcal{C}^{ {\color{black} 3 },\alpha} (\partial D_{\delta})$ with $\| h \|_{\mathcal{C}^1} <1$, considering an $\varepsilon$-perturbation of $\partial D_{\delta} $ along the direction $h$, namely $( \partial D_{\delta} )_{\varepsilon} $, we have
\beqnx
&  & | [ \text{\rm Proj}_{V_{s,\partial D_{\delta} } } (h H) ] (x) | \notag \\
&\leq&
\begin{cases}
  \lambda^{-1}  C_d \frac{ R_0^s +  \varepsilon C_d  R_0^{ \frac{d-1}{2} } (s + d-2)^2  R_0^{s+ 1 } }{  1 -  \varepsilon C_d R_0^{ \frac{d-1}{2} }   (s + d-2)^2  R_0^{s+ 1 }   }  \, \| \mathcal{M}^{(1)}_{L,M} (\lambda, D , h)  \|_{\mathcal{L} (V_{s, \mathbb{S}^{d-1} }, V_{s, \mathbb{S}^{d-1} })}    & \text{ if } R_0 \geq 1, \\
  \lambda^{-1}  C_d \frac{ 1 +   \varepsilon C_d  R_0^{ \frac{d-1}{2} }\,  ( R_0^{-1} +  d )  }{  R_0^s   -   \varepsilon C_d   R_0^{ \frac{d-1}{2} } \, ( R_0^{-1} +  d )  } \, \| \mathcal{M}^{(1)}_{L,M} (\lambda, D , h)  \|_{\mathcal{L} (V_{s, \mathbb{S}^{d-1} }, V_{s, \mathbb{S}^{d-1} })}  & \text{ if } R_0 \leq 1 .
\end{cases}
\eqnx
Similarly,
{
\beqnx
&  & | [ \text{\rm Proj}_{V_{s,\partial D_{\delta}} } (h H) ] (x) | \notag \\
&\leq&
\begin{cases}
  \lambda^{-1}  C_d \frac{R^{2s + d -2}}{c_{d,s}}  \frac{ R_0^s +  \varepsilon C_d  R_0^{ \frac{d-1}{2} } (s + d-2)^2  R_0^{s+ 1 }  }{  1 -   \varepsilon C_d  R_0^{ \frac{d-1}{2} }   (s + d-2)^2 R_0^{s+ 1 }    }\times \\
\  \left \|
\frac{\partial }{\partial \varepsilon} \left( \left \langle Y_L(\omega_x)  \, , \, \bigg(u_{( \partial D_{\delta} ) _\varepsilon} -  r^{n} Y_M(\omega) \bigg) (R \, \omega_x ) \right \rangle_{ L^2( R \mathbb{S}^{d-1} , d \omega_x ) }   \right)
\right \|_{\mathcal{L} (V_{s, \mathbb{S}^{d-1} }, V_{s, \mathbb{S}^{d-1} })}   & \text{ if } R_0 \geq 1, \\
  \lambda^{-1}  C_d  \frac{R^{2s + d -2}}{c_{d,s}} \frac{ 1 +  \varepsilon C_d  R_0^{ \frac{d-1}{2} }\,  ( R_0^{-1} +  d )  }{  R_0^s   -  \varepsilon C_d  R_0^{ \frac{d-1}{2} } \, ( R_0^{-1} +  d )  } \times\\
 \ \left \|
\frac{\partial }{\partial \varepsilon} \left( \left \langle Y_L(\omega_x)  \, , \, \bigg(u_{ ( \partial D_{\delta} ) _\varepsilon} -  r^{n} Y_M(\omega) \bigg) (R \, \omega_x ) \right \rangle_{ L^2( R \mathbb{S}^{d-1} , d \omega_x ) }   \right)
\right \|_{\mathcal{L} (V_{s, \mathbb{S}^{d-1} }, V_{s, \mathbb{S}^{d-1} })}  & \text{ if } R_0 \leq 1 ,
\end{cases}
\eqnx
for some dimensionality constants $C_d$.
}
\end{Corollary}
}

\begin{proof}
For a given resolution $s \in \mathbb{N}$, we have, {\color{black} applying Theorem \ref{all_GPT} }
\beqnx
&  & | [ \text{Proj}_{V_{s,\partial D_{\delta}} } (h H) ] (x) | \notag \\
& = & \left | \text{inv}_3 \circ \text{inv}_2 \circ \text{inv}_1 \left[  \left(  \mathcal{M}^{(1)}_{L,M} (\lambda, D , h) \right)_{L\in I_k, M \in I_n, k,n \leq s}  \right ] \right| \\
&\leq &  \lambda^{-1} C_d'  \, \kappa( U_{L , p , \partial D_{\delta} } |_{\mathcal{L} ( V_{s,\partial D_{\delta} }, W_{s,\partial D_{\delta} } ) }  )    \limsup_{t \rightarrow \infty} \|  \chi_x\, e^{i t \varphi_{x,\xi} } \|_{L^2(\partial S, d \sigma )}  \, \| \mathcal{M}^{(1)}_{L,M} (\lambda, D , h)  \|_{\mathcal{L} (V_{s, \mathbb{S}^{d-1} }, V_{s, \mathbb{S}^{d-1} })} ,
\eqnx
for some constant $C_d'$. The results follow from the computations in Example 2 and \eqref{eq:gpt2}.
\end{proof}

{\color{black}
\begin{Remark}
It is noted that in the above example we assume $ \tilde{h} \in C^{ {\color{black} 4 },\alpha} $ and $h \in C^{ {\color{black} 3 },\alpha} $ but we only require $ \| h \|_{\mathcal{C}^1},  \| \tilde{h} \|_{\mathcal{C}^1} <1$.  This allows flexibility to embrace non-convex $\partial D_{\delta}$.
\end{Remark}
}

{\color{black}
\begin{Remark}
For this example, we can see that the reconstruction of $h(x)$ from $ \mathcal{M}^{(1)}_{L,M} (\lambda, D , h)$ is more stable with points of high mean curvature $|H(x)|^2$ if $D$ is not too far away from $R_0 \mathbb{S}^{d-1}$. In fact, the growth of the norm of the linearized inverse (corresponding to the reconstruction of $H(x) h(x)$ from $\mathcal{M}^{(1)}_{L,M}(\cdot,\cdot, h)$) with respect to the restriction of the Laplacian eigenspaces (corresponding to the resolution along the surface) is quite uniform across the domains $D$. Hence if $H(x)$ is large, $h(x)$ can be clearly seen; whereas if $H(x)$ is small, the reconstruction of $h(x)$ becomes more unstable since one needs to additionally divide the reconstructed $H(x) h(x)$ by $H(x)$.

This can be more explicitly explained as follows. If we have $\bigg( \left(u- r^{n} Y_M(\omega) \right)^{\text{meas}}|_{\Gamma} \bigg)_{M \in I_n, n \in \mathbb{N}  }$, one can obtain 
$ \bigg( ( \mathcal{M}_{L,M} )^{\text{meas}} (\lambda, D_{\text{exact}} ) \bigg)_{L\in I_k, M \in I_n, k,n \in \mathbb{N}} $. 
{\color{black} Then one may reconstruct $D$ via the following Newton type iteration for solving the equation
\[
 \mathcal{M}_{L,M} (\lambda, D^n ) = ( \mathcal{M}_{L,M} )^{\text{meas}} (\lambda, D_{\text{exact}} )
\]
which give $(D^{n} , h^{n} )$, $n=1,2,\ldots$ recursively} as follows:
\begin{equation}\label{eq:alg1}
\begin{split}
 \mathcal{M}^{(1)}_{L,M} (\lambda, D^n , h^{n})  := &  ( \mathcal{M}_{L,M} )^{\text{meas}} (\lambda, D_{\text{exact}} ) - \mathcal{M}_{L,M} (\lambda, D^n )  \, , \\
\partial D^{n+1} :=& \{ x + {\color{black} h^n}(x) \nu^n(x) \, : \, x \in \partial D^n  \} \, ,
\end{split}
\end{equation}
where $\nu^n$ denotes the unit normal vector on $\partial D^n$.  As we can see, the reconstruction step for $h^n$ is now more stable with points of high mean curvature $|H(x)|^2$ when $D^n$ starts to move close to $(\lambda, D_{\text{exact}})$, {\color{black}
with the inversion step $ \mathcal{M}^{(1)}_{L,M} (\lambda, D^n , h^{n})  \mapsto H(x) h^{n}(x)$ has a bound uniform with $x \in \partial D$ as in Corollary \ref{cor:1}, indicating that the inversion step $ \mathcal{M}^{(1)}_{L,M} (\lambda, D^n , h^{n})  \mapsto h^{n}(x)$ is more stable when $|H(x)|^2$ is high.
}
In two remarks below, we shall further discuss how Corollary \ref{cor:1} helps to develop a more stable Newton-type method.
\end{Remark}
}

\begin{Remark}
It is remarked that the recovery of $h^n$ in \eqref{eq:alg1} shall be numerically performed {\color{black} via the explicit numerical computation of $ \mathcal{M}^{(1)}_{L,M} (\lambda, D, \cdot ) $ from \eqref{seriesvariation} and \eqref{variationvariation} with the choice of a quadrature rule} instead of $ \text{inv}_3 \circ \text{inv}_2 \circ \text{inv}_1$ in \eqref{invinvinv} and division by $H(x)$.  The composition of the three operators is considered only for the theoretical analysis of sensitivity, and is not ideal to perform numerically. 
\end{Remark}

\begin{Remark}
It is emphasized that we did not claim that
\beqnx
h \mapsto \left(  \left[ \bigg(u -  r^{n} Y_M(\omega) \bigg) (R \, \omega_x ) \right] (h) \right)_{M \in I_n, n \in \mathbb{N}}
\eqnx
has a bounded inverse, but only that
\beqnx
[ \text{\rm Proj}_{V_{s,\partial D} } (h H) ] (x) \mapsto \left(  \mathcal{M}^{(1)}_{L,M} (\lambda, D,h ) \right)_{L \in I_k, M \in I_n, k, n  \leq s }   
\eqnx
has a bounded inverse for each $s$, considering the fact that the inverse problem to reconstruct $D$ from the scattered fields is exponentially ill-posed as is shown in \cite{curv_Liu_2,CR} and indicated by the decay order of $ \mathcal{M}_{L,M} (\lambda, D ) $. 
\end{Remark}

{ \color{black}

\begin{Remark}
Stable optimization methods have been discussed and employed for the shape reconstruction from GPTs; see e.g \cite{gpt,mono_1,mono_2,spectral}. In what follows, we describe a numerically more stable Newton-type method instead of \eqref{eq:alg1} to elucidate how our theoretical understanding helps precondition the computation of the Gauss-Newton directions (as suggested by Corollary \ref{cor:1}). To be specific, let us consider the inverse problem of reconstructing a target shape $D_{\text{exact}}$ from the noisy measurement of the scattered fields by first taking their Fourier coefficients to obtain the GPTs, in such a way that GPTs are taken up to order $k$, i.e. $ \left ( ( \mathcal{M}_{L,M} )^{\text{meas}} (\lambda, D_{\text{exact}} ) \right)_{L \in I_r, M \in I_s, r, n \leq k }  $. The signal-to-noise ratio provides the error bound for all $L \in I_r, M \in I_s, r, n \leq k $,
\beqn
|( \mathcal{M}_{L,M} )^{\text{meas}} (\lambda, D_{\text{exact}} ) -  \mathcal{M}_{L,M} (\lambda, D_{\text{exact}} ) | < \text{err} \,.
\label{error}
\eqn
Here, we refer to \cite{noise,phaseless,book,resol1} for more relevant details on the choice of $k$. Write $d_k := |\{ (L,M) : L \in I_r, M \in I_s, r, n \leq k  \} |$.  

We consider the following minimization problem (without regularization for simplicity) for any given shape $D$:
\beqn
J(D) &:= &\frac{1}{2} \sum_{L \in I_r, M \in I_s, r, n \leq k }  |  \mathcal{M}_{L,M} (\lambda, D )  - ( \mathcal{M}_{L,M} )^{\text{meas}} (\lambda, D_{\text{exact}} ) |^2  \notag \\
&:=& \frac{1}{2} \|  \mathcal{M} (\lambda, D )  - \mathcal{M}^{\text{meas}} (\lambda, D_{\text{exact}} ) \|^2_{\mathbb{C}^{d_k}} \,.
\label{optimization}
\eqn
For a given shape $D^n$, we write $ \left ( \mathcal{M}^{(1)}_{L,M} (\lambda, D^n , h )  \right)_{L \in I_r, M \in I_s, r, s  \in \mathbb{N} }    =  A_n [ H^n(x) h(x) ] $,  
where $H^n(x)$ denotes the mean curvature of $D^n$. We notice that $v \mapsto A_n [v]$ can be computed explicitly for the numerical purpose either via first a division of $H^n(x)$ and then applying $\mathcal{M}^{(1)}_{L,M} (\lambda, D^n , \cdot) $ from \eqref{seriesvariation} and \eqref{variationvariation} or directly taking a difference quotient as 
\[
A_n [v] \approx \frac{1}{\tilde{\epsilon}} \left[ \mathcal{M} (\lambda, ( D^n )^{\tilde{\epsilon} } ) - \mathcal{M} (\lambda, D^n ) \right], 
\]
where $( \partial D^n )^{\tilde{\epsilon} } := \{ x + \tilde{\epsilon} \frac{v(x)}{H^n(x)} \nu(x) : x \in \partial D^n \}$ for a small $\tilde{\epsilon}$.  With this definition of $A_n$, we now consider the following preconditioned Gauss-Newton descent iterative sequence:
\begin{equation}
\begin{split}
v^n \in &\  \mathrm{argmin}_{v \in V_{k,\partial D^n} } \, \frac{1}{2} \| A_n v -  \mathcal{M}^{\text{meas}} (\lambda, D_{\text{exact}} ) + \mathcal{M} (\lambda, D^n )   \|^2_{\mathbb{C}^{d_k}}  \, , \\
D^{n+1} :=& \left \{ x + \alpha_n \frac{{\color{black} v^n}(x)}{H^n(x)} \nu^n(x) \, : \, x \in D^n  \right \} \, ,
\end{split}
\label{iteration}
\end{equation}
where $\alpha_n$ can be chosen via any appropriate means, say a line searching \cite{Source}. Here, it is remarked that one may instead consider a regularized version, say the Levenberg-Marquardt algorithm \cite{Newton1, Newton2} with the regularization parameter following a discrepancy principle \cite{Source}, but for simplicity, let us only consider \eqref{iteration}. The computation of $v^n$ in the first step in \eqref{iteration} can be done via a descent-type method or a conjugate gradient method restricted on the subspace $V_{k,\partial D^n}$. To proceed, let us denote
\[
A_{n,k} :=  A_n |_{   \mathcal{L} ( V_{k, \partial D^n } , \mathbb{C}^{d_k} )   } \,.
\]
Suppose $D_{\text{exact}}$ and $D^n$ are both $\mathcal{C}^1$ $\varepsilon$-perturbations of $R_0 \mathbb{S}^{d-1}$ for $\varepsilon < \varepsilon_0$ in such a way that the perturbation directions are also in $\mathcal{C}^{3,\alpha}$. Then from Corollary \ref{cor:1} and the equivalence of norms between the finite dimensional spaces $\mathbb{C}^{d_k}$ and $\mathcal{L} (V_{k, \mathbb{S}^{d-1} }, V_{k, \mathbb{S}^{d-1} } )$, we have
\[
 \|  A_{n,k} \|  \geq C_{\varepsilon_0} \frac{ C_{d} }{ d_k} \left( R_0 + \frac{1}{R_0} \right)^{-2k}   := \frac{ C_{\varepsilon_0,d} }{ d_k}  \left( R_0 + \frac{1}{R_0} \right)^{-2k}  \,.
\]
Therefore $A_{n,k} $ is invertible, and moreover the condition number of $A_{n,k} $ is independent of the choice of $D^n$'s and their curvature.
The minimization in the linearized step in \eqref{iteration} is numerically stable to compute if $k$ is not huge, and this is uniform across the iteration steps. 
{ \color{black}
Further discussion of our choice of the Gauss-Newton method, e.g. the descent property of the Gauss-Newton direction, is deferred to the Appendix.}
Instead of the projection to $V_{k,\partial D^n}$ as in \eqref{iteration}, more sophisticated regularization choices of the Gauss-Newton method can be introduced, and their convergence (and their optimal convergence rate) under a source condition are discussed in e.g. \cite{Source}. Detailed numerical studies of the method can be found in e.g. \cite{numerical_Gauss_Newton}.

\end{Remark}

}

\begin{Remark}
We remark that the mechanism of detecting geometric singularities in \cite{Maj} is of high frequency nature, while our analysis is in the low frequency regime.  
High resolution boundary information of the inclusion only enters scattered fields as high order GPTs or SCs as will be shown in the next section, which decay exponentially.   This is consistent with the exponential ill-posedness results \cite{curv_Liu_2,CR,LRX,LPRX}, that high resolution information is more prone to the noise contamination. 
However, if we have large curvature points, higher order GPTs and SCs, albeit still exponentially decaying, will be pushed up to a high magnitude.  If a perturbation is further applied, fine details of the perturbation near a high curvature point will be amplified in the far field and easily reconstructable (c.f. Theorem \ref{all_GPT} and Corollary \ref{cor:1}.)

We can describe the above understanding using an example in two dimensions. For instance, we may take a shape $D_1$ coming as an $\varepsilon_0$-perturbation of a disk $D_0$ along a direction $\tilde{h} = h/ \| h\|$,
where
\[
h := \sum_{k=-N}^N C^{|k|} e^{i k \theta} = 1 + 2 \frac{ 1 - C \cos (\theta) + C^{N+1} \cos ( (N-1) \theta) - C^{N} \cos (N \theta) } {1 + C^2} ,
\]
for a fixed small $\varepsilon_0$ and a fixed large $N$ where $C > 1$, which is highly irregular around the point zero if $N$ is very large.  As we expect, the resulting object has the property that the corresponding GPT with order $(m,n)$ has a decaying order of $\delta_{mn} (R/C)^{2n} + \varepsilon_0 (R/C)^m (R/C)^n + O( \varepsilon_0^2) $. We see that if $C$ is comparable with $R$, the decay is less rapid (up till order $N$), and hence the high frequency information of the boundary inclusion enters the scattered field more stably.  
If we further perturb $D_1$ to $D_2$ as a $\delta$-perturbation of $D_1$, with the fact that the geometry of $D_1$ has already made GPT with order $(m,n)$ where $|m|,|n| < N$ have a relatively larger magnitude, the scattered measurements projected onto these components are no longer hindered by noise and propagate to the far field.  Therefore, the perturbations from $D_1$ to $D_2$ can be better detected from the far-field measurements.
\end{Remark}

\section{Transmission Helmholtz problem in the low frequency regime and its sensitivity analysis}  \label{sec3}

In this section, we consider the Helmholtz transmission problem. {\color{black} In what follows, for given positive constants $\epsilon_0, \mu_0, \epsilon_1 , \mu_1, \varpi$, we set $k_0:= \varpi \sqrt{ \epsilon_0 \mu_0} $ and $k_1:=\varpi \sqrt{ \epsilon_1 \mu_1} $. The transmission problem is given as follows for $u\in H_{loc}^1(\mathbb{R}^d)$:}
\beqn
    \begin{cases}
        \nabla \cdot (\frac{1}{\mu_{D} } \nabla u) + {\color{black}\varpi^2} \epsilon_{D} u = 0 &\text{ in }\; \mathbb{R}^d, \\[1.5mm]
         (\frac{\partial}{\partial |x| } - i k_0 ) (u - u_0) = o(|x|^{ - \frac{d-1}{2}}) &\text{ as }\; |x| \rightarrow \infty,
    \end{cases}
    \label{transmissionk}
\eqn
where 
$\mu_{D}  = \mu_1 \chi(D) +  \mu_0 \chi(\mathbb{R}^d \backslash \overline{D})$, 
$\color{black}\epsilon_{D}  = \epsilon_1 \chi(D) +  \epsilon_0 \chi(\mathbb{R}^d \backslash \overline{D})$
and $u_0 $ satisfies $( \Delta + k_0^2 ) = 0$. In \eqref{transmissionk}, $\varpi\in\mathbb{R}_+$ denotes the operating frequency. Throughout the rest of our study, we consider the case that $\varpi \ll 1$, or equivalently $k_0\ll 1$. This is refered to as the quasi-static regime. It is in fact an important regime regarding the fact that when $\varpi$ is small, the resolution of the corresponding inverse problem {\color{black} of
recovering the shape of the inclusion from the scattering data} is considerably poor {\color{black} (cf. \cite{resol1,ammari2004reconstruction})}.

{\color{black}
In what follows, we present the description of the general potential theory and the Neumann-Poincar\'e operator (cf. \cite{CK_book,regularity5,kellog, folland, mcowen,shapiro}).}
For a given ${\tilde{k}}\in\mathbb{R}_+$, we introduce
\beqn
    \mathcal{S}^{\tilde{k}}_{\partial D} [\phi] (x) &:=&  \int_{\partial D} G_{\tilde{k}}(x-y) \phi(y) d \sigma(y) ,\\
  \mathcal{D}^{\tilde{k}}_{\partial D} [\phi] (x) &:=& \int_{\partial D} \frac{\partial }{\partial \nu_y } G_{\tilde{k}}(x-y) \phi(y) d \sigma(y) ,
\eqn
for $x \in \mathbb{R}^d\backslash\partial D$ with $d \geq 2$, where $G_{\tilde{k}}$ is the fundamental solution of the Helmholtz equation with outgoing radiation condition in $\mathbb{R}^d$ as follows:
\beqn
    G_{\tilde{k}} (x-y) = C_{{\tilde{k}},d} ( {\tilde{k}}|x-y| )^{- \frac{d-2}{2}} H^{(1)}_{\frac{d-2}{2}}({\tilde{k}}|x-y|) ,
    \label{fundamental2}
\eqn
with $C_{{\tilde{k}},d}$ some constant depending only on $d$, and $H^{(1)}_{\frac{d-2}{2}}$ is the Hankel function of the first kind and order $(d-2)/2$.  
It is known that the single-layer potential $\mathcal{S}^{\tilde{k}}_{\partial D}$ satisfies the following jump condition on $\partial D$:
\beqn
    \f{\p}{\p \nu} \left(  \mathcal{S}^{\tilde{k}}_{\partial D} [\phi] \right)^{\pm} = (\pm \f{1}{2} I + {\mathcal{K}^{\tilde{k}}_{\partial D}}^* )[\phi]\,,
    \label{jump_condition2}
\eqn
where the superscripts $\pm$ indicate the limits from outside and inside $D$ respectively, and
${\mathcal{K}^{\tilde{k}}_{\partial D}}^*: L^2(\partial D) \rightarrow L^2(\partial D)$ is the Neumann-Poincar\'e operator defined by
\beqn
    {\mathcal{K}^{\tilde{k}}_{\partial D}}^* [\phi] (x) := \int_{\partial D} \partial_{\nu_x} G_{\tilde{k}}(x-y) \phi(y) d \sigma(y) \, .
    \label{operatorK2}
\eqn

With the above preparations, $u\in H_{loc}^1(\mathbb{R}^d)$ in \eqref{transmissionk} can be given by
\beqnx
u = 
\begin{cases}
u_0 + \mathcal{S}^{k_0}_{\partial D} [\psi] & \text{ on } \mathbb{R}^d \backslash \overline{D}, \\
\mathcal{S}^{k_1}_{\partial D} [\phi]  & \text{ on }D,
\end{cases}
\eqnx
where $(\phi,\psi) \in L^2 (\partial D) \times L^2 (\partial D)$ is the unique solution to (provided that $k_1^2$ is not a Dirichlet eigenvalue of the Laplacian in $D$)
\beqnx
\begin{cases}
\mathcal{S}^{k_1}_{\partial D} [\phi]  - \mathcal{S}^{k_0}_{\partial D} [\psi]  = u_0 \\
\frac{1}{\mu_1}( - \f{1}{2} I + {\mathcal{K}^{k_1}_{\partial D}}^* )[\phi] - \frac{1}{\mu_0}( \f{1}{2} I + {\mathcal{K}^{k_0}_{\partial D}}^* ) [\psi] =  \frac{1}{\mu_0}   \f{\p u_0}{\p \nu}   \,,
\end{cases}
\eqnx
or that
\beqnx
& &\left\{   \f{1}{2}  \left( \frac{1}{\mu_0} I + \frac{1}{\mu_1} \, \left( \mathcal{S}^{k_1}_{\partial D} \right)^{-1}  \mathcal{S}^{k_0}_{\partial D} \right)
+
\frac{1}{\mu_0} {\mathcal{K}^{k_0}_{\partial D}}^* 
- \frac{1}{\mu_1} {\mathcal{K}^{k_1}_{\partial D}}^*   \left( \mathcal{S}^{k_1}_{\partial D} \right)^{-1}  \mathcal{S}^{k_0}_{\partial D}
 \right\} [\psi] \\
&=&    \frac{1}{\mu_1}( - \f{1}{2} I + {\mathcal{K}^{k_1}_{\partial D}}^* ) \circ \left(\mathcal{S}^{k_1}_{\partial D} \right)^{-1} \left[ u_0 \right]  -  \frac{1}{\mu_0}  \f{\p u_0}{\p \nu}   \\
&=&   \left( \frac{1}{\mu_1} -  \frac{1}{\mu_0}  \right) \f{\p u_0}{\p \nu}  \, .
\eqnx
{\color{black}
where the inversion of $\mathcal{S}^{k_1}_{\partial D}$ is understood as the inversion of the operator $\mathcal{S}^{k_1}_{\partial D}: L^2(\partial D, d \sigma) \rightarrow L^2(\partial D, d \sigma)$ instead of $\mathcal{S}^{k_1}_{\partial D}: L^2(\partial D, d \sigma) \rightarrow H^1_{loc}(\mathbb{R}^d)$,}
As in \cite{new_exp}, we can now write, for all $x$ outside $D$
\begin{equation}\label{eq:rep1}
\begin{split}
u  - u_0 = & \left( \frac{1}{\mu_1} -  \frac{1}{\mu_0}  \right) \mathcal{S}^{k_0}_{\partial D} \circ \bigg\{   \f{1}{2}  \left( \frac{1}{\mu_0} I + \frac{1}{\mu_1} \, \left( \mathcal{S}^{k_1}_{\partial D} \right)^{-1}  \mathcal{S}^{k_0}_{\partial D} \right)\\
&+
\frac{1}{\mu_0} {\mathcal{K}^{k0}_{\partial D}}^* 
- \frac{1}{\mu_1} {\mathcal{K}^{k_1}_{\partial D}}^*   \left( \mathcal{S}^{k_1}_{\partial D} \right)^{-1}  \mathcal{S}^{k_0}_{\partial D}
 \bigg\}^{-1}  \left[ \f{\p u_0}{\p \nu}   \right ]
 \end{split}
\end{equation}
where the inverse in the equation exists by the Fredholm alternative theorem, {\color{black} and the first operator $ \mathcal{S}^{k_0}_{\partial D} $ in \eqref{eq:rep1} is understood as $\mathcal{S}^{k_0}_{\partial D}: L^2(\partial D, d \sigma) \rightarrow H^1_{loc}(\mathbb{R}^d \backslash D)$.}

From the following asymptotics as $z\rightarrow+0$ {\color{black}(cf. \cite{Kor02})},
\begin{equation}\label{eq:as1}
J_{\alpha}(z) = \frac{1}{\Gamma(\alpha +1)} \left(\frac{z}{2}\right)^\alpha+ O(z^{\alpha+2})  \quad \text{ and }
\quad
Y_{\alpha}(z) =
 \begin{cases}
\frac{2}{\pi} \left(  \log \left( \frac{z}{2} \right) + \gamma \right) + O(z)   \\
 \frac{\Gamma(\alpha +1)}{\pi} \left(\frac{z}{2}\right)^{-\alpha}+ O(z^{ - \alpha+2}) 
\end{cases},
\end{equation}
{\color{black}
where $J_\alpha$ is the Bessel function of first kind and $Y_\alpha$ is the Bessel function of second kind,
}we have
\beqnx
G_{\tilde{k}} (x-y) =
\begin{cases}
C_{{\tilde{k}}}  \left[ C_1  \log \left(  {\tilde{k}}|x-y| \right)  + C_2 + O\left({\tilde{k}}|x-y|\right) \right]  & \quad \text{ if } d = 2,\medskip\\
C_{{\tilde{k}},d} \left[ {\tilde{k}}^{2-d}  |x-y|^{2-d}  + O\left( {\tilde{k}}^{4 -d} |x-y|^{4 -d} \right) \right] & \quad \text{ if } d > 2 \, ,
\end{cases}
\eqnx
where $C_d$ only depends on $d$ and $C_2$ is another constant.   Since $\mathcal{S}_{\partial D} $ and $ {\mathcal{K}_{\partial D}}^*$ are both of order $-1$, we have the following lemma as in \cite{new_exp} together with tracing back the regularity of the coefficients in the operators with \eqref{regularity}.
\begin{Lemma} \label{quick1}
We have the following decompositions for the boundary potential operators, 
\beqnx
    \mathcal{S}^{\tilde{k}}_{\partial D} =  \mathcal{S}_{\partial D}  + \varpi^2 \, \mathcal{S}^{\tilde{k}}_{\partial D,-3}  \quad \text{ and }\quad 
    {\mathcal{K}^{\tilde{k}}_{\partial D}}^*  =  {\mathcal{K}_{\partial D}}^* +  \varpi^2  \,  \mathcal{K}^{\tilde{k}}_{\partial D,-3}  \,,
\eqnx
where $ \mathcal{K}^{\tilde{k}}_{\partial D,-3} , \mathcal{S}^{\tilde{k}}_{\partial D,-3} $ are uniformly bounded w.r.t. $\varpi$ and are of order  $-3$, and in the sense of $ \mathcal{C}^{k-3,\alpha} \,  \Phi \text{SO}^{-3} $ when $\partial D \in \mathcal{C}^{k,\alpha}$ for $k \geq 4$.
\end{Lemma}

Next, by following the same arguments as those for establishing Theorem~\ref{theoremK}, one can perturb $\partial D$ along the normal direction $\nu$, which in turn gives the perturbations of the boundary potential operators, being pseudo-differential operators with one order higher than the respective original operators.  We actually have the following result. 

\begin{Corollary}  \label{quick2}
For all ${\tilde{k}} \in\mathbb{R}_+$, $N \in \mathbb{N}$, there exists a constant $C$ depending only on $N$, $||\mathbb{X}||_{\mathcal{C}^2}$ and $||h||_{\mathcal{C}^1}$ such that
the following estimate holds for any $\widetilde{\phi} \in L^2 (\partial D_\varepsilon)$ and
$\phi := \tilde{\phi}\circ \Psi_\varepsilon$ {\color{black}(it is recalled that $\Psi_\varepsilon$ is the diffeomorphism taking $\partial D$ to $\partial D_\varepsilon$)}:
\beqnx
   \bigg|\bigg|  \mathcal{S}^{\tilde{k}}_{\partial D_{\varepsilon} } [ \tilde{\phi} \circ \Psi_\varepsilon ]  - \mathcal{S}^{\tilde{k}}_{\partial D}[{\phi} ] - \sum_{n=1}^N \varepsilon^n  \left(\mathcal{S}^{\tilde{k}}\right) ^{(n)}_{D,h} [\phi] \bigg|\bigg|_{L^2(\partial D)}
   \leq C \varepsilon^{N+1} ||\phi||_{L^2(\partial D)}\, ,\\
   \bigg|\bigg|  {\mathcal{K}^{\tilde{k}}_{\partial D_\varepsilon}}^*[ \tilde{\phi} \circ \Psi_\varepsilon ]  -  {\mathcal{K}^{\tilde{k}}_{\partial D}}^* [{\phi} ] - \sum_{n=1}^N \varepsilon^n  \left(\mathcal{K}^{\tilde{k}} \right) ^{(n)}_{D,h} [\phi] \bigg|\bigg|_{L^2(\partial D)}
   \leq C \varepsilon^{N+1} ||\phi||_{L^2(\partial D)}\,,
\eqnx
with
\beqnx
\left(\mathcal{S}^{\tilde{k}}\right) ^{(1)}_{D,h} = \left(\mathcal{S}^0\right) ^{(1)}_{D,h} + \varpi^2 \left(\mathcal{S}^{\tilde{k}}\right) ^{(1)}_{D,h,-2}  \quad \text{ and }\quad
\left(\mathcal{K}^{\tilde{k}}\right) ^{(1)}_{D,h} = \left(\mathcal{K}\right) ^{(1)}_{D,h} + \varpi^2 \left(\mathcal{K}^{\tilde{k}}\right) ^{(1)}_{D,h,-2}  \, ,
\eqnx
where $ \left(\mathcal{K}^{\tilde{k}}\right) ^{(1)}_{D,h,-2}  , \left(\mathcal{S}^{\tilde{k}}\right) ^{(1)}_{D,h,-3} $ are uniformly bounded w.r.t. $\varpi$ and are of order $-2$, and in the sense of $ \mathcal{C}^{k-3,\alpha} \,  \Phi \text{SO}^{-2} $ when $\partial D \in \mathcal{C}^{k,\alpha}$ and $h \in \mathcal{C}^{k-1,\alpha}$ when $k \geq 4$.
\end{Corollary}

Moreover similar to the argument in Section \ref{section_D_N}, since $ u_0 $ satisfies $(\Delta + k_0^2) u_0 = 0$, the mapping $u_0 \mapsto \f{\p u_0}{\p \nu} $ can be viewed as the Dirichlet to Neumann map with respect to the operator $\Delta + k_0^2$, which we denote as $\Lambda_{k_0}$. 
{\color{black} We recall the factorization provided in Theorem 3.1 in \cite{uhlmann}: for a general potential $q$, $- \Delta + q$ can be factorized as
\[
- \Delta + q = \left(i \partial_\nu + i E(x) + i B_q(x, i \nabla ) \right) \left(i \partial_\nu - i B_q(x, i  \nabla ) \right),
\]
{\color{black}
where $B_q(x, i  \nabla ) $ is a pseudo differential operator with its symbol recursively defined as in (3.11)-(3.14) in \cite{uhlmann},}
and when $\partial D \in \mathcal{C}^{\infty}$,
\[
\Lambda_{q} = B_q(x, i  \nabla ) |_{\partial D}  \, (  \text{mod} \, \Phi \text{SO}^{-\infty} ),
\]
where it is recalled that $ \text{mod} \, \Phi \text{SO}^{-\infty} $ means $ \text{mod} \, \Phi \text{SO}^{-m} $ for all $m \in \mathbb{N}$. From this, one can show that $\Lambda_{k_0}$ is a pseudo-differential operator of order $1$ with its symbol in the geodesic normal coordinate given by
\beqnx
p_{\Lambda_{k_0}}(x,\xi) 
=
p_{\Lambda_{0}}(x,\xi)  - \frac{1}{2} k_0^2 \, |\xi|^{-1}   +  O\left( |\xi|^{-2}\right) \, .
\eqnx
{\color{black}
where $p_{\Lambda_{0}}(x,\xi) $ is given by \eqref{eq:s1} from the recursive formulas (3.11)-(3.14) in \cite{uhlmann}.
}
{\color{black} If $\partial D \in \mathcal{C}^{k,\alpha}$, we then have $\Lambda_{q} = B(x, i  \nabla ) |_{\partial D}  \, (  \text{mod} \, \mathcal{C}^{k-4,\alpha} \Phi \text{SO}^{-2} )$ and the above equality holds with $O( |\xi|^{-2})$ now representing a symbol of $\mathcal{C}^{k-4,\alpha} S^{-2}$. Comparing this symbol with $p_{\Lambda_{0}}(x,\xi) $ in \eqref{eq:s1}, we can obtain the following lemma.} }
%
%

\begin{Lemma}  \label{quick3}
The following decomposition on $\partial D \in \mathcal{C}^{k,\alpha}$ holds:  
\beqnx
   \Lambda_{k_0}  = \Lambda_{0}  +  \varpi^2  \,  \Lambda_{k_0,-1}  \,,
\eqnx
where $\Lambda_{k_0,-1} $ is uniformly bounded w.r.t. $\varpi$ and is a pseudo-differntial operator of order $-1$ (in a sense of $\mathcal{C}^{k-3,\alpha} S^{-1}$) with its symbol given by
\beqnx
p_{\Lambda_{k_0,-1}}(x,\xi)  = - \frac{1}{2} \epsilon_0 \mu_0 \, |\xi|^{-1}  +  O\left( |\xi|^{-2}\right)  \,.
\eqnx
\end{Lemma}
Similarly, when we consider the perturbation of $ \Lambda_{k_0}$ with respect to $h$ along the normal direction of $\partial D$, we may consider the decomposition \eqref{decompsition} and follow Subsection \ref{section_pert_D_N} to obtain the perturbation of $ \Lambda_{k_0}$ in terms of $ (d-1) H(x) h(x) \Lambda_{k_0}$ and $ h(x) \Delta_{\partial D}$, but now with an additional term $h(x) k_0^2$ on the boundary.  Therefore, we readily obtain the following:

\begin{Corollary}
For all ${\tilde{k}} \in\mathbb{R}_+$, $N \in \mathbb{N}$, there exists a constant $C$ depending only on $N$, $||\mathbb{X}||_{\mathcal{C}^2}$ and $||h||_{\mathcal{C}^1}$ such that
the following holds for any $\widetilde{\phi} \in H^{\frac{1}{2}} (\partial D_\varepsilon)$ and $\phi := \tilde{\phi}\circ \Psi_\varepsilon$:
\beqnx
   \bigg|\bigg| \Lambda_{k_0, D_{\varepsilon}}  [ \tilde{\phi} \circ \Psi_\varepsilon ]  - \Lambda_{k_0, D}  [{\phi} ] - \sum_{n=1}^N \varepsilon^n 
\Lambda^{(n)}_{k_0, D,h}  [\phi] \bigg|\bigg|_{H^{\frac{1}{2}}(\partial D)}
   \leq C \varepsilon^{N+1} ||\phi||_{L^2 (\partial D)}\, ,
\eqnx
where
\beqnx
\Lambda^{(1)}_{k_0, D,h}  = \Lambda^{(1)}_{0, D,h}  +  \varpi^2 \Lambda^{(1)}_{k_0, D,h,0} \,.
\eqnx
Here $\Lambda^{(1)}_{k_0, D,h,0} $ is uniformly bounded w.r.t. $\varpi$ and is a pseudo-differential operator of order $0$ in the sense of $\mathcal{C}^{k-2,\alpha} \Phi \text{SO}^{0}$ when $\partial D \in \mathcal{C}^{k,\alpha}, h \in \mathcal{C}^{k-1,\alpha}$ and $k \geq 4$ with
\beqnx
p_{\Lambda^{(1)}_{k_0, D,h,0} }(x,\xi)  = -  \epsilon_0 \mu_0 \, h(x)  +  O\left( |\xi|^{-1}\right)  \,.
\eqnx
\end{Corollary}

We would like to remark that {\color{black} the shape derivative of $\Lambda_{k_0, D} $ with respect to the variation direction $h \in \mathcal{C}^1 $, i.e. $\Lambda^{(1)}_{k_0, D,h} $,} contains the term $\Lambda^{(1)}_{k_0, D,h,0}$ of order $0$ instead of $-1$, in contrast to what one may have expected.

\subsection{Scattering coefficients in arbitrary dimensions}

With the above preparations, we have all the tools to analyse the scattered field in \eqref{transmissionk} in terms of the scattering coefficients as defined in \cite{homoscattering, heteroscattering}, which we shall first extend to arbitrary dimensions in what follows. 

{\color{black}
Similar to the derivation in Subsection \ref{subsection:harmonic}, from an analogous proof of the Graf's addition formula \cite{watson}, as in \cite{handbook,stein}, one can utilize the generating function and the reproducing kernel property of the Gegenbauer's polynomial to obtain the following Gegenbauer's addition formula (c.f. Equation (4) in \cite{watson}, P. 365 ) :
\beqnx
& &( k_0 |x-y| )^{- \frac{d-2}{2}} H^{(1)}_{\frac{d-2}{2}}(k_0 |x-y|)   \\
&=&   \sum_{k=0}^{\infty}  ( k_0|x|)^{- \frac{d-2}{2} +k } H^{(1)}_{\frac{d-2}{2} + k }(k_0 |x|)  ( k_0 |y| )^{- \frac{d-2}{2} +k} J_{\frac{d-2}{2} + k }(k_0 |y|) \sum_{ L \in I_k} Y_{L} (\omega_x) \, \overline{ Y_{L} (\omega_y) }
\eqnx
for $|x| > |y|$. }
From that, the scattering coefficients of a $\partial D \in \mathcal{C}^{2,\alpha}$ domain can then be defined by putting $ u_0 = ( k r )^{- \frac{d-2}{2} + k} J_{\frac{d-2}{2} + k }(k r)  Y_{L} (\omega) $ which satisfies $(\Delta + k_0^2) u_0 = 0$, and then taking the coefficient with respect to the function $Y_{L} (\omega_x)$, i.e. as follows:
\begin{Definition}
The scattering coefficients associated with the scattered field $u-u_0$ given by \eqref{transmissionk} with given $\mu_0,\mu_1,\epsilon_0,\epsilon_1, \varpi$ and a domain $D \subset \mathbb{R^d}$ with a $\mathcal{C}^{2,\alpha}$ boundary are defined by
\beqnx
& &
\mathcal{W}_{L, M} ( \mu_0,\mu_1,\epsilon_0,\epsilon_1 , D) \\
&:=&  \int_{\partial D}  \, 
 ( k_0 |y| )^{- \frac{d-2}{2} + k } J_{\frac{d-2}{2} + k }(k_0 |y|) \, \overline{ Y_{L} (\omega_y) } \bigg\{
 \left( \frac{1}{\mu_1} -  \frac{1}{\mu_0}  \right) \bigg\{   \f{1}{2}  \left( \frac{1}{\mu_0} I + \frac{1}{\mu_1} \, \left( \mathcal{S}^{k_1}_{\partial D} \right)^{-1}  \mathcal{S}^{k_0}_{\partial D} \right) \\
 && +
\frac{1}{\mu_0} {\mathcal{K}^{k0}_{\partial D}}^* 
- \frac{1}{\mu_1} {\mathcal{K}^{k_1}_{\partial D}}^*   \left( \mathcal{S}^{k_1}_{\partial D} \right)^{-1}  \mathcal{S}^{k_0}_{\partial D}
 \bigg\}^{-1}  \circ \Lambda_{k_0} \bigg\}\\
& & \quad \quad \quad \quad \times \left[ ( k_0 r )^{- \frac{d-2}{2} + n} J_{\frac{d-2}{2} + n }(k_0 r)  Y_{M} (\omega)  \right] (y) \, d \sigma(y)  \, ,
\eqnx
where $k_0 = \varpi \sqrt{ \mu_0 \epsilon_0 }$ and $k_1 = \varpi \sqrt{ \mu_1 \epsilon_1 }$.
\end{Definition}
By direct calculations, one can establish the following lemma. 
\begin{Lemma}
Let $D \subset \mathbb{R}^d$ be a bounded domain with a $\mathcal{C}^{2,\alpha}$ boundary. Consider the solution to \eqref{transmissionk} with 
$$
u_0(x) = \sum_{k=0}^{\infty}  \sum_{ L \in I_k }   a_{L} \,   ( k_0 r_x )^{- \frac{d-2}{2} + k} J_{\frac{d-2}{2} + k }(k_0 r_x)  Y_{L} (\omega_x)  \,. 
$$ 
Then the scattered field for $|x| > \sup\{ |x| : x \in D \} $ is given by
\begin{equation}\label{eq:sf1}
\begin{split}
 &(u - u_0) (x)\\
 = \, & C_{k_0,d} \sum_{k=0}^{\infty}  \sum_{L \in I_k} 
\sum_{n=0}^{\infty}  \sum_{ M \in I_n}  a_{M} 
 \, (  k_0 |x|)^{- \frac{d-2}{2} + k } H^{(1)}_{\frac{d-2}{2} + k }(k_0|x|)  Y_{L} (\omega_x) \, \mathcal{W}_{L, M} ( \mu_0,\mu_1,\epsilon_0,\epsilon_1 , D) \,.
 \end{split}
\end{equation}
Hence, for $x \in R \, \mathbb{S}^{d-1}$ where $R > \sup\{ |x| : x \in D  \} $, we have
\begin{equation}\label{eq:sf2}
\begin{split}
  & \mathcal{W}_{L, M} ( \mu_0,\mu_1,\epsilon_0,\epsilon_1 , D) \\
  = \, &   \frac{1}{ C_{k_0,d}  }    (  k_0 |x|)^{ \frac{d-2}{2} - k } \frac{1}{H^{(1)}_{\frac{d-2}{2} + k }(k_0 R) }  \int_{\mathbb{S}^{d-1}}  \overline{ Y_L(\omega_x) }  \bigg(u - ( k_0 r )^{- \frac{d-2}{2} + n} J_{\frac{d-2}{2} + n }(k_0 r)  Y_{M} (\omega)    \bigg) (R \, \omega_x )  d \omega_{x} \,.
\end{split}
\end{equation}
\end{Lemma}

We note one important property in force when $\partial D \in \mathcal{C}^{k,\alpha}$ for $k \geq 4$:
\[
\begin{split}
& \left( \frac{1}{\mu_1} -  \frac{1}{\mu_0}  \right)  \left\{   \f{1}{2}  \left( \frac{1}{\mu_0} I + \frac{1}{\mu_1} \, \left( \mathcal{S}^{k_1}_{\partial D} \right)^{-1}  \mathcal{S}^{k_0}_{\partial D} \right)
+
\frac{1}{\mu_0} {\mathcal{K}^{k0}_{\partial D}}^* 
- \frac{1}{\mu_1} {\mathcal{K}^{k_1}_{\partial D}}^*   \left( \mathcal{S}^{k_1}_{\partial D} \right)^{-1}  \mathcal{S}^{k_0}_{\partial D}
 \right\}^{-1}  \circ \Lambda_{k_0} ( u_0 ) \\
&= 
  \left\{  \lambda  I
- {\mathcal{K}_{\partial D}}^*   
 \right\}^{-1}   \circ \Lambda_{k_0} (u_0) +  \varpi^2 \, R_{\mu_0,\mu_1,\epsilon_0,\epsilon_1,\varpi,\partial D, -1}  ( u_0 ) ,
\end{split}
\]
where $\lambda = \frac{\mu_0 + \mu_1}{2( \mu_0 - \mu_1 ) }  $ which is a similar equation as that in Section 2, modulus $ \varpi^2 \, R_{\partial D, -1}  ( u_0 )$ for a certain operator $R_{\partial D, -1} $ uniformly bounded with respect to $\varpi$ and of order $-1$, in the sense of $ \mathcal{C}^{k-3,\alpha} \Phi \text{PO}^{-1}$ tracing back the regularity of coefficients. 
Therefore we obtain:
\begin{Theorem} \label{approx_int}
For $\partial D \in \mathcal{C}^{k,\alpha}, k \geq 4$, we have
\beqnx
u  - u_0 =  \mathcal{S}^{k_0}_{\partial D} \circ \left(    \left\{  \lambda  I
- {\mathcal{K}_{\partial D}}^*   
 \right\}^{-1}   \circ \Lambda_{k_0} (u_0) +  \varpi^2 \, R_{\mu_0,\mu_1,\epsilon_0,\epsilon_1,\varpi,\partial D, -1}   ( u_0 )  \right),
\eqnx
where $R_{\partial D, -1} $ is uniformly bounded with respect to $\varpi$ and is of order $-1$ in the sense of $ \mathcal{C}^{k-3,\alpha} \Phi \text{PO}^{-1}$. 
\end{Theorem}
Using the asymptotic properties of $J_{\alpha}$ and $Y_{\alpha}$ in \eqref{eq:as1}, we can further obtain the following corollary by straightforward calculations.
\begin{Lemma}
We have
\beqnx
\mathcal{W}_{L, M} ( \mu_0,\mu_1,\epsilon_0,\epsilon_1 , D) = \mathcal{M}_{L, M} ( \lambda , D) + O(\varpi^2)\, ,
\eqnx
where $\lambda = \frac{\mu_0 + \mu_1}{2( \mu_0 - \mu_1 ) }  $.
\end{Lemma}

\subsection{Sensitivity analysis of the scattering coefficients in the low frequency regime}

Combining Lemmas \ref{approx_int} and \ref{haha1}, and tracing back regularity \eqref{regularity}, we readily obtain:

\begin{Lemma} \label{haha11}
{\color{black} For $d >2 $,} and $\partial D \in \mathcal{C}^{k,\alpha}, k \geq 4$, the scattering coefficient $\mathcal{W}_{L, M} ( \mu_0,\mu_1,\epsilon_0,\epsilon_1 , D)  $ has the following representation
\begin{equation}\label{eq:re1}
\begin{split}
 &\mathcal{W}_{L, M} ( \mu_0,\mu_1,\epsilon_0,\epsilon_1 , D)  \\
=&
\bigg \langle    ( k_0 r )^{- \frac{d-2}{2} + k } J_{\frac{d-2}{2} + k }(k_0 r)  Y_{L} (\omega) \,,  \Big( P_{D,1} + P_{D,0} + P_{D,-1} \\
&
\qquad\qquad+ \varpi^2 \, R_{\mu_0,\mu_1,\epsilon_0,\epsilon_1,\varpi,\partial D, -1} \Big)\Big(  ( k_0 r )^{- \frac{d-2}{2} + n } J_{\frac{d-2}{2} + n }(k_0 r)  Y_{M} (\omega)  \Big)  \bigg\rangle_{L^2 (\partial D, d \sigma )},
\end{split}
\end{equation}
where $P_{D,m}$ are pseudo-differential operators of order $m$ for $m = 1,0,-1$ as given in Lemma \ref{haha1}, and $R_{\mu_0,\mu_1,\epsilon_0,\epsilon_1,\varpi,\partial D, -1} $ is of order $-1$ in the sense of $ \mathcal{C}^{k-3,\alpha} \Phi \text{PO}^{-1}$ and is uniformly bounded with respect to $\varpi$.
\end{Lemma}

Following the proof of Theorem \ref{lemmaDM} and utilizing Lemmas \ref{quick1}-\ref{quick3} and \ref{approx_int}, as well as tracing back regularity \eqref{regularity}, we can also obtain that

\begin{Theorem} \label{lemmaDM11}
{\color{black} For $d >2 $} and $N \in \mathbb{N}$, $\partial D \in \mathcal{C}^{k,\alpha}$ and $h \in \mathcal{C}^{k-1,\alpha}$ when $k \geq 4$, there exists a constant $C$ depending only on $N$, $L\in I_{k}, M \in I_n$, $||\mathbb{X}||_{\mathcal{C}^2}$ and $||h||_{\mathcal{C}^1}$ such that 
\beqn
\begin{split}
&\left |  \mathcal{W}_{L, M} ( \mu_0,\mu_1,\epsilon_0,\epsilon_1 , D_{\varepsilon})   - \mathcal{W}_{L, M} ( \mu_0,\mu_1,\epsilon_0,\epsilon_1 , D) 
- \sum_{n=1}^N \varepsilon^n  
\mathcal{W}_{L, M}^{(n)} ( \mu_0,\mu_1,\epsilon_0,\epsilon_1 , D,h) \right|\\
&   \leq C \varepsilon^{N+1},
   \end{split}
   \label{seriesvariation21}
\eqn
for some $\mathcal{W}_{L, M}^{(n)} ( \mu_0,\mu_1,\epsilon_0,\epsilon_1 , D,h) $ with $\mathcal{W}_{L, M}^{(1)} ( \mu_0,\mu_1,\epsilon_0,\epsilon_1 , D,h)  $ given by
\beqnx
&  &\mathcal{W}_{L, M}^{(1)} ( \mu_0,\mu_1,\epsilon_0,\epsilon_1 , D,h) \\
&=&  \left \langle  ( k_0 r )^{- \frac{d-2}{2} + k } J_{\frac{d-2}{2} + k }(k_0 r)  Y_{L} (\omega) \,,\,
 \widetilde{ Q_{\mu_0,\mu_1,\epsilon_0,\epsilon_1,\varpi, D,h} } \left(   ( k_0 r )^{- \frac{d-2}{2} + n } J_{\frac{d-2}{2} + n }(k_0 r)  Y_{M} (\omega) \right)  \right\rangle_{L^2 (\partial D, d \sigma )},
\eqnx
where
\beqnx
 \widetilde{ Q_{\mu_0,\mu_1,\epsilon_0,\epsilon_1,\varpi, D,h} }  = Q_{D,h,1,I} + Q_{D,h,1,II} + Q_{D,h,0} + \varpi^2 R_{\mu_0,\mu_1,\epsilon_0,\epsilon_1,\varpi,\partial D, 0}
\eqnx
with $Q_{D,h,1,I}, Q_{D,h,1,II}$, $Q_{D,h,0}$ being the same as those in Theorem \ref{lemmaDM}, and $R_{\mu_0,\mu_1,\epsilon_0,\epsilon_1,\varpi,\partial D, 0}$ a pseudo-differential operator of order $0$ in the sense of $ \mathcal{C}^{k-3,\alpha} \Phi \text{PO}^{0}$ and uniformly bounded with respect to $\varpi$.
\end{Theorem}

\subsection{Localization of sensitivity of scattering coefficients at points of high mean curvature}

Similar to Section \ref{sechahaha}, let us consider 
\beqnx
 \text{tr}_{\partial D} \text{Ker}( \Delta + k_0^2) :=  \{ u\mid_{\partial D} : ( \Delta + k_0^2)  u = 0 \text{ in } \mathbb{R}^d \} ,
\eqnx
where we notice that $ \overline{\text{tr}_{\partial D} \text{Ker}( \Delta +k_0 )^2 ) }^{H^{s}(\partial D, d \sigma) } = H^{s} (\partial D, d \sigma) $ is well-defined {\color{black} for all $s \in [-k,k]$ when $\partial D \in \mathcal{C}^{k,\alpha}$ \cite{regularity1,regularity2,regularity3,regularity4,regularity5}.} . Similar to the situation for the generalized polarization tensors, one has
\[
\begin{split}
& \bigg\{ \left \langle \psi  \,,\,
 \widetilde{ Q_{\mu_0,\mu_1,\epsilon_0,\epsilon_1,\varpi, D,h} }  \, \phi \right\rangle_{L^2 (\partial D, d \sigma )} \, : \, \psi \in  H^{s} (\partial D, d \sigma) , \phi \in  H^{t} (\partial D, d \sigma) ,  s , t \in [-k,k] , s + t - 1 = 0   \bigg\} \\
&=  \left\{ \lim_{n \rightarrow \infty}  \sum_{\tiny \begin{matrix} k \leq n , \\ m \leq n \end{matrix}}  \sum_{\tiny \begin{matrix} L\in I_k, \\ M \in I_m \end{matrix} } a_L b_M  \mathcal{W}_{L, M}^{(1)} ( \mu_0,\mu_1,\epsilon_0,\epsilon_1 , D,h)  \, : \, a_L, b_M  \in \mathbb{C} \text{ such that the limit exists.} \right\} \,.
\end{split}
\]
It is worth emphasizing in what follows that with suitable choices of $\psi \in  H^{s} (\partial D, d \sigma) , \phi \in  H^{t} (\partial D, d \sigma) $ such that $ s , t \in \mathbb{R}, s + t  -1 =0 $, one can obtain the principal symbol in the geodesic normal coordinate at each point $x$ as follows
\beqnx
 \lim_{t \rightarrow \infty} t^{-1} e^{- i t \varphi_{x,\xi} }   \widetilde{ Q_{\mu_0,\mu_1,\epsilon_0,\epsilon_1,\varpi, D,h} }   e^{i t \varphi_{x,\xi} } \chi_x =  p_{Q_{D,h,1,I}}(x,\xi)  + p_{Q_{D,h,1,II}}(x,\xi) \, ,
\eqnx 
where $\xi \in \mathbb{S}^{d-1}$ and $\varphi_{x,\xi} (\,\cdot\,) = \langle \xi, \log_x(\,\cdot\,) \rangle $ in half of the injectivity radius.
Let us consider the same complete orthornormal bases on $L^2(\partial D, d \sigma)$ as in Section \ref{sechahaha}, namely 
$ \{ \eta_{ k , \partial D} \}_{k \in \mathbb{N}} $, and $ \{  ( k_0 r )^{- \frac{d-2}{2} + n} J_{\frac{d-2}{2} + n }(k_0 r)  Y_{M} (\omega) \}_{M \in I_n, n \in \mathbb{N}} $ are also a complete frame by density of $ \text{tr}_{\partial D} \text{Ker}( \Delta + k_0^2)$.  For $r_0$ such that $J_{\frac{d-2}{2} + n }(k_0 r_0) \neq 0$ for all $n \in \mathbb{N}$ (otherwise some basis needs to be dropped), let us denote by $\left( \tilde{U}_{ L , k,  \partial D} \right)$ the map that changes the basis to the orthonormal one and $ \left( \tilde{U}^{-1}_{L , k , \partial D} \right) $ as its inverse.  Then we render the following:
\begin{Theorem}
{\color{black} For $d >2 $} and $r_0$ such that $J_{\frac{d-2}{2} + n }(k_0 r_0) \neq 0$ for all $n \in \mathbb{N}$, we have the following inversion formula for $\partial D \in \mathcal{C}^{{\color{black}4},\alpha}$ and $h \in \mathcal{C}^{ {\color{black} 3},\alpha}$,
\begin{equation}
\begin{split}
 { \color{black}  \lambda^{-1}  \omega_d   (2-d) }   h(x) H(x) =
  \int_{\mathbb{S}^{d-1}}   \lim_{t \rightarrow \infty} \widetilde{\mathcal{G}}(\xi, t, x)
 d \sigma (\xi) ,
\end{split}
\label{invinvinvhahaha}
\end{equation}
where
\[
\begin{split}
\widetilde{\mathcal{G}}(\xi, t, x):=\sum_{ \begin{matrix} L \in I_k, \ M \in I_n , \\   k, n, r, s \in \mathbb{N} \end{matrix} }
& |\xi|^{-1} t^{-1} e^{- i t \varphi_{x,\xi} }  \, \eta_{s , \partial D }  \, \tilde{U}_{  L , s,  \partial D}  \,  \mathcal{W}^{(1)}_{L,M} (\mu_0,\mu_1,\epsilon_0,\epsilon_1, D , h) \\
&\quad \times  \tilde{U}^{-1}_{M , r , \partial D}  \, \left\langle \eta_{r , \partial D } \,,\,  \chi_x\, e^{i t \varphi_{x,\xi} } \right\rangle_{L^2(\partial D, d \sigma)}.
\end{split}
\]
\end{Theorem}

In the following, we define
\[
\text{\rm Proj}_{\tilde{W}_{s,\partial D} } : L^2(\partial D, d \sigma) \rightarrow \tilde{W}_{s,\partial D} := \text{Span}\{  
 ( k_0 r )^{- \frac{d-2}{2} + n } J_{\frac{d-2}{2} + n }(k_0 r)  Y_M(\omega)  |_{\partial D}  \}_{m_{d-1} \leq s},
 \] 
 for $ s \in \mathbb{N}$. Then, via a perturbation analysis as in Lemma \ref{perturb_GPT}, we have

\begin{Lemma} \label{perturb_SC}
{\color{black} For $d >2 $,} given a general $\partial D \in \mathcal{C}^{ {\color{black} 4 },\alpha}$, {\color{black} we suppose $\text{Ric} \geq - C_0 \, g  $ for some $C_0 > 0$.} 
Let $\partial D_{\varepsilon}$ be an $\varepsilon$-perturbation along $h \in \mathcal{C}^{{\color{black} 3 },\alpha}$ and let $S= | \{ T : t_{d-1} \leq s \} | $. Then for $\varepsilon\in\mathbb{R}_+$ sufficiently small, we have
\beqnx
& &\max \bigg\{  \left| \| \tilde{U}^{-1}_{L , p , \partial D_{\varepsilon} } |_{\mathcal{L} ( V_{s,\partial D_{\varepsilon} }, \tilde{W}_{s,\partial D_{\varepsilon} } ) }  \|_{l^2 \rightarrow l^2} -  \| \tilde{U}^{-1}_{L , p , \partial D } |_{\mathcal{L} ( V_{s,\partial D }, \tilde{W}_{s,\partial D } ) }  \|_{l^2 \rightarrow l^2}  \right|  \,, \,\\
& & \quad \quad \quad \left| \| \tilde{U}_{L , p , \partial D_{\varepsilon} } |_{\mathcal{L} ( V_{s,\partial D_{\varepsilon} }, W_{s,\partial D_{\varepsilon} } ) }  \|^{-1}_{l^2 \rightarrow l^2} -  \| \tilde{U}_{L , p , \partial D } |_{\mathcal{L} ( V_{s,\partial D }, W_{s,\partial D } ) }  \|^{-1}_{l^2 \rightarrow l^2}  \right| \bigg \} \\
&<& { \color{black} C_d } \varepsilon \, \max_{1\leq P \leq S} \bigg \{   \max \left \{ 1, \max_{z \neq \lambda_P, z \in \lambda(\Delta_{\partial D}) }{\frac{  { \color{black} \| g^{-1} \|_{C^0}^2 \max\{ 1,  \|g\|_{C^1}^3  \}   }  }{|z^2 - \lambda_P^2|}} \right \} { \color{black} \|  h \|_{C^1} }    \| A \|_{C^1}   (  \lambda_p^2 +  { \color{black} C_0 } )   \\
& & \quad  \times   \| ( k_0 r )^{- \frac{d-2}{2} + s }  J_{\frac{d-2}{2} + s }(k_0 r) |_{\partial D}\|_{L^2(\partial D,d \sigma)} \\
&  &  \quad +  \|  h \|_{C^1}   \| \partial_\nu (  ( k_0 r )^{- \frac{d-2}{2} + s } J_{\frac{d-2}{2} + s }(k_0 r)  Y_T(\omega) ) |_{\partial D} \|_{L^2(\partial D,d \sigma)}  \\
& & \quad +   (d-1) \|  h \|_{C^0}     \| H \|_{C^0} \|   ( k_0 r )^{- \frac{d-2}{2} + s } J_{\frac{d-2}{2} + s }(k_0 r)  Y_T(\omega)  |_{\partial D} \|_{L^2(\partial D,d \sigma)} \bigg\}. 
\eqnx
for some dimensionality constant $C_d$.
\end{Lemma}

Similar to the specific examples in Section 2, we can have the following results. \medskip

\noindent \textbf{Example II.1.} Let us consider $\partial D = R_0 \mathbb{S}^{d-1}$ with $R_0<1$. Using the previous notations, we have
\beqnx
\tilde{U}^{-1}_{L , M , R_0 \mathbb{S}^{d-1}} =  \omega_{d-1}^{-1} ( k_0 R_0 )^{k- \frac{d-2}{2} } [J_{\frac{d-2}{2} + k }(k_0 R_0)]^{-1}   \delta_{L M}  \,.
\eqnx
For $k_0 R_0 /2 \ll 1 $, we have
\beqnx
( k_0 R_0 )^{k- \frac{d-2}{2} } [J_{\frac{d-2}{2} + k }(k_0 R_0)]^{-1} \sim \Gamma\left(\frac{d-2}{2} + k\right) \left( \frac{k_0^2 R_0^2}{2} \right)^{\frac{d-2}{2}-k} \,.
\eqnx
Hence we have the following estimate of the condition number when $\partial D = R_0 \mathbb{S}^{d-1}$,
\beqnx
\kappa( \tilde{U}_{M , L , R_0 \mathbb{S}^{d-1}} |_{\mathcal{L} ( \tilde{W}_{s,R_0 \mathbb{S}^{d-1}} , V_{s,R_0 \mathbb{S}^{d-1}})}  )  \leq C \sqrt{ \frac{s + d-2}{2} } \left ( \frac{ k_0^2 R_0^2 \, (s+ d-2)}{2}  \right)^{\frac{d-2}{2} + s}  .
\eqnx

\medskip

\noindent \textbf{Example II.2.} Consider $\partial D_{\delta} $ as follows: $\partial D = R_0 \mathbb{S}^{d-1}$ {\color{black} with the perturbation $\tilde{h} \in \mathcal{C}^{ {\color{black} 3 },\alpha}$, $\| \tilde{h} \|_{C^1} <1$ and a small magnitude $\delta$.}
Now, applying Lemma \ref{perturb_SC}, together with $\lambda_{T,\mathbb{S}^d} = s(s + d-2)$ for $t_{d-1} = s$, we have for $\delta$ small enough and large enough $s$, 
\beqnx
& &\max \bigg\{  \left| \| \tilde{U}^{-1}_{L , p , \partial D_{\delta} } |_{\mathcal{L} ( V_{s,\partial D_{\delta} }, \tilde{W}_{s,\partial D_{\delta} } ) }  \|_{l^2 \rightarrow l^2} - \omega_{d-1}^{-1} \, R_0^{-\frac{d-1}{2}} \right|  \, , \, \\
&  & \quad \quad \quad \left| \| \tilde{U}_{L , p , \partial D_{\delta} } |_{\mathcal{L} ( V_{s,\partial D_{\delta} }, \tilde{W}_{s,\partial D_{\delta} } ) }  \|_{l^2 \rightarrow l^2}^{-1} -   \omega_{d-1}^{-1} \,  R_0^{- \frac{d-1}{2} } \Gamma\left(\frac{d-2}{2} + k\right) \left( \frac{k_0^2 R_0^2}{2} \right)^{\frac{d-2}{2}-s}  \right| \bigg\}  \\
&<&    C_d \delta (2 R_0 + 2s - d+2)  (2s - d+2) \left( \left(\frac{d-2}{2} + s\right) k_0^2 \right)^{\frac{d-2}{2}-s}  R_0^{-2s - 1 }  \,,
\eqnx
and therefore, for small enough $\delta$, we obtain
\beqnx
& & \kappa( \tilde{U}_{L , p , \partial D_{\delta} } |_{\mathcal{L} ( V_{s,\partial D_{\delta} }, W_{s,\partial D_{\delta} } ) }  ) \\
&\leq&
\frac{    k_0^{d-2-2s} R_0^{\frac{d-2}{2}-2s}   + \delta \, C_d  (2 R_0 + 2s - d+2)   k_0^{d-2-2s}   R_0^{-2s - 1 }  }{ (2s - d+2)^{- \frac{d-2}{2} - s -1}  \, R_0^{-\frac{d-1}{2}} -  \delta \, C_d (2 R_0 + 2s - d+2)   k_0^{d-2-2s}  R_0^{-2s - 1 }   } .
\eqnx

From the above example, we have similar results to those in Corollary~\ref{cor:1}.

{
\begin{Corollary}\label{cor:2}

{ \color{black} For $d > 2$,} let us consider $\partial D_{\delta} $ as a $\delta$-perturbation of $\partial D = R_0 \mathbb{S}^{d-1}$ along the direction $\tilde{h}  \in C^{ {\color{black} 4 },\alpha} (\partial D)$ with {\color{black}$\| \tilde{h}  \|_{\mathcal{C}^1} <1$ for sufficiently small $\delta\in\mathbb{R}_+$}. Then for $h \in \mathcal{C}^{ {\color{black} 3 },\alpha} (\partial D_{\delta})$ with $\| h \|_{\mathcal{C}^1} <1$, considering an $\varepsilon$-perturbation of $\partial D_{\delta} $ along the direction $h$, $( \partial D_{\delta} )_{\varepsilon} $, we have
\beqnx
&  & | [ \text{\rm Proj}_{V_{s,\partial D_{\delta}} } (h H) ] (x) | \notag \\
&\leq&
\lambda^{-1} C_d' 
\frac{    k_0^{d-2-2s} R_0^{\frac{d-2}{2}-2s}   + 2 \delta \, \omega_{d-1}  (2 R_0 + 2s - d+2)   k_0^{d-2-2s}   R_0^{-2s - 1 }  }{ (2s - d+2)^{- \frac{d-2}{2} - s -1}  \, R_0^{-\frac{d-1}{2}} -  2 \delta \, \omega_{d-1}  (2 R_0 + 2s - d+2)   k_0^{d-2-2s}  R_0^{-2s - 1 }   }  \\
& & \times \, \| \mathcal{W}_{L, M}^{(1)} ( \mu_0,\mu_1,\epsilon_0,\epsilon_1 , D_{\delta} ,h)  \|_{\mathcal{L} (V_{s, \mathbb{S}^{d-1} }, V_{s, \mathbb{S}^{d-1} })}  \,.
\eqnx
Similarly,
\beqnx
&  & | [ \text{\rm Proj}_{V_{s,\partial D_{\delta}} } (h H) ] (x) | \notag \\
&\leq&
\lambda^{-1} C_{k_0,d}'    (  k_0 |x|)^{ \frac{d-2}{2} - k } \frac{1}{| H^{(1)}_{\frac{d-2}{2} + k }(k_0 R) | } \\
 &\times&
\frac{    k_0^{d-2-2s} R_0^{\frac{d-2}{2}-2s}   +  \delta \, C_d  (2 R_0 + 2s - d+2)   k_0^{d-2-2s}   R_0^{-2s - 1 }  }{ (2s - d+2)^{- \frac{d-2}{2} - s -1}  \, R_0^{-\frac{d-1}{2}} -  \delta \,  C_d  (2 R_0 + 2s - d+2)   k_0^{d-2-2s}  R_0^{-2s - 1 }   }  \\
& \times & \, 
\left \|
\frac{\partial }{\partial \varepsilon} \left( \left \langle Y_L(\omega_x)  \, , \, \bigg(u_{ ( \partial D_{\delta} ) _\varepsilon} - ( k_0 r )^{- \frac{d-2}{2} + n} J_{\frac{d-2}{2} + n }(k_0 r)  Y_{M} (\omega)  \bigg) (R \, \omega_x ) \right \rangle_{ L^2( R \mathbb{S}^{d-1} , d \omega_x ) }   \right)
\right \|_{\mathcal{L} (V_{s, \mathbb{S}^{d-1} }, V_{s, \mathbb{S}^{d-1} })}  
\eqnx
for some dimensionality constants $C_d$  and $ C_{k_0,d}'$ which depends also on $k_0$.
\end{Corollary}
}

The proof of Corollary \ref{cor:2} is similar to that of Corollary~\ref{cor:1}, and therefore we skip it. By Corollary~\ref{cor:2}, we readily see that the reconstruction of $h(x)$ from $ \mathcal{W}_{L, M}^{(1)} ( \mu_0,\mu_1,\epsilon_0,\epsilon_1 , D,h)  $ is more stable at points with high mean curvature $|H(x)|^2$ when $D$ is not too far from $R_0 \mathbb{S}^{d-1}$.

\section*{Acknowledgment}
The work of H. Ammari was supported by the SNF grant 200021--172483. The work of H. Liu was supported by the Hong Kong UGC General Research Funds, 12301218, 12302919 and 12301420. The authors would like to thank the anonymous referees for many constructive comments and suggestions, which have led to significant improvements on the results as well as the presentation of the paper.

{\color{black}

\section*{Appendix}
In this appendix, we show the descent nature of the proposed preconditioned Gauss-Newton method \eqref{iteration} under a particular choice of $\alpha_n$ in the functional \eqref{optimization}. It is mainly for an illustrative purpose and hence we shall make some simple assumptions in order to simplify the argument.   {\color{black} For simplicity, let us assume $\partial D_{\text{exact}}$ is in $\mathcal{C}^{\infty}$.}

First, it is easy to see from the definition in \eqref{iteration} that
\begin{equation}\tag{A.1}
v^n(x) =  - [A_{n,k} ]^{-1} \text{Proj}|_{A_{n,k}   [ V_{k, \partial D^n } ] } \left( \mathcal{M} (\lambda, D^n )  -  \mathcal{M} ^{\text{meas}} (\lambda, D_{\text{exact}} )   \right)  \in V_{k, \partial  D^n } ,
\label{temp1}
\end{equation}
{\color{black} as the least-squared solution in \eqref{iteration}, 
where we recall $A_{n,k} :=  A_n |_{   \mathcal{L} ( V_{k, \partial D^n } , \mathbb{C}^{d_k} )   } $
}and
\begin{equation}\tag{A.2}
 \mathcal{M}^{(i)} \left(\lambda, D^n , \frac{v^n(x)}{H^n(x)}  \right) = - \text{Proj}|_{A_{n,k}   [ V_{k, \partial D^n } ] } \left( \mathcal{M} (\lambda, D^n )  -  \mathcal{M} ^{\text{meas}} (\lambda, D_{\text{exact}} )   \right) . 
\label{temp2}
\end{equation}
For notational sake, let us denote
\beqnx
( \Upsilon^n )^2  :=  \| \text{Proj}|_{A_{n,k}   [ V_{k, \partial \partial  D^n } ] } \left( \mathcal{M} (\lambda, D^n )  - \mathcal{M} ^{\text{meas}} (\lambda, D_{\text{exact}} )   \right) \|^2_{\mathbb{C}^{d_k}}  \,.
\eqnx

Next, for a given $n \geq 0$, we suppose that $D^n$ has a $C^{\infty}$ boundary (for the sake of simplicity and illustrative purpose) and $D_{\text{exact}}$ has a $\mathcal{C}^{ {\color{black} 4 }, \alpha }$ boundary.  Let us also assume that $D_{\text{exact}}$ and $D^n$ are both $\mathcal{C}^{{\color{black} 3 },\alpha}$ $\varepsilon$-perturbations of $R_0 \mathbb{S}^{d-1}$ with $\varepsilon < \varepsilon_0$ such that $ ( 1 - \varepsilon_0 \left \|   H^{-1} \right \|^2_{C^{{\color{black} 3 },\alpha}}  \left \|   A \right \|^2_{C^{{\color{black} 3 },\alpha}}  ) \left \|   H^{-1}  \right \|^2_{C^{{\color{black} 3 },\alpha}} $ is bounded below, where $H^{-1}$ is the inverse of the mean curvature (again for the sake of simplicity).  Then via the bound \eqref{bound_K_n} in Theorem \ref{theoremK}, a detailed expansion of the operators in $\mathcal{M}^{(i)}(\lambda, D^n , \cdot)$ in \eqref{seriesvariation2} as in \eqref{variationvariation} (c.f. \cite{resol1,heteroscattering,yu,gpt,book}) as well as the equivalence of norms in finite dimensions, one can show that the terms in Theorem \ref{lemmaDM} satisfy
$$ \| \mathcal{M}^{(i)}(\lambda, D^n , h)  \|_{\mathbb{C}^{d_k}}^2 \leq C_k \| \mathcal{M}^{(i)}(\lambda, D^n , h)  \|^2_{\mathcal{L} (V_{k, \mathbb{S}^{d-1} }, V_{k, \mathbb{S}^{d-1} })}  \leq  L_{\varepsilon_0,i,k}  \| h\|_{C^1}^i $$
with $L_{\varepsilon_0,i,k} $ being independent of $D^n$.
Write for $0 \leq t \leq 1$
\[
D^{n}(t) := \left\{ x + \alpha_n t \frac{v^n(x)}{H^n(x)}  \nu^n(x) \, : \, x \in D^n  \right\} \, . 
\]
{\color{black} Recalling the definition of $J(D)$ in \eqref{optimization},} we have
\beqnx
 J(D^{n+1}) 
= J(D^{n}) + J^{(1)}\left(D^{n} ,  \frac{v^n(x)}{H^n(x)}  \right)   + \int_0^1 J^{(2)}\left(D^{n}(t) ,  \frac{v^n(x)}{H^n(x)}  \right)  (1-t) dt,
\eqnx
where one can directly verify from \eqref{temp2} that
\beqnx
 J^{(1)}\left(D^{n} ,  \frac{v^n(x)}{H^n(x)}  \right)  := \left \langle \mathcal{M}^{(i)} \left(\lambda, D^n , \frac{v^n(x)}{H^n(x)}  \right) ,  \mathcal{M} (\lambda, D^n )  - ( \mathcal{M})^{\text{meas}} (\lambda, D_{\text{exact}} ) \right \rangle_{\mathbb{C}^{d_k}}  =  - ( \Upsilon^n )^2,
\eqnx
which shows that our choice of $ \frac{v^n}{H^n} $ is always a descent direction, and
\beqnx
J^{(2)}\left(D^{n}(t) ,  \frac{v^n(x)}{H^n(x)}  \right) &:=&  \left \langle \mathcal{M} \left (\lambda, D^n(t) \right) -   \mathcal{M}^{\text{meas}} (\lambda, D_{\text{exact}} )  , \mathcal{M}^{(2)} \left (\lambda, D^n(t) , \frac{{\color{black} v^n}(x)}{H^n(x)}  \right )  \right \rangle_{\mathbb{C}^{d_k}} \\
& &+ \left \| \mathcal{M}^{(1)} \left(\lambda, D^n(t) ,  \frac{v^n(x)}{H^n(x)}  \right)  \right \|^2_{\mathbb{C}^{d_k}} \,.
\eqnx

Suppose $\overline{\alpha_n} > 0$ is sufficiently small such that for all chosen $\alpha_n < \overline{\alpha_n}$, $D^n(t)$'s still have $\mathcal{C}^{\infty}$ boundaries (again for the sake of simplicity and illustrative purpose). The same assumption is made for all $\varepsilon(t)$-perturbations of $D^{\text{exact}}$ for all $t \in [0,1]$ and $\varepsilon(t) < \varepsilon_0$. Then for all $\alpha_n < \overline{\alpha_n}$, via \eqref{error} and \eqref{temp1}, we have
\beqnx
\left \| \mathcal{M} \left (\lambda, D^n(t) \right) -  \mathcal{M}^{\text{meas}} (\lambda, D_{\text{exact}} ) \right \|_{\mathbb{C}^{d_k}} & \leq &  L_{\varepsilon_0,1,k}  \left( \alpha_n t \left \|   \frac{v^n(x)}{H^n(x)}\right \|_{C^1} + \varepsilon_0 \right) + \text{err} \\
\left \| \mathcal{M}^{(2)} \left (\lambda, D^n(t) , \frac{{\color{black} v^n}(x)}{H^n(x)}  \right ) \right \|_{\mathbb{C}^{d_k}} & \leq&   L_{\varepsilon_0,2,k} \left \|   \frac{v^n(x)}{H^n(x)}  \right \|^2_{C^1}  \\
 \left \| \mathcal{M}^{(1)} \left (\lambda, D^n(t) , \frac{{\color{black} v^n}(x)}{H^n(x)}  \right )  - A_n \left [  v^n(x)  \right ]   \right\|_{\mathbb{C}^{d_k}} &\leq& L_{\varepsilon_0,2,k}  \alpha_n \, t \left \|   \frac{v^n(x)}{H^n(x)}  \right \|^2_{C^1} \\
 \left \| A_n \left [  v^n(x)  \right ]   \right\|^2_{\mathbb{C}^{d_k}}  &=& ( \Upsilon^n )^2 \,.
\eqnx
{\color{black} where $\text{err}$ is a choice of error threshold given as in \eqref{error}.}
Moreover, by Morrey's inequality, {\color{black} (which we recall as
\[
\| u \|_{\mathcal{C}^{s - [ \frac{d-1}{2}] -1, \gamma} (\partial D) } \leq C \|u \|_{H^{s}(\partial D)}
\]
for all $s \geq 1+ [ \frac{d-1}{2}] $, and $\gamma = 1+ [ \frac{d-1}{2}]  - \frac{d-1}{2}  $ if $ \frac{d-1}{2} \in \mathbb{N}$ and $\gamma \in (0,1)$ otherwise,)}
together with \eqref{temp1} and \eqref{eigenvalue}, and the fact that
\[
( 1 - \varepsilon_0 \left \|   H^{-1} \right \|^2_{C^{{\color{black} 3 },\alpha}}  \left \|   A \right \|^2_{C^{{\color{black} 3 },\alpha}}  ) \left \|   H^{-1}  \right \|^2_{C^{{\color{black} 3 },\alpha}}  \leq \left \|   (H^n)^{-1}  \right \|^2_{C^{{\color{black} 3 },\alpha}}  \leq \left \|   H^{-1}  \right \|^2_{C^{{\color{black} 3 },\alpha}} ( 1 + \varepsilon_0 \left \|   H^{-1} \right \|^2_{C^{{\color{black} 3 },\alpha}}  \left \|   A \right \|^2_{C^{{\color{black} 3 },\alpha}}  ), 
\]
one can readily show that $   (H^n)^{-1} $ is in $ \mathcal{C}^4$ and
\begin{equation}\tag{A.3}\label{A.3}
\begin{split}
  \left \|   \frac{v^n}{H^n}  \right \|^2_{C^1} 
\leq & C_{d} \left \|   (H^n)^{-1}  \right \|^2_{C^1} \left  \| v^n \right \|^2_{H^{2 + [ (d - 1)/2 ]}}  \\
 \leq & C_{d} \left \|   (H^n)^{-1}  \right \|^2_{C^1} \lambda_k^{8 + 2(d - 1)} \left  \| v^n \right \|^2_{V_{k, \partial \mathbb{S}^{d-1} }}  \\
  \leq & \frac{ C_{\varepsilon_0,d} }{ d_k^2}  \left \|   H^{-1}  \right \|^2_{C^1} ( 1 + \varepsilon_0 \left \|   H^{-1} \right \|^2_{C^1}  \left \|   A \right \|^2_{C^1}  )
\\
  \times \bigg(  1+ & \varepsilon_0 \, C_d   \| A_{\partial D_{\text{exact}}} \|_{C^0}   \| g^{-1} _{\partial D_{\text{exact}}}\|_{C^0}   \bigg)^{4 + (d - 1)}   \lambda_{k,\partial D_\text{exact}}^{8 + 2(d - 1)} 
 \left( R_0 + \frac{1}{R_0} \right)^{2k} ( \Upsilon^n )^2. 
\end{split}
\end{equation}
where the first inequality utilized Morrey's inequality with $s = 2+ [ \frac{d-1}{2}] $.
Here, the inequality in the last term in \eqref{A.3} is independent of the choice of $n$, {\color{black} which shows that $ \frac{v^n}{H^n} $ is in $\mathcal{C}^1$ for all $n$, and that $D^{n+1}$ as defined in \eqref{iteration} is indeed a $\mathcal{C}^1$ perturbation of $D^n$.}  Similarly, one can furthermore verify that $\left \|   \frac{v^n}{H^n}  \right \|^2_{C^{3,\alpha}} $ is bounded.
For notational sake, let us write $s_n$ be such that
\[
\alpha_n := s_n \min\{ 1 ,       ( \Upsilon^n )^{-1} \}  \leq \overline{\alpha_n}\,.
\]
The above set of inequalities now combine to give
\beqnx
J(D^{n+1})
& \leq & J(D^{n}) - \alpha_n   \, ( \Upsilon^n )^2 + \frac{1}{2} \alpha_n^2 \,
 ( \Upsilon^n )^2  \widetilde{L_{k,d,\varepsilon_0,2}}   + \alpha_n^3  ( \Upsilon^n )^3 \,
 \widetilde{L_{k,d,\varepsilon_0,3}}  +  \alpha_n^4 ( \Upsilon^n )^4  \widetilde{L_{k,d,\varepsilon_0,4}} \\
& \leq & J(D^{n}) - s_n  \left( 1  -  s_n \,
  \widetilde{L_{k,d,\varepsilon_0,2}}   - s_n^2   \,
 \widetilde{L_{k,d,\varepsilon_0,3}}  -  s_n^3  \widetilde{L_{k,d,\varepsilon_0,4}} \right) \, \Upsilon^n  \min\{1 , \Upsilon^n  \} \,,
\eqnx
for some $ 0 < \widetilde{L_{k,d,\varepsilon_0,i}} , i =2,3,4$ independent of $n$.  Notice that there exists a constant $\overline{s_{k,d,\varepsilon_0}}$ independent of $n$ such that for all $0 < s \leq \overline{s_{k,d,\varepsilon_0}}$
\[
   s \left ( \,
  \widetilde{L_{k,d,\varepsilon_0,2}}   + s   \,
 \widetilde{L_{k,d,\varepsilon_0,3}}  +  s ^2 \widetilde{L_{k,d,\varepsilon_0,4}}  \right ) \leq  \frac{1}{2}. 
\]
Hence for all choices of $ s_n  \leq \min \{ \overline{\alpha_n}  \max\{ 1 ,   \Upsilon^n \} \, , \, \overline{s_{k,d,\varepsilon_0}} \} $, we have
\begin{equation}\tag{A.4}
J(D^{n+1}) \leq J(D^{n}) - \frac{s_n}{2} \, \Upsilon^n  \min\{1 , \Upsilon^n  \} \,. \label{descent_step}
\end{equation}

Thus, by supposing $D^0$ has a $C^{\infty}$ boundary and that 
$D_{\text{exact}}$ and $D^n$ are both $\varepsilon$-perturbations of $R_0 \mathbb{S}^{d-1}$ with $\varepsilon < \varepsilon_0$ and the perturbation direction small in $\mathcal{C}^1$ and finite in $\mathcal{C}^{{\color{black} 3 },\alpha}$, we can inductively show that for all $n > 0$ and all small choices of $ s_n  \leq \min \left\{ \overline{\alpha_n}  \max\{ 1 ,   \Upsilon^n \} \, , \, \overline{s_{k,d,\varepsilon_0}} \right \} $, \eqref{descent_step} always holds, namely $ 0 \leq J(D^{n+1}) < J(D^{n}) $.
This verifies the descent nature of our Gauss-Newton method with appropriate step-sizes. Moreover, from this, we obtain $J(D^{n}) \rightarrow J^*$ for some $J^*$ and 
\[
\sum_{n=0}^\infty s_n  \, \Upsilon^n  \min\{1 , \Upsilon^n  \} \leq 2 ( J(D^{0}) - J^* ) < \infty ,
\]
which guarantees that
\[
s_n  \, \Upsilon^n  \min\{1 , \Upsilon^n  \}  \rightarrow 0 \text{ as } n \rightarrow \infty\,.
\]

{\color{black}
More observations may follow: (i) with this choice of $s_n$ such that all $D^n$'s are $C^{{\color{black} 3 },\alpha}$ $\varepsilon^n$-perturbation of $D_{\text{exact}}$ ($\varepsilon^n < \varepsilon_0$), the convergence of a subsequence of $D^n$ to $D^*$ in $C^{{\color{black} 3 },\beta}$ for $\beta < \alpha$ is evident via the Ascoli-Arzela theorem;  (ii) in case that $\{D^n\}$ never comes close to the boundary of the $\varepsilon_0$-neighborhood of $C^{{\color{black} 3 },\alpha}$-perturbation of $D_{\text{exact}}$, then $\overline{\alpha_n}$ can be chosen to have a uniform lower bound and thus $\Upsilon^n$ converges to $0$ itself. However, we shall not go further on the convergence of this numerical scheme as it is not the focus of our work. We only remark that, in practice, it is usually preferred to perform a line search to obtain $\alpha_n$ for a greater descent magnitude. 
}

}


\begin{thebibliography}{99}







\bibitem{Pfeuffer}
H. Abels, C. Pfeuffer, {\em Characterization of non-smooth pseudodifferential operators}, J. Fourier Anal. Appl.,  \textbf{24} (2018), 371--415.

\bibitem{handbook}
M. Abramowitz and I.A. Stegun, {\em Handbook of Mathematical Functions}, 9th edition, Dover Publications, 1970.



\bibitem{regularity1}
M.S. Agranovich, {\em Sobolev spaces, their generalizations and elliptic problems in smooth and Lipschitz domains}, 2015, Springer, Switzerland.


\bibitem{ergodicity}
{H. Ammari, Y.T. Chow and H. Liu,} 
{\em Quantum ergodicity and localization of plasmon resonances},
Preprint, arXiv: 2003.03696.

\bibitem{spectral} {H. Ammari, Y.T. Chow, K. Liu, and J. Zou},
{\it Optimal shape design by partial spectral data}, 
 SIAM J. Sci. Comput., {\bf 37} (2015), B855-B883.

\bibitem{resol1} { H. Ammari, Y.T. Chow, and J. Zou}, {\em Super-resolution in imaging high contrast targets from the perspective of scattering coefficients},  J. Math. Pures Appl. {\bf 111} (2018), 191-226.


\bibitem{heteroscattering}
{H. Ammari, Y.T. Chow and J. Zou}, {\em The concept of heterogeneous scattering coefficients and its application in inverse medium scattering}, SIAM J. Math. Anal., \textbf{46} (2014), 2905-2935.



\bibitem{phaseless}
H. Ammari, Y.T. Chow and J. Zou, Phased and phaseless domain reconstructions in the inverse scattering problem via scattering coefficients. SIAM J. Appl. Math., \textbf{76} (2016), 1000--1030.

\bibitem{ACKLM} H. Ammari, G. Ciraolo, H. Kang, H. Lee, and G. Milton, {\it Spectral analysis of a Neumann-Poincar\'e-type operator and analysis of cloaking due to anomalous localized resonance}, Arch. Ration. Mech. Anal., {\bf 208} (2013), 667--692.


\bibitem{gpt}
{H. Ammari, Y. Deng, H. Kang and H. Lee }
{\em Reconstruction of inhomogeneous conductivities via the
concept of generalized polarization tensors,} Ann. Instit. Henri Poicar\'e, \textbf{31} (2014), 877-897.

\bibitem{mono_2}
H. Ammari, B. Fitzpatrick, H. Kang, M. Ruiz, S. Yu, and H. Zhang, {\it Mathematical and Computational Methods in Photonics and Phononics}, Mathematical Surveys and Monographs, Vol. 235, Am. Math. Soc., Providence, 2018.

\bibitem{yu}
{H. Ammari, J. Garnier, H. Kang, M. Lim, and S. Yu}, {\em
  Generalized polarization tensors for shape description}, Numer. Math., {\bf 126} (2014), 199--224. 

\bibitem{noise}
H. Ammari, J. Garnier, H. Kang, M. Lin, S. Yu, Generalized polarization tensors for shape descrip-
tion, Numer. Math. 126, (2014), pp. 119-224.

\bibitem{book}
{H. Ammari and H. Kang}, {\em Polarization and Moment
    Tensors: With Applications to Inverse Problems and Effective
    Medium Theory}, Appl. Math. Sci. 162, Springer-Verlag, New York, 2007.

\bibitem{ammari2004reconstruction}
H. Ammari and H. Kang,
{\it Reconstruction of Small Inhomogeneities from Boundary Measurements}, {Springer}, 2004.

\bibitem{mono_1}
H. Ammari, H. Kang and H. Lee, {\it Layer Potential Techniques in Spectral Analysis}, Mathematical Surveys and Monographs, Vol. 153, Am. Math. Soc., Providence, 2009.



\bibitem{homoscattering}
{H. Ammari, H. Kang, H. Lee, and M. Lim}, {\em Enhancement of near-cloaking. Part II: the Helmholtz equation}, Comm. Math. Phys., \textbf{317} (2013), pp. 485-502.


\bibitem{expansion}
H. Ammari, H. Kang, H., M. Lim, H. Zribi,  {\it Conductivity interface problems. Part I: small perturbations of an interface}, Trans. Amer. Math. Soc., \textbf{362} (2010), 2435-2449.

\bibitem{zribi}
{H.~Ammari, H.~Kang, M.~Lim, and H.~Zribi}, {\em The
generalized
  polarization tensors for resolved imaging. {P}art {I}: {S}hape reconstruction
  of a conductivity inclusion}, Math. Comp., \textbf{81} (2012), 367--386.

\bibitem{new_exp}
H. Ammari, P. Millien, M. Ruiz, H. Zhang, {\em Mathematical analysis of plasmonic nanoparticles: the scalar case}, Arch. Ration. Mech. Anal., \textbf{224} (2017), pp. 597-658.

\bibitem{AKL} K. Ando, H. Kang and H. Liu, {\em Plasmon resonance with finite frequencies: a validation of  the quasi-static approximation for diametrically small inclusions}, SIAM J. Appl. Math., {\bf 76} (2016), no. 2, 731--749. 

\bibitem{Bateman}
H. Bateman, {\em Higher Transcendental Functions}, Vol.I-III. McGraw-Hill Book Company, 1953. 



\bibitem{BL17}
E. Bl{\aa}sten and H. Liu, {\it Recovering piecewise constant refractive indices by a single far-field pattern}, Inverse Problems, \textbf{36} (2020), 085005. 

\bibitem{curv_Liu_2}  {E. Bl{\aa}sten and H. Liu}, {\em On corners scattering stably and stable shape determination by a single far-field pattern}, Indiana Univ. Math. J., \textbf{70} (2021), no. 3, 907--947.  

\bibitem{curv_Liu} {E. Bl{\aa}sten and H. Liu}, {\em Scattering by curvatures, radiationless sources, transmission eigenfunctions and inverse scattering problems,}  arXiv: 1808.01425, 2018. 

\bibitem{curv_Liu_3} { E. Bl{\aa}sten, H. Li, H. Liu and Y. Wang}, {\em Localization and geometrization in plasmon resonances and geometric structures of Neumann-Poincar\'e eigenfunctions}, ESAIM Math. Mode.. Numer. Anal., \textbf{54} 3 (2020), 957--976. 

\bibitem{BPS}
  {E. Bl{\aa}sten, L. P\"aiv\"arinta and J. Sylvester},  {\it Corners always scatter},
  Comm.\ Math.\ Phys., {\bf 331} (2014), 725--753.

\bibitem{BZ} E. Bonnetier and H. Zhang, {\it Characterization of the essential spectrum of the Neumann-Poincar\'e operator in 2D domains with corner via Weyl sequences}, Rev. Mat. Iber., in press, 2018. 

\bibitem{DoCarmo}
M.P. do Carmo, {\em Riemannian Geometry}, Birkh\"auser, Boston, 1992. 

\bibitem{CDL} X. Cao, H. Diao and H. Liu, {\em Determining a piecewise conductive medium body by a single far-field measurement}, CSIAM Trans. Appl. Math., \textbf{1} (2020), 740-765.

\bibitem{CHL} D. Choi, J. Helsing and M. Lim, {\it Corner effects on the perturbation of an electric potential }, SIAM J. Appl. Math., {\bf 78} (2018), 1577-1601. 

\bibitem{CMM}
R.R. Coifman, A. McIntosh, and Y. Meyer, {\em L'integrale de Cauchy definit un operateur
borne sur $L^2$ pour les courbes Lipschitziennes}, Ann. Math., \textbf{116} (1982), 361--387.


\bibitem{CK} D. Colton and R. Kress, {\it Looking back on inverse scattering theory}, SIAM Rev., {\bf 60} (2018), 779--807. 

\bibitem{CK_book} D. Colton and R. Kress, Inverse acoustic and electromagnetic scattering theory. Vol. 93. Berlin: Springer, 1998.


\bibitem{CR} M. D. Cristo and L. Rondi, {\it Exponential instability for inverse elliptic problems with unknown boundaries}, J. Phys.: Conf. Ser., {\bf 73} (2007), 012005.

\bibitem{DLL1} {Y. Deng, H. Li and H. Liu}, {\em On spectral properties of Neuman-Poincar\'e operator and plasmonic resonances in 3D elastostatics}, J. Spectr. Theory, \textbf{9} (2019), no. 3, 767--789. 

\bibitem{DLL2} Y. Deng, H. Li and H. Liu, {\em Analysis of surface polariton resonance for nanoparticles in elastic system}, SIAM J. Math. Anal., \textbf{52} (2020), no. 2, 1786--1805.

\bibitem{DLL3} Y. Deng, H. Li and H. Liu, {\em Spectral properties of Neumann-Poincar\'e operator and anomalous localized resonance in elasticity beyond quasi-static limit}, J. Elasticity, \textbf{140} (2020), 213--242.

\bibitem{DCL} H. Diao, X. Cao and H. Liu, {\em On the geometric structures of transmission eigenfunctions with a conductive boundary condition and applications}, Comm. PDE, DOI: 10.1080/03605302.2020.1857397, 2020. 

\bibitem{regularity2}
Z. Ding, {\em A proof of the trace theorem of Sobolev spaces on Lipschitz domains}, Proc. Amer. Math. Soc., \textbf{124} (1996), 591--600.


\bibitem{folland}
G. Folland, {\em Introduction to Partial Differential Equations}, Vol. 102, Princeton University Press, 1995.

\bibitem{numerical_Gauss_Newton}
C. Fraley, {\em Computational behavior of Gauss-Newton methods}, SIAM J. Sci. Stat. Comput., \textbf{10} (1989), 515--532.


\bibitem{regularity3}
G. Geymonat, {\em Trace theorems for Sobolev spaces on Lipschitz domains. Necessary conditions}, Ann. Math. Blaise Pascal, \textbf{14} (2007), 187--197.



\bibitem{homogeneous}
L. Grafakos, R.H. Torres, {\it Pseudodifferential operators with homogeneous symbols}, Michigan Math. J, \textbf{46}(1999), 261--269.

\bibitem{harmonics}
A. Higuchi, {\em Symmetric tensor spherical harmonics on the $N$-sphere and their application to the de Sitter group $SO(N,1)$}, J. Math. Phys., \textbf{28} (1987), 1553--1566.

\bibitem{Hor1}
L. H\"ormander, {\em The Analysis of Linear Partial Differential Operators. I}, Vol. 256, Springer-Verlag, Berlin,1983.

\bibitem{Hor2}
L. H\"ormander, {\em The Analysis of Linear Partial Differential Operators. II}, Vol. 257, Springer-Verlag, Berlin,1983.

\bibitem{Isakov}
V. Isakov, {\em Stability estimates for obstacles in inverse scattering}, J. Comput. Appl. Math., \textbf{42} (1992), 79--88.

\bibitem{ILW}
V. Isakov, R.-Y. Lai and J.-N. Wang, {\em Increasing stability for the conductivity and attenuation coefficients}, SIAM J. Math. Anal., \textbf{48} (2016), 569--594. 

\bibitem{100years}
V. Ivrii, {\em 100 years of Weyl's law}, Bulletin Math. Sci., \textbf{6} (2016), 379--452.


\bibitem{Jost}
J. Jost, {\em Riemannian Geometry and Geometric Analysis}, Springer, Berlin, 2008.

\bibitem{KLY} H. Kang, M. Lim and S. Yu, {\it Spectral resolution of the Neumann-Poincar\'{e} operator on intersecting disks and analysis of plasmon resonance}, Arch. Ration. Mech. Anal., {\bf 226} (2017), 83--115.

\bibitem{seo} 
{H. Kang and J.K. Seo,}
{\em Inverse conductivity problem with one measurement: uniqueness of balls in $\mathbb{R}^3$}, SIAM J. Appl. Math., \textbf{59} (1999), 851--867.

\bibitem{Source}
S. Langer, Preconditioned Newton methods for ill-posed problems, Doctoral Dissertation, Univ. G\"ottingen, 2007. 


\bibitem{kellog} 
O. D. Kellogg, {\em Foundations of Potential Theory}, Springer-Verlag, Berlin-New York, 1967.
 
\bibitem{shapiro} {D. Khavinson, M. Putinar and H.S. Shapiro}, {\em Poincar\'e's variational problem
in potential theory}, Arch. Ration. Mech. Anal., \textbf{185} (2007) 143--184.

\bibitem{Klingenberg}
W. Klingenberg, {\em A Course in Differential Geometry}, Vol. 51, Springer Science \& Business Media, 2013. 

\bibitem{Kobayash1}
S. Kobayashi and N. Katsumi, {\em Foundations of Differential Geometry}, Vol. I, New York, London, 1963.

\bibitem{Kobayash2}
S. Kobayashi and N. Katsumi, {\em Foundations of Differential Geometry}, Vol. II, New York, London, 1969.


\bibitem{Kor02} B. G. Korenev, {\it Bessel  functions and their applications}, Chapman \& Hall/CRC, 2002.

\bibitem{Hadamard}
S. G. Krantz,H.R.  Parks,  The implicit function theorem: history, theory, and applications. Springer Science \& Business Media,  2012.

\bibitem{Lee}
J.M. Lee, {\em Riemannian Manifolds}, Vol. 176. Springer Science \& Business Media, 2006.


\bibitem{Newton1}
K. Levenberg, {\em A Method for the solution of certain non-Linear problems in least squares}, Quart. Appl. Math, \textbf{2} (1944), 164--168.

\bibitem{LL} {H. Li and H. Liu}, {\em On anomalous localized resonance and plasmonic cloaking beyond the quasistatic limit}, Proc. A,  {\bf 474}: 20180165

\bibitem{Annal}
Y. Liokumovich, F.C.  Marques and A. Neves, {\em Weyl's law for the volume spectrum}, Ann. Math., \textbf{187} (2018), 933--961.

\bibitem{LPRX} H.~Liu, M.~Petrini, L.~Rondi and J.~Xiao, {\it Stable determination of sound-hard polyhedral scatterers by a minimal number of scattering measurements}, {J. Differential Equations}, {\bf 262} (2017), no. 3, 1631--1670.

\bibitem{LRX} {H. Liu, L. Rondi and J. Xiao}, {\it Mosco convergence for $H(\mathrm{curl})$ spaces, higher integrability for Maxwell's equations, and stability in direct and inverse EM scattering problems}, J. Eur. Math. Soc.,  \textbf{21} (2019), no. 10, 2945--2993.  

\bibitem{LT} H. Liu and C.-H. Tsou, {\em Stable determination of polygonal inclusions in Calder\'on's problem by a single partial boundary measurement}, {Inverse Problems}, \textbf{36} (2020), 085010. 

\bibitem{LX} H. Liu and J. Xiao, {\em On electromagnetic scattering from a penetrable corner}, SIAM J. Math. Anal., {\bf 49}  (2017), no. 6, 5207--5241.

\bibitem{LZ} H. Liu and J. Zou, {\em Uniqueness in an inverse acoustic obstacle scattering problem
for both sound-hard and sound-soft polyhedral scatterers}, Inverse Problems, {\bf 22} (2006), 515--524. 

\bibitem{Maj} {A. Majda}, {\it High frequency asymptotics for the scattering matrix and the inverse problem of acoustical scattering}, Comm. Pure Appl. Math., {\bf 29} (1976), 261--291. 

\bibitem{Newton2}
D. Marquardt, {\em An Algorithm for least-squares estimation of nonlinear parameters}, SIAM J. Appl. Math., \textbf{11} (1963), 431--441.

\bibitem{plasmon1} {I.D. Mayergoyz, D.R. Fredkin, and Z. Zhang}, {\em Electrostatic (plasmon) resonances in nanoparticles}, Phys. Rev. B, \textbf{72} (2005), 155412.

\bibitem{McL} W. McLean, {\em Strongly Elliptic Systems and Boundary Integral Equations}, Cambridge Univ. Press, 2000. 

\bibitem{mcowen}
R. McOwen, {\em Partial Differential Equations: Methods and Applications}, Prentice Hall, 1996.


\bibitem{Melrose}
R. B. Melrose and G. Uhlmann, {\em An introduction to microlocal analysis}, manuscript, Massachusetts Institute of Technology, 2008. 


\bibitem{weyl2}
{Y. Miyanishi,}
{\em Weyl's law for the eigenvalues of the Neumann-Poincar\'e
operators in three dimensions: Willmore energy and surface
geometry,} Preprint, arXiv:1806.03657

\bibitem{weyl1}
{Y. Miyanishi and G. Rozenblum,}
{\em Eigenvalues of the Neumann-Poincar\'e operators in dimension 3: Weyl's Law and geometry, }
Preprint, arXiv:1812.00582

\bibitem{Myers}
S. B. Myers, {\em Riemannian manifolds with positive mean curvature}, Duke Math. J., \textbf{8} (1941), 401--404.

\bibitem{NUW} S. Nagayasu, G. Uhlmann and J.-N. Wang, {\em Increasing stability in an inverse problem for the acoustic equation}, Inverse Problems, \textbf{29} (2013),
025012.

\bibitem{osborn} J. E. Osborn, {\em Spectral approximation for compact operators}, Math. Comp., \textbf{29} (1975), 712--725.

\bibitem{Petersen}
P. Petersen, S. Axler and K. A. Ribet, {\em Riemannian Geometry}, Vol. 171, Springer, New York, 2006.

\bibitem{regularity4}
G. Savare, {\em Regularity results for elliptic equations in Lipschitz domains}, J. Funct. Anal., \textbf{152} (1998), 176--201.

\bibitem{stein}
E.M. Stein and G. Weiss, {\em Introduction to Fourier Analysis on Euclidean Spaces}, Princeton University Press, 1971. 

\bibitem{Taylor}
M. Taylor, {\em Partial Differential Equations II}, Springer-Verlag, 1996.

\bibitem{Taylorbook}
M. Taylor, {\em Pseudodifferential Operators and Nonlinear PDE}, Vol. 100. Springer Science \& Business Media, 2012.

\bibitem{uhlmann}
{G. Uhlmann}, {\em The Dirichlet to Neumann map and inverse problems}, preprint.


\bibitem{regularity5}
G. Verchota, {\em Layer potentials and regularity for the Dirichlet problem for Laplace's equation in Lipschitz domains}, J. Funct. Anal., \textbf{59} (1984), 572--611.

\bibitem{watson}
G. N. Watson and G. Neville, {\em A Treatise on the Theory of Bessel Functions}, Cambridge University Press, 1995.


\end{thebibliography}
\end{document}